\documentclass[11pt,reqno]{amsart}

\usepackage{pgfplots}
\usepgfplotslibrary{colormaps,fillbetween}
\pgfplotsset{compat=1.15}
\usepackage{mathrsfs}
\usetikzlibrary{arrows}
\usepackage{a4wide}
\usepackage{changepage, parskip}
\usepackage{amsmath,amssymb} 
\usepackage{bbm}
\usepackage{color}
\usepackage{tikz}
\usepackage{tkz-graph}
\usepackage{tikz-cd} 
\usetikzlibrary{decorations.pathreplacing,calligraphy}
\usepackage{dsfont}
\usepackage{amscd}
\usepackage{graphicx}
\usepackage{mathtools}
\usepackage{todonotes}
\usepackage{enumerate}
\usepackage[normalem]{ulem}
\usepackage{float}

\usepackage{scalerel,stackengine}

\usepackage{comment}

\usepackage[pagewise]{lineno}

\usepackage{bm}

\theoremstyle{plain}
\newtheorem{theorem}{Theorem}[section]
\newtheorem{prop}[theorem]{Proposition}
\newtheorem{lemma}[theorem]{Lemma}
\newtheorem{coro}[theorem]{Corollary}

\theoremstyle{definition}
\newtheorem{definition}[theorem]{Definition}
\newtheorem{remark}[theorem]{Remark}

\newcommand{\birk}{b}

\newcommand{\Z}{{\mathbb Z}}
\newcommand{\Q}{{\mathbb Q}}
\newcommand{\R}{{\mathbb R}}
\newcommand{\N}{{\mathbb N}}
\newcommand{\C}{{\mathbb C}}
\newcommand{\X}{{\mathbb X}}
\newcommand{\T}{{\mathbb T}}
\newcommand{\cP}{\mathcal{P}}

\newcommand{\mc}{\mathcal}
\newcommand{\im}{{\mathrm{i}}}

\newcommand{\dom}{\operatorname{dom}}

\newcommand{\card}{\operatorname{card}}
\newcommand{\me}{\mathrm{e}}
\newcommand{\dd}{\,\mathrm{d}}
\newcommand{\ts}{\hspace{0.5pt}}
\newcommand{\pa}{\phantom{a}}

\newcommand{\sinc}{\operatorname{sinc}}
\renewcommand{\hat}{\widehat}
\renewcommand{\emptyset}{\varnothing}

\renewcommand{\c}{c}
\newcommand{\supp}{\mathrm{supp}}
\newcommand{\TT}{\mathbb{T}}

\newcommand{\myfrac}[2]{\frac{\raisebox{-2pt}{$#1$}}
      {\raisebox{0.5pt}{$#2$}}}

\usepackage{mathtools}

\DeclarePairedDelimiter{\ceil}{\lceil}{\rceil}
\DeclarePairedDelimiter{\floor}{\lfloor}{\rfloor}

\stackMath
\newcommand\reallywidecheck[1]{%
\savestack{\tmpbox}{\stretchto{%
   \scaleto{%
\scalerel*[\widthof{\ensuremath{#1}}]{\kern-.6pt\bigwedge\kern-.6pt}%
     {\rule[-\textheight/2]{1ex}{\textheight}}
   }{\textheight}%
}{0.5ex}}%
\stackon[1pt]{#1}{\scalebox{-1}{\tmpbox}}%
}

\begin{document}

\title[Generalized Thue{\ts}--Morse measures]{Generalized Thue{\ts}--Morse measures:  \\spectral and fractal analysis
}
 
\author{Philipp Gohlke}
\address{Faculty of Mathematics and Computer Science, Friedrich Schiller University,
\newline 
\hspace*{\parindent}Ernst-Abbe-Platz 2, 07745 Jena, Germany}
\email{philipp.gohlke@uni-jena.de}

\author{Marc Kesseb\"ohmer}
\address{Institute for Dynamical Systems, Faculty 3 – Mathematics und Computer Science,\newline 
\hspace*{\parindent}University of Bremen,\newline 
\hspace*{\parindent}Bibliothekstr.\ 5, 28359 Bremen, Germany}
\email{mhk@uni-bremen.de}

\author{Tanja I.\ Schindler}
\address{Faculty of Mathematics and Computer Science, Jagiellonian University\newline 
\hspace*{\parindent} Ul.\ Łojasiewicza 6, 30-348 Krakow, Poland\newline 
Department of Mathematics and Statistics,
University of Exeter\newline 
\hspace*{\parindent} Harrison Building, Exeter EX4 4QF,
United Kingdom}
\email{tanja.schindler@uj.edu.pl, t.schindler@exeter.ac.uk}
%

\begin{abstract}
We investigate a family of Riesz products  and show that they can be regarded as diffraction measures of generalized Thue{\ts}--Morse sequences, possibly over an infinite alphabet. 
These measures are closely related to the dynamical system arising from the doubling map together with an observable exhibiting a logarithmic singularity.  For this system, we develop a generalized thermodynamic formalism beyond the standard setting, which yields explicit formulas for  Birkhoff and dimension spectra.  
A further novel aspect is the identification of a precise connection between these spectra and the 
$L^q$-spectrum of the underlying Riesz product. This new link allows us to determine, explicitly, the Fourier and quantization dimension, and to describe the spectral asymptotics of the associated Kreĭn--Feller operator, providing new insights into the interplay between diffraction, fractal geometry, and spectral theory in the Thue{\ts}--Morse context. \end{abstract}



\maketitle

\section{Introduction}
Both the thermodynamic formalism and the concept of multifractal analysis provide powerful tools to analyze in detail the statistical properties of a given dynamical system and corresponding invariant measures. In the most classical setup, this provides a good understanding of the Birkhoff averages of H\"{o}lder continuous functions that are sampled along orbits of expanding (compact) dynamical systems and the scaling properties of associated equilibrium measures; see for example \cite{Bow,Falconer,Pesin} for comprehensive introductions to the topic. 
The last few decades have seen generalizations of this classical situation into many different directions, relaxing many of the strong regularity assumptions. This includes the consideration of non-uniformly hyperbolic and parabolic dynamical systems (see for example \cite{CFT19,IT11,JJOP09,KessStrat04,N00}) as well as progress under less restrictive assumptions on the continuity of the sampling function  (referred to as a potential in the following) \cite{KKS16,FO03,Olivier,BCJOe, ITZ}. We also refer to \cite{Cli14} for an overview.

In this work we want to contribute to the understanding of unbounded potentials, while maintaining the most simple setup of an expanding dynamical system (the doubling map). The multifractal study of unbounded potentials appears quite naturally, for example in the context of number theory, in an attempt to get a detailed understanding of the statistics of the coefficients in continued fraction expansions \cite{FLWW09}. Other instances include continuous potentials over non-compact expanding dynamical systems like shift-spaces over countable alphabets. In this setup, a systematic framework for the thermodynamic formalism was provided by Sarig \cite{Sarig99} and results about the multifractal analysis of (unbounded) continuous functions can be found for example in \cite{IJ15}.
Moreover, Baladi and Castorrini provide results concerning the thermodynamic formalism for unbounded potentials over piecewise expanding interval maps under some condition on the behaviour of the potential at the boundary of the partition the interval map is generating \cite{BalCas}. Kautzsch et al.\ showed in \cite{KKS16} that for intermittent maps and unbounded potentials with one singularity,  the convergence of the associated transfer operator depends on the   structure of  the $\omega$-limit set of the singularity.
Countable alphabet shifts can also be used via an appropriate encoding in the context of singular potentials over the doubling map, as was shown in \cite{KLRW19} at the example of the so-called Saint-Petersburg potential.

Despite the aforementioned progress, the understanding of singular potentials is far from what can be considered complete, even over such a simple system as the doubling map and under the restriction of having only one (one-sided) singularity.
Fan, Schmeling and Shen recently studied the (rescaled) singular potential $\psi^0(x) =2 \log |\cos(\pi x)|$ over the doubling map on the torus $ \mathbb{T}$ and a family of rigid translations $\psi^c(x) = \psi^0(x-c)$ and showed that the appearance of certain pathological features in the multifractal spectrum depends critically (and in a discontinuous manner) on the translation parameter $c\in  \mathbb{T}$ \cite{fan_generalized_2020, fan_generalized_2019}. We will take this family as the starting point for our analysis.
The Birkhoff averages of the potential $\psi^{1/2}$ are also of independent interest in aperiodic order because they are closely related to the scaling properties of the classic Thue{\ts}--Morse measure, which can be interpreted as the corresponding equilibrium measure \cite{BGKS, GLS}. 
In fact, each element of the family $(\psi^c)_{c \in \mathbb{T}}$ can be related to a generalized Thue{\ts}--Morse sequence $t^c$ as defined later in \eqref{eq: tnc} \cite{fan_generalized_2020}. Indeed, as we will establish below, the equilibrium measure corresponding to $\psi^c$ is given by the diffraction measure $\mu_c$ of the sequence $t^c$. Here, we rely on the fact that $t^c$ emerges as a limiting object of a substitution procedure over a compact alphabet.
While the theory of substitutions on finite alphabets is well-established \cite{baake,queffelec}, their counterparts on infinite alphabets have been studied much more sporadically \cite{DOP18,Fer06}.
A systematic study of substitution systems over compact alphabets was recently undertaken in \cite{MRW23,MRW} and we will make use of this framework for an appropriate interpretation of the diffraction measures.
Hence, the present work can be seen both as a case study of Birkhoff averages of singular potentials and as an in-depth analysis of spectral measures that arise naturally in the realm of aperiodic order. 
In fact, we will see that both frameworks complement each other in this study and that each of them offers additional insight.

With regards to the thermodynamic formalism, we show that the classical variational principle for the pressure function $p_c$ (with $p_c(t)$ denoting  the topological pressure of $t\psi^c$) extends to the present setting using an appropriate notion of topological pressure. Also, we verify that the pressure function coincides (up to a multiplicative constant) with the $L^q$-spectrum of the associated equilibrium measure. Combining this with the multifractal analysis performed in \cite{fan_generalized_2020}, this yields several ways to compute the dimension spectrum of both the Birkhoff averages of $\psi^c$ and the local dimensions of $\mu_c$, at least some of them numerically accessible \cite{BGKS}.
The equality between $L^q$-spectrum and pressure function also allows us to further address the question about the regularity of the parameter dependence. As was shown in \cite{fan_generalized_2020}, the dependence of the pressure function $p_c(t)$ on $c$ is highly discontinuous for negative $t$ and there is recent work indicating that it is in fact continuous for positive $t$ \cite{GLSprep}. 
Here, we show that $p_c(2)$ is an analytic function in $c$. In fact, we prove that the exponential of $p_c(2)$ coincides with the spectral radius of an explicit matrix $M_c$ whose entries are analytic in $c$. This observation relies on the interpretation of $p_c(2)$ as a Fourier dimension and makes decisive use of the autocorrelation structure of the generalized Thue{\ts}--Morse sequence $t^c$.
This extends previous work on the Thue{\ts}--Morse measure \cite{Zaks}, using the characterization of Fourier dimension as provided in \cite{Strichartz}.
Continuing along the path of dimensional analysis as for the Fourier dimension,   we use the relation between the pressure function  and  the $L^q$-spectrum of $\mu_c$ and prove that this spectrum has crucial regularity properties, which in turn affect further regularity properties of $\mu_c$.
In particular, the spectral dimension for Kre\u{\i}n--Feller operators associated to the measure $\mu_c$ and the quantization dimension for $\mu_c$ of order $r\geqslant 0$ both exist and can be read off entirely from the $L^q$-spectrum.

\section{Notation and main results}

The base dynamics throughout this work is given by the doubling map $(\TT,T)$, with $\TT\coloneqq \R/\Z$ the one-dimensional torus, and 
\[
T \colon \TT \to \TT, \quad x \mapsto 2 x \mod 1 .
\]
It is often useful to think of this action in terms of a binary representation of $x = \sum_{i=1}^{\infty} x_i 2^{-i}$, where each $x_i \in \Sigma = \{0,1\}$. The representation $(x_i)_{i \in \N}$ is unique for all $x$ except for the countable set of dyadic points, and the doubling map corresponds to the shift action 
\[
S \colon \Sigma^{\N} \to \Sigma^{\N}, \quad
(x_i)_{i \in \N} \mapsto (x_{i+1})_{i \in \N}.
\]
More precisely, if $\tilde{x}$ is a binary representation of $x$, then $S(\tilde{x})$ is a binary representation of $T(x)$.  
Stated formally, $\pi \colon (x_i)_{i \in \N} \mapsto \sum_{i=1}^{\infty}x_i 2^{-i}$ is a factor map from the dynamical system $(\Sigma^{\N},S)$ to $(\TT,T)$, which is invertible on the complement of a countable set. Indeed, equipping $\Sigma$ with the discrete topology and $\Sigma^{\N}$ with the corresponding product topology, it is straightforward to verify that $\pi$ is continuous. With some abuse of notation, we sometimes identify $x \in \TT$ with (any of) its binary representation(s) $(x_i)_{i \in \N}$ in $\Sigma^{\N}$.

Staying on the symbolical side for a moment, we refer to the elements of $\Sigma^n$ as \emph{words of length $n$} (for $n \in \N_{0}$) and define for $\omega = \omega_1 \cdots \omega_n \in \Sigma^n$ the corresponding \emph{cylinder set} by
\[
\langle \omega \rangle \coloneqq \left\{(x_i)_{i \in \N} \in \Sigma^\N : x_1 \cdots x_n = \omega_1 \cdots \omega_n \right\},
\]
in particular, if $\omega\in \Sigma^0$, we have the empty word and $\langle \omega \rangle=\Sigma^\N$.
The image of $\langle \omega \rangle$ under $\pi$ with $\omega\in \Sigma^n$ is a closed interval of length $2^{-n}$ in $\TT$, specified by the first $n$ digits in the binary representation and with the same abuse of notation as above, we also denote this interval by $\langle \omega \rangle$.

\subsection{A family of singular potentials}
Our main interest is centered around the family $(\psi^c)_{c \in \TT}$ with
\[
\psi^0 \colon \TT \to \TT, 
\quad x \mapsto 2 \log|\cos(\pi x)|,
\]
and $\psi^c(x) = \psi^0(x-c)$ for all $x,c \in \TT$. Note that the function $\psi^c$ obtains its maximal value $0$ at $x=c$ and has a singularity at the position $\breve{c} := c + 1/2$. 
Sampling the function $\psi^c$ along orbits of the doubling map, we obtain the corresponding \emph{Birkhoff sums}, given by
\[
\psi_n^{c}(x)\, \coloneqq\, \sum_{k=0}^{n-1} \psi^c(T^k x),
\]
for all $n \in \N$ and $x \in \TT$. The \emph{Birkhoff average} is given by
\[
\overline{\psi}^{c}(x) \,\coloneqq\, \lim_{n \to \infty} \frac{1}{n} \psi^{c}_n(x),
\]
whenever the limit exists. A classical object of interest is the corresponding \emph{dimension spectrum}
\[
\birk_c(\alpha)\, \coloneqq \,\dim_H \left\{ x \in \TT : \overline{\psi}^c(x) = \alpha \right\},
\]
where $\dim_H$ denotes Hausdorff dimension. 
As it is standard, we always take $\dim_H(\emptyset)=0$.

For each $t \in \R$, the \emph{variational pressure} of $t \psi$ is given by
\[
\mathcal{P}_{\operatorname{var}}\left(t \psi^{c} \right)
\,\coloneqq \,\sup_{\nu\in\mathcal{M}_{T}} h\left( \nu\right)+ \nu\left(t\psi^c\right),
\]
where $\mathcal{M}_T$ is the set of all $T$-invariant Borel probability measures on $\TT$, we denote $\nu\left(t\psi^c\right)=\int t\psi^c\,\mathrm{d}\nu$ and we call $\nu$ an \emph{equilibrium measure} of $t \psi^{c}$ if it achieves the above supremum.
A variant of this notion that was considered in \cite{fan_generalized_2020} is the restricted variational pressure, defined as
\[
 p_c(t)\,\coloneqq\,\mathcal{P}_{\operatorname{var},c}\left(t\psi^c\right)\,\coloneqq\,\sup_{\nu\in\mathcal{M}_{T,c}} h\left(\nu\right)+ \nu\left(t\psi^c\right),
\]
where $\mathcal{M}_{T,c}\coloneqq\{\nu\in\mathcal{M}_T\colon \breve{c} \notin \supp\left(\nu\right)\}$ and $\supp( \nu)$  denotes the topological support of $\nu$. 
The interest in this quantity is partly justified by the fact that it is intimately linked to the \emph{Birkhoff spectrum} $\birk_{c}(\alpha)$ via its Legendre transform, given by
\begin{equation*}
  p_c^{*} (a) \, \coloneqq \, \sup_{q\in\R}
  \bigl( q \ts a - p_{c} (q) \bigr),
\end{equation*}
for all $a \in \R$. By construction, both $p_c$ and $p_c^{\ast}$ are convex functions from $\R$ to $\R \cup \{\infty\}$. The domain of $p_c^{\ast}$ is $\operatorname{dom}(p_c^{\ast}) = \left\{\alpha \in \R : p_c^{\ast}(\alpha) < \infty \right\}$, and the interior of this domain is always given by an open interval $(\underline{\alpha}_c,\overline{\alpha}_c)$. In the present setting, $\underline{\alpha}_c \in \R \cup \{ - \infty\}$ and $\overline{\alpha}_c \in \R$ \cite{fan_generalized_2020}. With this notation the statement reads as follows.

\begin{theorem}[{\cite[Thm.~B]{fan_generalized_2020}}]\label{cor: dim}
   For every $c \neq 0$ and $\alpha \in (\underline{\alpha}_c,\overline{\alpha}_c)$, the Birkhoff spectrum of the function\/ $\psi^c$ is given by
   \begin{equation*}
  \birk_c (\alpha ) \, = \, \frac{-p_c^{*}\ts (\alpha)}{\log 2},
\end{equation*}
and $b_c(\alpha) = 0$ if $\alpha < \underline{\alpha}_c$ or $\alpha \geqslant \overline{\alpha}_c$.
\end{theorem}

\begin{remark}
It is worth pointing out that the statement of Theorem~\ref{cor: dim} fails in the special case $c=0$; see \cite[Thm.~5.1]{fan_generalized_2020} for details. We will come back to this particular case at a later point.
\end{remark}

As we will see below, the distinction between restricted and unrestricted pressure is unnecessary on the positive halfline. In fact, both expressions can be related to the \emph{topological pressure}. We define the \emph{upped} and \emph{lower topological pressure} as
\begin{align}
  \overline{\cP}_{\operatorname{top}} (t\ts \psi^{c}) \, 
  &\coloneqq \,\limsup_{n\to\infty}\myfrac{1}{n}\log\sum_{\omega\in\Sigma^n}
  \exp \bigl(t\ts \sup_{x\in \langle\omega\rangle}\psi^{c}_{n}(x)\bigr)\notag\\
  &\,= \,\,\begin{cases}\limsup_{n\to\infty} \myfrac{1}{n}\log\sum_{\omega\in\Sigma^n}
  \sup_{x\in \langle\omega\rangle}\exp \bigl(t\ts \psi^{c}_{n}(x)\bigr)& t\geqslant 0\\
  \limsup_{n\to\infty} \myfrac{1}{n}\log\sum_{\omega\in\Sigma^n}
  \inf_{x\in \langle\omega\rangle}\exp \bigl(t\ts \psi^{c}_{n}(x)\bigr)&t<0,  \end{cases}\label{eq: def top pressure}\end{align}
  and
  \begin{align*}
    \underline{\cP}_{\operatorname{top}} (t\ts \psi^{c})
  &\coloneqq \,\liminf_{n\to\infty}\myfrac{1}{n}\log\sum_{\omega\in\Sigma^n}
  \exp \bigl(t\ts \sup_{x\in \langle\omega\rangle}\psi^{c}_{n}(x)\bigr).
\end{align*}
We talk about the \emph{topological pressure} $\cP_{\operatorname{top}}$ if $\overline{\cP}_{\operatorname{top}}=\underline{\cP}_{\operatorname{top}}$.
 Our first main result is that a version of the classical variational principle persists in the present setting.

\begin{theorem}\label{thm: Ptop=Pvar}
Fix  $c\in \T$. Then for all $t\in\mathbb{R}$   we have 
 \begin{align}
  \mathcal{P}_{\operatorname{top}}\left(t\psi^c\right)=\mathcal{P}_{ \operatorname{var},c}\left(t\psi^c\right)\eqqcolon \mathcal{P}\left(t\psi^c\right)\label{eq: pressure 1}
 \end{align}
 and for $t\geqslant 0$,
\begin{align}
\mathcal{P}_{\operatorname{var}}\left(t\psi^c\right)= \mathcal{P}\left(t\psi^c\right).
\label{eq: pressure 2}
\end{align}
Moreover, the pressure function $t\mapsto \mathcal{P}\left(t\psi^c\right)$ is convex on $\R _{>0}$.
\end{theorem}

In fact, the topological pressure is helpful to obtain numerical calculations of the pressure itself or the dimension and the Birkhoff spectrum as was illustrated for the case $c =1/2$ in \cite{BGKS}.

\subsection{Equilibrium measures via diffraction}
Each of the functions $\psi^c$ has a very particular structure as it is the logarithm of a $g$-function $g^c$ in the sense of Keane \cite{Kea72}. More precisely, we call $g\colon \TT \to \TT$ a $g$-function if 
\[
\sum_{y \in T^{-1}x} g(y) = 1,
\]
for all $x \in \TT$. This relation is easily checked for the function $g^0(x) = \cos^2(\pi x)$ and naturally extends to all of its rigid translations $g^c(x) = g^0(x-c)$. Due to a well-known result by Ledrappier, the logarithm of a $g$-function has vanishing pressure and  admits an equilibrium measure \cite{Ledrappier}. Each such equilibrium measure is called a corresponding $g$-measure (or Doeblin measure).
That is,
\[
0 = \mathcal{P}_{\operatorname{var}}(\psi^c) = h_{\mu_c} + \int \psi^c \dd \mu_c,
\]
for some unique choice of $\mu_c \in \mathcal{M}_T$. 
If -- as in our case -- the $g$-function has at most one zero, then additionally the equilibrium measure is also unique, \cite{Kea72}.
Since $g^c$ is continuous and has a single $0$ for each $c$ the following product representation is a direct corollary of a result by Keane \cite{Kea72} (compare  Subsection~\ref{subsec:Gibbs-type} for details)
\begin{equation}
\label{eq:mu_c_product}
\mu_c = \prod_{n=0}^{\infty} 2 g^c(T^n x)
= \prod_{n=0}^{\infty} \bigl(1 + \cos(2\pi (2^n x - c) ) \bigr),
\end{equation}
where the right hand side is understood as the weak limit of densities with respect to Lebesgue measure. It should be noted that $\mu_0 = \delta_0$ coincides with the point mass at the origin and that $\mu_c$ is purely singular continuous with full topological support, for all $c\neq 0$ \cite{BCEG}.

In the special case $c=1/2$, it is well established that the Riesz product given in \eqref{eq:mu_c_product} describes the diffraction of the Thue{\ts}--Morse sequence. In order to extend this relation to arbitrary values of $c \in \TT$, 
we consider a {\em generalized Thue{\ts}--Morse sequence} $t^{c} = (t_n^{c})_{n \in \N_0}$, taking values in the unit circle $\mc A \coloneqq \left\{ \me^{2 \pi \im x} \mid x \in [0,1) \right\}$ in $\C$. 
More precisely, we set
\begin{equation}
t^{c}_n\, \coloneqq \, \me^{2 \pi \im c \mathsf{S}_2(n)},\label{eq: tnc}
\end{equation}
where $c\in \T$ and $\mathsf{S}_2(n)$ denotes the sum of digits in the binary expansion of $n$ (e.g.\@ $\mathsf{S}_{2}(7)=3$). A connection to the study of $\psi^c$ was already noted in \cite{fan_generalized_2020}. Here, we extend $t^c$ arbitrarily to the left, yielding a two-sided sequence $\widetilde{t^c}$ in $\mc A^{\Z}$ (equipped with the product topology), and let $X_c$ be the set of all accumulation points of $(S^n \widetilde{t^c})_{n \in \N_0}$. As before, we denote by $S$ the left shift on $X_c$, defined by $S(x_n) \coloneqq x_{n+1}$. Since $S X_c = X_c$, this gives a well-defined dynamical system $(X_c, S)$. In fact, we will see in Section~\ref{SEC:diffraction} that $(X_c,S)$ is strictly ergodic. 

To each $x = (x_n)_{n \in \Z} \in X_c$, we can assign a diffraction measure in the following manner. 
First, (denoting complex conjugation with a bar) we define the autocorrelation coefficients via
\begin{equation}\label{eq: autocorr coeff}
 \eta^c_n\, \coloneqq\, \lim_{k \to \infty} \frac{1}{k} \sum_{m = 0}^{k-1} \overline{x_m} x_{m+n},
\end{equation}
which is well-defined and independent of $x \in X_c$ for each $n \in \Z$ due to unique ergodicity. We therefore have a well-defined autocorrelation measure
\begin{equation*}
 \gamma_c\, \coloneqq\, \sum_{n \in \Z} \eta^c_n \delta_n,
\end{equation*}
where $\delta_x$ denotes the Dirac measure supported on $x\in \R$. 
The diffraction measure associated to $(X_c,S)$ is given by $\widehat{\gamma}_c$ as the Fourier transformation of $\gamma_c$. The Fourier transform $\widehat{\gamma}_c$ is a well-defined positive measure because $\gamma_c$ is a positive definite measure.
We refer to \cite{baake} and Section~\ref{sec: autocorr} for general background on autocorrelation measures and diffraction. 
The following result establishes the desired link to our earlier discussion.

\begin{prop}
For every $c \in \TT$, we have $\mu_c = \widehat{\gamma}_c\vert_{[0,1)}$. 
\end{prop}

In order to describe the local structure of a measure $\nu$ on $\TT$, we use the notion of a local dimension, given by
\[
d_{\nu}(x)\, \coloneqq\, \lim_{r \searrow 0} \frac{\log(\nu( B_r(x)))}{\log r},
\]
 if the limit exists, and where we denote by $B_r(x)$ the closed ball around $x$ with radius $r$ with respect to the Euclidean distance $\rho$. 
Similarly to the investigation of Birkhoff averages, we approach this quantity for $\mu_c$ via the corresponding dimension spectrum
\[
f_c(\alpha)\,\coloneqq\, \dim_H \left\{ x \in \TT : d_{\mu_c}(x) = \alpha \right\}.
\]
Making use of the analysis in \cite{fan_generalized_2020} we obtain the following.

\begin{theorem}\label{thm: dim spectrum}
For each $c \neq 0$ and $\alpha \in \R \setminus \left\{ -\underline{\alpha}_c/\log 2 \right\}$, we have $f_c(\alpha) =  \birk_c( - \alpha \log 2)$.
\end{theorem}

Hence, the dimension spectrum can also be expressed in terms of the Legendre transform of the pressure via Theorem~\ref{cor: dim}.

Another convenient way to study the structure of a (fully supported) measure $\nu$ on $\TT$ is via its \emph{upper} and \emph{lower} $L^q$-\emph{spectrum}, given for all $q \in \R$ by
\begin{align}
 \overline{\beta}_{\nu}(q)\, &\coloneqq\,  \limsup_{n\to\infty} \frac{1}{n\log 2} \log \left(\sum_{\omega\in\Sigma^n} {\nu}(\langle \omega \rangle)^q\right),\label{eq: Lq-spectrum}\\
\underline{\beta}_{\nu}(q)\, &\coloneqq\, \liminf_{n\to\infty} \frac{1}{n\log 2} \log \left(\sum_{\omega\in\Sigma^n} {\nu}(\langle \omega \rangle)^q\right).\nonumber
\end{align}
Moreover, in case that the limit in \eqref{eq: Lq-spectrum} exists, we simply write $\beta_{\nu}(q)$ and denote it as the   $L^q$-\emph{spectrum}.
In the present setup, this is closely related to the pressure function as the following theorem shows. 

\begin{theorem}\label{thm: pressure beta}
For each $c \neq 0$, we have  
\[
p_c = \log (2)\,\, \beta_{\mu_c}.
\]
\end{theorem}

This result can be regarded as a generalization of the corresponding statement for the Gibbs measures of H\"{o}lder continuous potentials. Indeed, although $\mu_c$ is not a (weak) Gibbs measure, our proof relies on a Gibbs type relation for $\mu_c$ given in Proposition~\ref{PROP:Gibbs-like}.

We remark here that the existence of the $L^q$-spectrum also has further implications concerning Kre\u{\i}n--Feller operators and the quantization dimension of the measure $\mu_c$. In particular, we have that the spectral dimension of the Kre\u{\i}n--Feller operator and the quantization dimension exist. 
This connection -- including the necessary notations and  concepts -- will be provided in detail in Section~\ref{sec: KreinFeller}.

A quantity that is particularly accessible is the Fourier dimension of the measure $\mu_c$ given by  $\beta_{\mu_c}(2)$. Using the autocorrelation structure of the generalized Thue{\ts}--Morse sequences, we will show the following.

\begin{theorem}\label{thm: beta(2) analytic}
The function  $c\mapsto \beta_{\mu_c}(2)$ is real-analytic on $\T$.
\end{theorem}

In fact, we will express $\beta_{\mu_c}(2)$ in terms of a \emph{correlation exponent} $D_2^{c}$, see \eqref{eq: def corr expo},  via the relation $\beta_{\mu_c}(2) = 1 - D_2^c$ (Corollary~\ref{COR:fourier-correlation}). The correlation exponent in turn can be written as the logarithm of the spectral radius of an explicit matrix (Theorem~\ref{PROP:D2-lambda}); see Figure~\ref{FIG:D2-illustration} for an illustration. 
This also shows that $p_c(2)$ depends analytically on $c$. The same is trivially true for $p_c(0) = 1$ and $p_c(1) = 0$, due to a common normalization, and certainly is wrong for $p_c(t)$ if $t < 0$ \cite{fan_generalized_2020}. For general $t>0$, continuity of $p_c(t)$ in $c$ holds \cite{GLSprep}, but stronger notions of regularity remain open in general.

\subsection{The special case $\boldsymbol{c=0}$}\label{subsec: c=0}

The point $c=0$ is the unique fixed point of the doubling map and as such, special properties may be expected for the potential $\psi^0$ that has its maximum at this position. In fact, this explains why $\mu_0 = \delta_0$ is the only $g$-measure in our family that is given by a pure point measure  \cite{ConzeRaugi}. Let us emphasize that $\psi^0$ is a coboundary, as was already noted in \cite{fan_generalized_2020}. More precisely,
\begin{equation}
\label{EQ:coboundary}
\psi^0(x) = 2 \log |\sin(2 \pi x)| - 2 \log |\sin(\pi x)| - 2 \log 2,
\end{equation}
for all $x \neq 0$. Following \cite{fan_generalized_2020}, this can be used to work out an explicit expression for the (variational) pressure function, given by
\begin{equation}
\label{EQ:var-P-at-0}
\mc P_{\operatorname{var},{0}}(t \psi^{0}) = \mc P_{\operatorname{var}}(t \psi^{0}) = \max\{ 1-2t, 0 \}.
\end{equation}
We will show explicitly that this expression coincides with the topological pressure, thus solving the case $c=0$ in Theorem~\ref{thm: Ptop=Pvar}.
On the other hand, the Birkhoff average $\overline{\psi}^c(x)$ can only assume three possible values, namely $0$ (for $x=0$), $- \infty$ (for dyadic rationals) and $- \log 2$ (Lebesgue almost surely). In particular, the relation between $b_c(\alpha)$ and $p_c^{\ast}(\alpha)$ given in Theorem~\ref{cor: dim} does not extend to the case $c=0$. 
Since $\mu_0$ is supported on a single point, both the $L^q$-spectrum and the dimension spectrum $f_0$ are trivial. In particular, the relation to the pressure function in Theorem~\ref{thm: pressure beta} is also not extendable to the case $c=0$.

\subsection{Structure of the paper}
In Section~\ref{sec: autocorr} we recall some facts about mathematical diffraction and in Section~\ref{SEC:diffraction} we introduce the generalization of the Thue{\ts}--Morse sequence $(t^c_n)$ and show that the diffraction measure of its associated system corresponds with the measure $\mu_c$. In Section~\ref{sec: autocorr coeff} we study in detail the autocorrelation coefficients of the generalized Thue{\ts}--Morse sequence $(t^c_n)$ and, in particular, we give a formula for the correlation exponent.
In Section~\ref{sec: pressure} we prove different statements regarding the variational and topological pressure and finally prove Theorem~\ref{thm: Ptop=Pvar}. 
In Section~\ref{sec:MeasureDec} we give some Gibbs type properties for the generalized Thue{\ts}--Morse measure and prove Theorems \ref{thm: dim spectrum} and \ref{thm: pressure beta}.
In Section \ref{sec: KreinFeller} we give 
applications of the $L^q$-spectrum to Fourier dimensions, Kre\u{\i}n--Feller operators, and quantization dimensions. 
Finally, in the last Section~\ref{sec:Technical} we prove 
combinatorial results and results regarding the growth rate of $\psi_n$ that we have used in Sections~\ref{sec: pressure} and \ref{sec:MeasureDec}.

\section{Autocorrelation and diffraction}\label{sec: autocorr}
In this section we recall some basic facts about mathematical diffraction in the context of compact alphabets. We refer the reader to \cite{baake} for background information and further details.
For convenience, we will simply write \emph{measure} for a (complex) Radon measure $\nu$ on $\R$, and given $g \in C_c(\R)$ (the compactly supported continuous functions on $\R$), we write $\nu(g)$ for the integral of $g$ with respect to $\nu$.
As before, we let $\mc A$ be the complex unit circle and $S$ the shift action on the compact sequence space $\mc A^{\Z}$. Let $X \subset \mc A^{\Z}$ be a subshift, that is, a closed and shift invariant subset of $\mc A^{\Z}$. 

Let us assume that $(X,S)$ is uniquely ergodic.
As in the previous section, we associate to a sequence $x \in X$ an \emph{autocorrelation measure}. This is a complex (Radon) measure, given by
\[
\gamma_x = \sum_{n \in \Z} \eta_n(x) \delta_n
, \quad
\eta_n(x) = \lim_{k \to \infty} \frac{1}{k} \sum_{m=0}^{k-1} \overline{x_m} x_{m+n}.
\]
Note that the autocorrelation coefficient $\eta_n(x)$ is the Birkhoff average of the continuous function $f_n \colon x \mapsto \overline{x_0} x_n$. Hence, the unique ergodicity of $(X,S)$ guarantees that $\eta_n = \eta_n(x)$ exists as a limit and is independent of $x \in X$. Similarly, $\gamma = \gamma_x$ is the same measure for all $x \in X$. 

Alternatively, the autocorrelation $\gamma$ is often introduced in terms of an Eberlein convolution. Given a sequence $x \in X$, we assign a weighted Dirac comb
\[
\omega = \omega(x) = \sum_{k \in \Z} x_k \delta_k.
\]
In the context of diffraction theory, this models an atomic configuration of equally spaced scatterers of different types. 
In the following, let $\nu_1 \ast \nu_2$ denote the convolution of two (finite) measures $\nu_1$ and $\nu_2$. Further, for a general measure $\nu$ on $\R$, let $\widetilde{\nu}$ be the measure determined by $\widetilde{\nu}(g) = \overline{\nu(\widetilde{g})}$, where $\widetilde{g}(x) = \overline {g(-x)}$, for every continuous function of compact support $g \in C_c(\R)$.
For $n \in \N$, let $\omega_n$ be the restriction of $\omega$ to the interval $[0,n]$. Then, using that $\widetilde{\alpha \delta_k} = \overline{\alpha} \delta_{-k}$, a straightforward calculation yields that
\begin{equation}
\label{EQ:gamma-Eberlein}
 \gamma = \omega \circledast \widetilde{\omega} := \lim_{n \to \infty} \frac{\omega_n \ast \widetilde{\omega_n}}{n},
\end{equation}  
to be understood as a vague limit of measures.
By construction, the measure $\gamma$ is \emph{positive definite} in the sense that $\gamma(g \ast \widetilde{g})\geqslant 0$ for every $g \in C_c(\R)$. 

The diffraction measure associated to $X$ is the Fourier transform $\widehat{\gamma}$ of the autocorrelation measure $\gamma$. Since the Fourier transformability of measures is a subtle issue, we recall some basic facts on the Fourier transform of tempered distributions and measures; see also \cite[Ch.~8]{baake} for more details. As an invertible transformation on the space $\mc S(\R)$ of Schwartz functions, we define the Fourier transform of $g \in \mc S(\R)$ as
\[
\widehat{g} \colon y \mapsto \int_{\R} g(x) \me^{-2 \pi \im x y} \dd x
\]
and the inverse Fourier transform
\[
\reallywidecheck{g} \colon y \mapsto \int_{\R} g(x) \me^{2 \pi \im x y} \dd x.
\]
This also induces an invertible transformation on the dual space $\mc S'(\R)$ of tempered distributions via $\widehat{u}(g) =  u(\widehat{g})$ for all $ u \in \mc S'(\R)$ and $g \in \mc S(\R)$ (and similarly for the inverse Fourier transform). If both $u$ and $\widehat{u}$ are also measures, this defines the (inverse) Fourier transform of the corresponding measures.

If $\nu$ is positive definite, it has a well-defined Fourier transform $\widehat{\nu}$ which is a positive measure, due to a theorem by Bochner--Schwartz; compare \cite[Prop.~8.6]{baake}. This applies in particular to the autocorrelation measure. Taking the Fourier transform in \eqref{EQ:gamma-Eberlein}, it is natural to inquire whether the Fourier transform commutes with the vague limit of measures. While this is not true in full generality \cite{SpindelerStrungaru}, it \emph{does} hold under the restriction to positive definite measures.

\begin{theorem}[{\cite[Lem.~4.11.10]{MS17}}]
\label{THM:FT-continuity}
The Fourier transform restricted to the cone of positive definite measures is continuous with respect to the vague topology.
\end{theorem}

Note that each of the measures $\omega_n \ast \widetilde{\omega}_n$ is positive definite with the Fourier transform given by $|\widehat{\omega_n}|^2 \mathrm{Leb}$ (with $\mathrm{Leb}$ denoting Lebesgue measure and $f\mathrm{Leb}$ the absolutely continuous measure with density $f$). For a finite measure $\nu$, we identify $\widehat{\nu}$ with the (density) function
\[
\widehat{\nu} \colon k \mapsto \int_{\R} \me^{-2\pi \im kx} \dd \nu(x).
\]
 Hence, we obtain the following immediate consequence of \eqref{EQ:gamma-Eberlein} and Theorem~\ref{THM:FT-continuity}.

\begin{coro}
\label{COR:diffraction-limit}
The diffraction measure is given by the vague limit
\[
\widehat{\gamma} = \lim_{n \to \infty} \frac{|\widehat{\omega_n}|^2}{n} \mathrm{Leb}.
\]
\end{coro}

Since $\gamma$ is supported on $\Z$, it is straightforward to verify that $\widehat{\gamma}$ is $\Z$-periodic. More precisely, it is of the form $\widehat{\gamma} = \nu \ast \delta_\Z$ with $\nu$ a finite measure on the unit interval. In this situation, the autocorrelation coefficients are effectively given by the Fourier--Stieltjes coefficients of $\nu$. This is well known, see for example \cite{BLvE15}. Since conventions on the Fourier transform differ throughout the literature, we give a short proof for convenience.

\begin{lemma}
\label{LEM:FS-coefficients}
Let $\widehat{\gamma}$ be the diffraction measure of a uniquely ergodic subshift and $\nu$ the restriction of $\widehat{\gamma}$ to the unit interval $[0,1)$. Then, the autocorrelation coefficients satisfy
\[
\overline{\eta_n} = \widehat{\nu}(n),
\]
for all $n \in \Z$.
\end{lemma}

\begin{proof}
Since $\widehat{\gamma} = \nu \ast \delta_\Z$, it follows from the Poisson summation formula that
\[
\gamma = \reallywidecheck{\nu \ast \delta_{\Z}} = \reallywidecheck{\nu} \delta_{\Z},
\]
compare for example \cite[Thm.~8.5, Ex.~9.2]{baake}, implying $\eta_n = \reallywidecheck{\nu}(n)$ for all $n \in \Z$. It is readily verified from the definition that $\overline{\eta_n} = \eta_{-n}$ and
\[
\reallywidecheck{\nu}(n) = \int_{[0,1)} \me^{2\pi \im n x} \dd \nu(x) = \widehat{\nu}(-n),
\]
for all $n \in \Z$. This yields the claimed equality.
\end{proof}

\section{The Riesz product measure as diffraction measure of a substitutional dynamical system}
\label{SEC:diffraction}

We consider the generalization of the Thue{\ts}--Morse sequence $t^{c} = (t_n^{c})_{n \in \N_0}$ given in \eqref{eq: tnc}.
This is a one-sided sequence on the compact alphabet $\mc A = \left\{ \me^{2 \pi \im x} \mid x \in [0,1) \right\}$. 
Let $R_c \colon z \mapsto \me^{2 \pi \im c} z$ denote the circle rotation on $\mc A$. 
The sequence $(t^{c}_n)$ is a fixed point of an appropriate substitution $\vartheta_c$ on the (compact) alphabet $\mc A$. Before we explore this connection further, let us recall some symbolic notation. We refer to elements of $\mc A$ as \emph{letters}. A \emph{word} is a formal concatenation of letters $u = u_1 \cdots u_n \in \mc A^n$, with $u_i \in \mc A$ for all $1 \leqslant i \leqslant n$ and $n \in \N$ and we denote by $|u|$ the length of the word $u$.
We equip $\mc A^n$ with the product topology and the set of finite words $\mc A^+ = \bigcup_{n\geqslant 1} \mc A^n$ with the disjoint union topology. A \emph{substitution} $\vartheta$ on $\mc A$ is a continuous function from $\mc A$ to $\mc A^+$. It is extended to a semigroup homomorphism on $\mc A^+$, $\mc A^\N$ and $\mc A^\Z$ via formal concatenation. That is $\vartheta(u_1 \cdots u_n) = \vartheta(u_1) \cdots \vartheta(u_n)$, and analogously for the action on (bi-)infinite sequences. For background on substitutions on \emph{finite} alphabets, we refer the reader to \cite{baake}. A systematic framework for substitutions on \emph{compact} alphabets was more recently given in \cite{MRW}.

It is a matter of straightforward calculation to verify that the sequence $t^c$ is invariant under
 the following substitution on the alphabet $\mc A$,
\begin{equation}
\label{EQ:theta-subst}
\vartheta_{c} \colon z \mapsto z \, R_c(z),
\end{equation}
where $z \, R_c(z)$ is supposed to be understood as a concatenation of $z$ and $R_c(z)$.

Extend $t^c$ arbitrarily to a sequence $\widetilde{t^c}$
in $\mc A^{\Z}$ and let $X_{c}$ be the set of all accumulation points of $(S^n \widetilde{t^c})_{n \in \N_0}$, where, as before, $S$ is the left shift on  $X_c$. 
Then, $(X_c,S)$ is a uniquely ergodic subshift of $(\mc A^{\Z},S)$, compare \cite[Cor.~12.2, Ex.~12.1]{queffelec}. As the following remark shows, $(X_{c},S)$ is minimal and hence, in fact, strictly ergodic.
\begin{remark}
If $c$ is rational, all entries of $t^c$ are contained in the finite alphabet $\mc B = \left\{R_c^n(1): n \in \N \right\}$. In this case, $t^c$ is the unique fixed point of the primitive substitution $\vartheta_c$ on $\mc B$ and it follows from the classical theory that $(X_c,S)$ is strictly ergodic \cite{baake}. 
If $c$ is irrational, $\vartheta_c$ is a substitution on a compact alphabet, and we can associate a subshift $(X_{\vartheta_c},S)$ that is generated from the language of the substitution, as defined in \cite{MRW}. In their notation, $\vartheta_c$ is a \emph{primitive} substitution and hence the minimality of $(X_{\vartheta_c},S)$ follows from \cite[Prop.~33]{MRW}. It is worth noticing that $X_c$ and $X_{\vartheta_c}$ coincide. Indeed, $X_c \subset X_{\vartheta_c}$ is verified from the definition and equality follows from the fact that $X_{\vartheta_c}$ is minimal and $X_{c}$ is a closed invariant subspace of $X_{\vartheta_c}$.
\end{remark}

Note that $t^{c}_{[0,2^n-1]} = \vartheta_c^n(1)$ for all $n \in \N$. We extend the circle rotation $R_c$ to words as a homomorphism, that is $R_c(v_1 \cdots v_m) = R_c(v_1) \cdots R_c(v_m) \in \mc A^m$. We observe $R_c \circ \vartheta_c = \vartheta_c \circ R_c$. For $u_c^n = \vartheta_c^n(1)$ we easily obtain the recursion
\begin{equation}
\label{EQ:w_n-recursion}
u_c^{n+1} =  \vartheta_c^{n+1}(1) =  \vartheta_c^n(1) \vartheta_c^n(R_c(1)) = u_c^n \, R_c(u_c^n).
\end{equation}

\begin{remark}
If $c$ is irrational, there is an uncountable family $\left\{t^c(\alpha)\right\}_{\alpha \in  \mc A}$ of points in $X_c$ that restrict to $t^c$ on the non-negative integers, but none of these is a fixed point for $\vartheta_c$ (or any power of $\vartheta_c$). 
These points can be constructed as follows. Given $\alpha \in \mc A$, define $\alpha_n = R_c^{-n}(\alpha)$ for all $n \in \N$. Note that all two-letter words in $\mc A \times \mc A$ are legal for $\vartheta_c$ (for a formal definition of legal words see \cite{MRW}). Hence, all of the words
\[
\vartheta_c^n(\alpha_n) \vartheta_c^n(1)
\] 
are legal. Note that $\vartheta_c^{n+1}(1) = \vartheta_c^n(1) R_c \vartheta_c^n(1)$ has $\vartheta_c^n(1)$ as a prefix. Similarly,
\[
 \vartheta_c^{n+1}(\alpha_{n+1}) = R_c^{-1}( \vartheta_c^{n+1}(\alpha_n)) = R_c^{-1} ( \vartheta_c^n(\alpha_n)) \,  \vartheta_c^n(\alpha_n)
\]
has $\vartheta_c^n(\alpha_n)$ as a suffix for all $n \in \N$. Hence, the point $t^c(\alpha)$ with 
\[
t^c(\alpha)_{[-2^n, 2^n-1]} = \vartheta_c^n(\alpha_n)  \vartheta_c^n(1)
\]
for all $n \in \N$ is well defined and contains only legal words.
Here and in the following $u_{[j,k]}$ denotes the finite word of the $j$th digit to the $k$th digit of the (possibly infinite) word $u$.
Note that $t^c(\alpha)_{-1} = \alpha$, that is, $t^c(\alpha)$ is parametrized by its first entry to the left of the origin.
\end{remark}

\begin{prop}
For $c\in \T$, the diffraction measure associated to $(X_{c},S)$ is given by the Riesz product
\begin{equation*}
\widehat{\gamma}_{c} = \prod_{n = 0}^{\infty} (1 + \cos(2\pi (2^n x - c))),
\end{equation*}
which converges in the vague topology. In particular, $\widehat{\gamma}_{c} = \mu_c \ast \delta_\Z$, with $\mu_c$ as in \eqref{eq:mu_c_product}.
\end{prop}

\begin{proof}
Let $t^c = t^c(\alpha)$ be any of the sequences in $X_c$ that coincides with $\widetilde{t}^c$ on all non-negative indices 
and $\omega^c$ the corresponding weighted Dirac comb. By Corollary~\ref{COR:diffraction-limit}, we can write the diffraction measure as
\[
\widehat{\gamma}_{c} = \lim_{n \to \infty} \frac{|\widehat{\omega_{2^n}^c}|^2}{2^n} \mathrm{Leb},
\]
where
\[
\widehat{\omega_{2^n}^c}(x) 
= \sum_{k = 0}^{2^n-1} t_k^c \me^{2 \pi \im k x}.
\] 
Recall that the first $2^n$ entries of $t^c$ are given by the word $u_c^n = \vartheta_c^n(1)$.
On the level of Dirac combs, \eqref{EQ:w_n-recursion} implies 
\[
\omega_{2^{n+1}}^c = \sum_{k = 0}^{2^n - 1} t_k^{c} \delta_k + \sum_{k = 0}^{2^n - 1} R_c(t_k^{c}) \delta_{2^n + k}  
= \omega_{2^n}^c + \me^{2 \pi \im c} \, \delta_{2^n} \ast \omega_{2^n}^c.
\]
Hence, we find for $h_n = |\widehat{\omega_{2^n}^c}|^2/2^n$ the recursion
\begin{align*}
h_{n+1}(x) &= \frac{1}{2^{n+1}} \Big|\big(1 + \me^{-2 \pi \im (2^n x - c)}\big)\, \widehat{\omega_{2^n}^c}(x)\Big|^2 
= \frac{1}{2} \Big|1 + \me^{-2 \pi \im (2^n x - c)}\Big|^2 h_n(x) 
\\ &=(1 + \cos(2\pi (2^n x - c)) ) h_n(x).
\end{align*}
Since $h_0(x) = 1$, the diffraction measure becomes
\[
\widehat{\gamma}_c = \prod_{k = 0}^{\infty} (1 + \cos(2\pi (2^n x - c))),
\]
to be understood as a vague limit of (Lebesgue) density functions.
\end{proof}

\section{Autocorrelation coefficients and correlation exponent}\label{sec: autocorr coeff}
For the Thue{\ts}--Morse sequence (corresponding to the case $c = 1/2$) it is well known that the autocorrelation coefficients given in \eqref{eq: autocorr coeff} satisfy the recursive relations
\[
\eta_{2n}^{1/2} = \eta_n^{1/2},
\quad \eta_{2n+1}^{1/2} = - \frac{\eta_n^{1/2} + \eta_{n+1}^{1/2}}{2},
\]
for all $n \in \N_0$, see \cite[Ch.~10]{baake}.
This recursive structure is very useful to study both the autocorrelation and the diffraction measure of the Thue{\ts}--Morse sequence. For example, it can be used to show that the Thue{\ts}--Morse sequence is singular continuous. It has also been used in \cite{BG19} to obtain a detailed understanding of the scaling behaviour of the Thue{\ts}--Morse measure at the origin. 
For general $c$, it follows already from 
\eqref{eq: autocorr coeff} that $\eta_{-n}^c = \overline{\eta_n^c}$ for all $n \in \Z$.
We will use the following generalization of the recursion relations of the Thue{\ts}--Morse sequence to obtain an expression for the Fourier dimension of the diffraction measure in the spirit of Zaks et al.\ in \cite{Zaks}.

\begin{prop}
\label{PROP:eta-recursion}
Let $c\in \T$ and $\phi_c = \me^{2 \pi \im c}$. Then, the autocorrelation coefficients of the sequence $(t_n^c)_{n \in \N_0}$ satisfy the relations
\begin{align*}
\eta_{2n}^{c} & = \eta_n^{c},
\\ \eta_{2n+1}^{c} & = \frac{1}{2}(\phi_{c} \eta_n^{c} + \overline{\phi_{c}} \eta_{n+1}^{c}),
\end{align*}
for all $n \in \N_0$. Further, $\eta_0^{c} = 1$ and 
\[
\eta_1^{c} = \frac{\phi_{c}}{2 - \overline{\phi_{c}}}.
\]
\end{prop}

\begin{proof}
Following the same steps as in \cite[Ch.~10]{baake}, this follows in a straightforward manner from the relations 
\[
t_{2n}^c = t_n^{c}, \quad t_{2n+1}^c = \phi_{c} t_n^c,
\]
for all $n \in \N_0$, which can be verified from the fact that the sequence is a fixed point of the substitution $\vartheta_{c}$ as defined in \eqref{EQ:theta-subst}. Indeed, observe that for $k \in \N$,
\begin{align*}
\sum_{m=0}^{2k-1} \overline{t^c_{m}} t_{m+2n}^c
& = \sum_{m=0}^{k-1} \overline{t^c_{2m}} t_{2m + 2n}^c + 
\sum_{m=0}^{k-1} \overline{t^c_{2m+1}} t_{2m + 2n + 1}^c
\\ &= \sum_{m=0}^{k-1} \overline{t^c_m} t^c_{m+n} 
+ |\phi_{c}|^2 \sum_{m=0}^{k-1} \overline{t^c_m} t^c_{m+n}
= 2 \sum_{m=0}^{k-1} \overline{t^c_m} t^c_{m+n}.
\end{align*}
Dividing by $2k$ and sending $k \to \infty$ yields the first relation $\eta_{2n}^{c} = \eta_n^{c}$. The second relation is obtained in a similar manner, again regrouping a finite sum according to odd and even indices. The statement $\eta_0^{c} = 1$ follows from the definition, noting that each element $t_n^c$ of the sequence is an element of the unit circle. The expression for $\eta_1^{c}$ follows by applying the second recursion relation in the case $n = 0$.
\end{proof}
Recall that $\hat \nu:\R\to \C, x\mapsto \int \me^{-2\pi i x \xi  }\dd \nu(\xi)  $ denotes the {\em Fourier transform} of the Borel measure $\nu$ on $\T$. Then the {\em upper Fourier dimension} of a measure $\nu$ is given by
\begin{eqnarray}\label{DimF}
\overline{\dim}_{F}(\nu)\coloneqq \limsup_{R\to \infty} \frac{\log \left(R^{-1}\int_{-R}^{R}|\hat \nu (x)|^{2}\dd x\right)}{\log R}
\end{eqnarray}
and we talk about the {\em Fourier dimension} $\dim_{F}(\nu)$ in case the $\limsup$ in \eqref{DimF} is a limit.
The Fourier dimension of  $\widehat{\gamma}_{c}$  is related to the growth behavior of the sequence $(\Theta_n^c)_{n \in \N}$, with
\[
\Theta_n^c = \sum_{m = 0}^{n-1} |\eta_m^c|^2.
\]
This growth property is captured in the \emph{correlation exponent} given by
\begin{equation}
D_{2}^{c} \coloneqq \limsup_{n \to \infty} \frac{\log \Theta_n^{c}}{\log n}.\label{eq: def corr expo}
\end{equation}

Since $\Theta_n^c$ is non-decreasing in $n$, a straightforward interpolation argument shows that $D_{2}^{c}$ may alternatively be defined along the subsequence $(2^n)_{n \in\N}$, i.e.\ 
$$D_2^c = \limsup_{n \to \infty} \frac{\log \Theta_{2^n}^{c} }{n \log 2}.$$
We will see in Section~\ref{SEC:Fourier-dimension} that this characteristic number is related to the Fourier dimension of $\widehat{\gamma}_{c}$ and  to  the corresponding pressure function evaluated in $2$. Similarly to the proof of Proposition~\ref{PROP:eta-recursion}, we would like to establish a recursive structure for $\Theta_n^{c}$ by splitting the sum according to odd and even indices and using  Proposition~\ref{PROP:eta-recursion}
for $\eta_{2m}^{c}$ and $\eta_{2m+1}^{c}$. Since this creates several ``cross-terms'', we introduce some auxiliary sequences, defined via
\begin{align*}
Z_n^c &= \sum_{m = n}^{2n-1} |\eta_m^c|^2,
\\ 
\Pi_n^c &= \frac{\phi_c^2}{2}\, \sum_{m=n}^{2n-1} \eta_{m}^{c} \overline{\eta_{m+1}^{c}},
\end{align*}
for all $n \in \N$. 

\begin{prop}
\label{PROP:v-M-recursion}
For every $\c\in \T$ and $n \in \N$, let $v_n =v_n^c = (Z_n^c, \Pi_n^c, \overline{\Pi^c_n})^{T}$. Then, we have $v_{2n} = M_c v_n$ for all $n \in \N$, where
\[
M_c = \frac{1}{2} \begin{pmatrix}
3 & 1 & 1 \\
\phi_{c} & 2 \phi_{c} & 0\\
\overline{\phi_{c}} & 0 & 2 \overline{\phi_{c}}
\end{pmatrix}.
\]
\end{prop}

\begin{proof}
We prove the matrix relation line by line. Recall,  $\phi_{c}=\me^{2 \pi \im c}$.  First, note that
\begin{align*}
Z^{c}_{2n} &= \sum_{m = n}^{2n -1} \left|\eta^{c}_{2m}\right|^2 + \sum_{m = n}^{2n-1} \left|\eta^{c}_{2m+1}\right|^2
 = \sum_{m =n}^{2n-1} \left|\eta^{c}_m\right|^2 + \frac{1}{4}\sum_{m = n}^{2n-1} \left| \phi_{c} \eta^{c}_m + \overline{\phi_{c}} \eta^{c}_{m+1}\right|^2
 \\ & = Z^{c}_n + \frac{1}{2} Z^{c}_n + \frac{1}{4} \sum_{m = n}^{2n-1} \left(\phi_{c}^2 \eta^{c}_m \overline{\eta^{c}_{m+1}} + \overline{\phi_{c}}^2\, \overline{\eta^{c}_m} \eta^{c}_{m+1}\right) 
 \\ & = \frac{3}{2} Z^{c}_n + \frac{1}{2} \Pi^{c}_n + \frac{1}{2} \overline{\Pi^{c}_n},
\end{align*}
which gives the first line of the matrix relation. The second line is found in a similar manner via
\begin{align*}
\Pi^{c}_{2n} & = \frac{\phi_{c}^2}{2} \biggl( \sum_{m=n}^{2n-1} \eta^{c}_{2m} \overline{\eta^{c}_{2m+1}} + \sum_{m = n}^{2n-1} \eta^{c}_{2m+1} \overline{\eta^{c}_{2m+2}} \biggr)
\\ & = \frac{\phi_{c}^2}{4} \sum_{m=n}^{2n-1} \left(\overline{\phi_{c}} |\eta^{c}_m|^2 + \phi_{c} \eta^{c}_m \overline{\eta^{c}_{m+1}} \right)
+ \frac{\phi_{c}^2}{4} \sum_{m=n}^{2n-1} \left(\phi_{c} \eta^{c}_m \overline{\eta^{c}_{m+1}} + \overline{\phi_{c}} |\eta^{c}_{m+1}|^2 \right)
\\ & = \frac{\phi_{c}}{2} Z^{c}_n + \phi_{c} \Pi^{c}_n,
\end{align*}
and the last line follows by taking the complex conjugate of the second line.
\end{proof}

\begin{lemma}
\label{LEM:M-eigenvalues}
For all $c\in \T$, there is a unique   eigenvalue $\lambda_1^c$ of $M_{c}$ with modulus strictly greater than $1$. This eigenvalue is  real and  depends  real-analytically on $c$. An eigenvalue lies on the unit circle  if and only if $c\in \{0,1/2\}$. 
\end{lemma}

\begin{proof}
By symmetry, we may restrict our attention to $c \in [0,1/2]$.
It is easily verified that a value $\lambda^{c}$ is an eigenvalue of $M_{c}$ if and only if the same holds for its complex conjugate $\overline{\lambda^{c}}$. Hence, for each $c$, either all eigenvalues are real, or one of them is real and the other two form a complex conjugate pair.
Since all matrix entries of $M_{c}$ depend analytically on $c$, it follows by standard perturbation theory that there is an (unordered) tuple $\lambda_r^{c},\lambda_s^{c},\lambda_t^{c}$ of eigenvalues that depends analytically on $c$ \cite{kato}. Comparing the coefficients of the characteristic polynomial of $M_{c}$ and its Jordan normal form yields
\begin{align*}
\lambda_r^{c} + \lambda_s^{c} + \lambda_t^{c} &= \frac{3}{2} + 2 \cos(2 \pi c),
\\ \lambda_r^{c} \lambda_s^{c} + \lambda_r^{c} \lambda_t^{c} + \lambda_s^{c} \lambda_t^{c} &= 1 + \frac{5}{2} \cos(2 \pi c),
\\ \lambda_r^{c} \lambda_s^{c} \lambda_t^{c} &= 1.
\end{align*}
Assuming $\lambda^{c} = \lambda_s^{c} = \lambda_t^{c}$ (implying $\lambda_r^{c} = (\lambda^{ c})^{-2}$) leads to the system 
\begin{align*}
4 (\lambda^c)^3 - (4 + 10 \cos(2 \pi c)) \lambda^{c} + 8 &= 0,
\\ 4 (\lambda^c)^3 - (3 + 4 \cos(2 \pi c)) (\lambda^c)^2 + 2 & = 0.
\end{align*}
This can be verified to have precisely two pairs of solutions $(c,\lambda^{c})$ with $c \in [0,1/2]$. In both cases $|\lambda^{c}|<1$. 
Further, an explicit calculation shows that $\lambda^{c} = 1$ is an eigenvalue if and only if $c = 0$ and $\lambda^{c} = -1$ is an eigenvalue if and only if $c = 1/2$. In both cases there is a unique third eigenvalue $\lambda_1^{c} > 1$. 
Without loss of generality, for $c=0$, assume that the indices $r,s,t$ are assigned in such a way that $\lambda_1^{c} = \lambda_r^{c}$, $\lambda_s^{c} = 1$ and $\lambda_t^{c} < 1$.

We argue that $\lambda_r^{c}> 1$ for all $c \in [0,1/2]$. Indeed, by continuity, it cannot decrease below $1$ along the real line because there is no eigenvalue $\lambda^{c} = 1$ for $c \in (0,1/2]$. Also, $\lambda_r^{c}$ cannot split into a complex conjugate pair because eigenvalues $\lambda^{c}$ of non-trivial multiplicity occur only for $|\lambda^{c}| < 1$. 
Conversely, for some $c \in (0,1/2)$ we obtain that $\lambda_s^{c} = \lambda_t^{c}$ is contained in the unit disc. Both $\lambda_s^{c}$ and $\lambda_t^{c}$ cannot leave the unit disc for any $c \in (0,1/2)$. Indeed, this is impossible as a complex conjugate pair because it would contradict $\lambda_r^{c} > 1$ and $\lambda_r^{c} \lambda_s^{c} \lambda_t^{c} =1$. Also, along the real line there are no eigenvalue solutions for $|\lambda^{c}| = 1$ if $c \in (0,1/2)$. 
\end{proof}

\begin{figure}
 \includegraphics[width=0.55\textwidth]{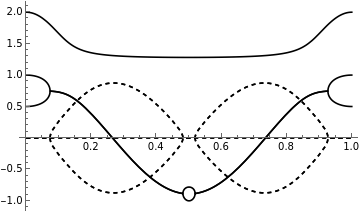}
    \caption{The graph of $c\mapsto \lambda^c_1$ (upper straight line) and the real parts (lower straight lines) 
    and imaginary parts (dashed lines) of the remaining two eigenvalues.}
\end{figure}

\begin{theorem}
\label{PROP:D2-lambda}
The correlation exponent is well defined and given by
\[
D_2^{c} = \frac{\log \lambda_1^{c}}{\log 2},
\] 
for all $c\in \T$. In particular, $c\mapsto D_2^{c}$ is a real-analytic function on $\T$.
\end{theorem}

\begin{figure}
 \includegraphics[width=0.55\textwidth]{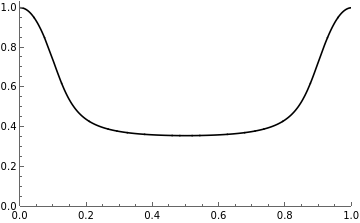}
    \caption{Illustration of $D_2^c$.}\label{FIG:D2-illustration}
\end{figure}

\begin{proof}
If $c=0$, the observation $\eta_n^{0} = 1$ for all $n \in \N$ easily implies the desired relation
\[
D_2^{0} = \frac{\log 2}{\log 2} = 1.
\]
We hence assume $c \in (0,1)$ in the following. Let $\lambda_1^{c},\lambda_2^{c},\lambda_3^{c}$ be the eigenvalues of $M_{c}$, ordered via $|\lambda_1^{c}| \geqslant |\lambda_2^{c}| \geqslant |\lambda_3^{c}|$.
For the sake of establishing a contradiction, assume that $v_1^{c} = ({Z}_1^{c},\Pi_1^{c},\overline{\Pi_1^{c}})^T$ is contained in the two-dimensional subspace that is spanned by the (generalized) eigenvectors corresponding to $\lambda_2^{c}$ and $\lambda_3^{c}$.
We distinguish two cases. 
If $c \neq 1/2$, it follows by Lemma~\ref{LEM:M-eigenvalues} that $|\lambda_2^{c}|,|\lambda_3^{c}| < 1$. We obtain that $v_{2^n}^{c} = M_{c}^n v_1^{c} \to 0$ as $n \to \infty$, in contradiction to the observation that $Z_{2n}^{c} \geqslant Z_n^{c}$ for all $n \in \N$ and 
\[
Z_1^{c} = |\eta_1^{c}|^2 = \frac{1}{|2 - \phi_{c}|^2} \neq 0.
\]
If $c = 1/2$, we have $\lambda_2^{c} = -1$, $|\lambda_3^{c}| < 1$, again in contradiction to $Z_{2n}^{c} \geqslant Z_n^{c}$ for all $n \in \N$.
It follows that $v_1^{c}$ has a non-trivial component in the direction of the eigenvector $u^{c}$ of $\lambda_1^{c}$.
Hence, we have
\begin{equation}
\label{EQ:v_n-growth}
\lim_{n \to \infty} (\lambda_1^{c})^{-n} v_{2^n}^{c}  = \lim_{n \to \infty} (\lambda_1^{c})^{-n} M^n_{c} v_1^{c} = \alpha u^{c},
\end{equation}
for some $\alpha > 0$. It is straightforward to see that $u_1^{c} \neq 0$. Indeed assuming $u^{c} = (0,u_2^{c},u_3^{c})$ would lead to
\[
\lambda_1^{c} u^{c} = M_{c} u^{c} = \begin{pmatrix}
(u_2^{c} + u_3^{c})/2
\\ \phi_{c} u_2^{c}
\\ \overline{\phi_{c}} u_3^{c} 
\end{pmatrix},
\]
which gives rise to the contradiction $1 < |\lambda_1^{c}| = |\phi_{c}| = 1$. The first component in \eqref{EQ:v_n-growth} therefore gives
\[
\lim_{n \to \infty} (\lambda_1^{c})^{-n} Z_{2^n}^{c} = \alpha u_1^{c} > 0.
\]
Since we may write
\[
\Theta_{2^n}^{c} = \sum_{k = 0}^{n-1} Z_{2^k}^{c},
\]
the fact that $Z_{2^k}^{c}$ grows like a constant multiple of $(\lambda_1^c)^k$
immediately implies  
\[
D_2^{c} = \lim_{ n \to \infty} \frac{\log \Theta_{2^n}^{c}}{n \log 2} = \frac{\log \lambda_1^{c}}{\log 2}.
\]
Since $\lambda_1^{c}$ is bounded away from $0$, it follows that $D_2^{c}$ is an analytic function of $\lambda_1^{c}$ and hence of $c$. Since both $c$ and $D_2^{c}$ take only real values, $D_2^{c}$ is in fact real-analytic.
\end{proof}

\section{Properties regarding the topological and variational pressure}\label{sec: pressure}

In this section we will prove Theorem \ref{thm: Ptop=Pvar}. We start with the case $c=0$, proving \eqref{EQ:var-P-at-0}. As pointed out in Section \ref{subsec: c=0}, this case is substantially different and different methods of proof are necessary.

\begin{prop}
\label{PROP:top-P-at-0}
We have $\mc P_{\operatorname{top}} (t \psi^0) = \max\{ (1-2t)\log 2, 0 \}$ for all $t \in \R$.
\end{prop}

\begin{proof}
From the coboundary relation \eqref{EQ:coboundary}, we easily get the explicit expression
\[
\psi^0_n(x) = - 2n \log 2 + 2 \log |\sin(2^n \pi x)| - 2 \log|\sin(\pi x)|,
\]
for $x \neq 0$ and fixed $n \in \N$. By symmetry of $\psi^0$, we can restrict our attention to $x \in (0,1/2]$. 
Let $0 \leqslant m < 2^{n-1}$ and $w^{n,m}$ be the word corresponding to the $m$th interval of length $2^{-n}$, that is, assume $\langle w^{n,m} \rangle = 2^{-n}[m, m+ 1]$. Let $x_m = 2^{-n}(m + 1/2)$ be the middle point of $\langle w^{n,m} \rangle$. 
At these points, the expression $|\sin(2^n \pi x)|$ assumes its maximum value $1$ on $\langle w^{n,m} \rangle$, and we have
\[
\psi_n^0(x_m) = - 2 n \log 2 - 2 \log |\sin(\pi x_m)|.
\]
On the other hand, since $2x \leqslant | \sin(\pi x ) | < \pi x$ on $(0,1/2]$ we get for $x \in \langle w^{n,m} \rangle$ and $m\geqslant 1$,
\[
 -2 n \log 2 - 2 \log |\sin(\pi x)| \in (-2 \log(\pi(m+1)), - 2 \log  (2m)].  
\]
Hence, up to an additive constant, $\sup_{x \in \langle w^{n,m} \rangle} \psi^0_n(x)$ evaluates to $- 2 \log m$. For $m =0$, we get $\sup_{x \in \langle w^{n, 0} \rangle} \psi_n^0(x) = \psi_n^0(0) = 0$.  Using these relations, we obtain for the upper topological pressure
\begin{equation}
\label{EQ:P-top-0}
\overline{\mc P}_{\operatorname{top}}(t \psi^0) = \limsup_{n \to \infty} \frac{1}{n} \log \sum_{m=1}^{2^{n-1}} m^{-2t}.
\end{equation}
Using an integral estimate, we get for $t = 1/2$,
\[
\sum_{m=1}^{2^{n-1}} m^{-1} \sim \int_{[1,2^{n-1}]}x^{-1} \dd x = (n-1) \log 2,
\]
where we write $a_n\sim b_n$ if $\lim_{n\to\infty} a_n/b_n=1$. 
Hence, $\mc P_{\operatorname{top}}( \psi^0/2) = 0$. For $t \neq 1/2$, we get
\[
\sum_{m = 1}^{2^{n-1}} m^{-2t} \sim \int_{[1,2^{n-1}]} x^{-2t} \dd x = \frac{1}{1-2t} \left( 2^{(1-2t)(n-1)} - 1\right).
\]
Plugging this into \eqref{EQ:P-top-0}, and distinguishing the cases for $t<1/2$ and $t>1/2$, we note that the limsup in \eqref{EQ:P-top-0} is indeed a limit and we finally get $\mc P_{\operatorname{top}}(t \psi^0) = \max \{ (1-2t)\log 2,0 \}$, as claimed.
\end{proof}

\begin{remark}
Before proceeding with the proof of Theorem~\ref{thm: Ptop=Pvar} a comment regarding the definition of the topological pressure is in place. While for H\"older continuous potentials the value in \eqref{eq: def top pressure} and the line below doesn't change if we consider the supremum or the infimum of the sum $\psi_n^c$, here the use of the supremum (infimum) for $t>0$ ($t<0$ respectively) is paramount as can be seen as follows. Since $\psi^c$ is unbounded on the negative side, there exists $x\in \Sigma^n$ such that $\inf_{x\in\langle w\rangle} \psi_n^c(x)=-\infty$ making the definition trivial. However, as shown in \cite{fan_generalized_2020} the set of points $c$ for which the restricted variational pressure as given in \eqref{eq: rest var press} equals $-\infty$ is a set of zero Hausdorff dimension. Hence, to obtain the statement of Theorem~\ref{thm: Ptop=Pvar}, indeed the choice of the infimum and supremum matters. 
\end{remark}

The next step in proving Theorem~\ref{thm: Ptop=Pvar} is to associate the restricted variational pressure with a restricted topological pressure. 
For a closed subshift $\X'\subseteq\X$ that is invariant under the
left shift, let us define the set of admissible $n$-words $\Sigma'^n=\left\{\omega\in\Sigma^n\colon \langle\omega\rangle \cap \X'\neq \varnothing\right\}$ and thus define the restricted topological pressure by
\[
 \cP_{\operatorname{top}}\left(t\ts \psi^c \ts | \ts \X'\right) \, := 
  \lim_{n\to\infty}\myfrac{1}{n}\log\sum_{\omega\in\Sigma'^{n}}
   \exp
  \Bigl(t\ts \sup_{x\in\langle\omega\rangle\cap\X'}\psi^{c}_{n}(x)\Bigr) .
\]
 By definition,
$\cP_{\operatorname{top}}\left(t\ts \psi^c\right)= \cP_{\operatorname{top}}\left(t\ts \psi^c \ts | \ts \X\right)$ (assuming the limit exits, otherwise analogous statements hold for the upper and lower (restricted) pressure).

Furthermore, we use the definition from \cite{fan_generalized_2020} for the restricted variational pressure 
\begin{align}
 p_{\operatorname{var}}(t; \delta) \coloneqq \sup \left\{h_{\nu} + t\int \psi^c\,\mathrm{d}\nu\colon \nu\in \mathcal{M}_{T}\left(\delta\right)\right\},\label{eq: rest var press}
\end{align}
where 
\begin{align*}
 \mathcal{M}_T\left(\delta\right) = \left\{\nu \in \mathcal{M}_T\colon \supp(\nu) \subset K_{\delta}\right\}
\end{align*}
and 
\begin{equation}
 K_\delta := K^c_\delta := \left\{ x \in \TT: \rho(T^n x, \breve{c} ) \geqslant \delta, \, \forall n \geqslant 0 \right\},\label{eq: def Kdelta}
 \end{equation}
 for $\delta>0$. 
For the following let $\breve{c}^{\ast}$ to be the alternative representation of $\breve{c}$ in case $\breve{c}$ is a dyadic point and has thus two dyadic representations. Otherwise, formally, we simply identify $\breve{c}$ and $\breve{c}^{\ast}$. 
Here we give a refinement of \cite[Prop.~6.2]{fan_generalized_2020}, using
the subshift
 \begin{align*}
  \mathbb{X}_m^c\coloneqq \left\{x \in\mathbb{X}\colon x_{[\ell,\ell+m]} \notin \{ \breve{c}_{[1,m+1],\breve{c}_{[1,m+1]}^{\ast}\}},\forall\ell \in\mathbb{N}\right\}.
 \end{align*}
We also note that $\mathbb{X}_m^c$ is closed and invariant under the left shift.
It will be one of the main steps in this section to prove $\lim_{n\to\infty}\cP_{\operatorname{top}}(t \psi^c \ts | \ts \X_m^c) = \cP_{\operatorname{top}}(t \psi^c)$, for all $t\in \mathbb{R}$ (for $t\geqslant 0$, by definition, we already have 
$\cP_{\operatorname{top}}(t \psi^c \ts | \ts \X_m^c) \leqslant \underline{\cP}_{\operatorname{top}}(t \psi^c)$). 
We will prove the cases $t<0$ and $t\geqslant 0$ separately in Propositions \ref{prop: pm to p} and \ref{prop: PtopmtoPtop}. However, before we start by showing that $\X_m^c$ is an irreducible, aperiodic Markov subsystem not containing $\breve{c}$.

\begin{prop}\label{prop: Xm Kdelta}
 Let $c\neq 0$. For every $\delta > 0$ there exists an irreducible aperiodic Markov subsystem $U = U_{\delta, c}\colon \bigcap_{j} J_j \to \mathbb{X}$ whose maximal invariant set contains $K_{\delta}^{c}$, and does not contain $\breve{c}$.
 This system can be chosen as the subshift $\mathbb{X}_m^{c}$ for $m$ sufficiently large. 
\end{prop}
Before giving the proof we introduce the flip of a digit $a$ by $\underline{a}$, i.e.\ we write $\underline{a}=1$ if $a=0$ and $\underline{a}=0$ if $a=1$.

\begin{proof}
 We remember the identification of $[0,1]$ with the shift space $\mathbb{X}$ up to a countable number of points. In case that $\breve{c}$ is not a dyadic point it has a unique representation in $\mathbb{X}$.
 If for such $\breve{c}$ we choose $m$ large enough we have that $K_{\delta}^{c}\cap \langle\breve{c}_{[1,m+1]}\rangle=\emptyset$. If $\breve{c}$ is a dyadic point we have 
 $K_{\delta}^c\cap \Big(\langle\breve{c}_{[1,m+1]}\rangle \cup \langle\breve{c}_{[1,m+1]}^{\ast}\rangle\Big)=\emptyset$.

We may identify $\mathbb{X}_m^{c}$ with a space of $2^{m+1}-u$ states where $u$ is the number of forbidden words.
More precisely, our alphabet for such a space contains all the $m+1$ letter words of $\mathbb{X}_m^{c}$, i.e.\ all words $\{0,1\}^{m+1}$ except $\breve{c}_1\ldots \breve{c}_{m+1}$ possibly with the two different ways of representation if $\breve{c}$ is dyadically rational. 
We will see that for $m$ sufficiently large this space is an irreducible and aperiodic Markov chain. 
To show aperiodicity, let us first assume that $\breve{c}\neq 0$. In that case, for $m$ sufficiently large, $0^{m+1}$ is an admissible word which can be reached within one step from itself giving aperiodicity immediately. 
For the following, for a finite word $u$ let $\overline{u}$ denote the infinite word with infinitely many concatenations of $u$. 
If $\breve{c}=0$, then $\overline{01}_{[1,m+1]}$ and $\overline{001}_{[1,m+1]}$ are both admissible words. 
Starting with the word $\overline{01}_{[1,m+1]}$, it can be reached within two steps going via the word $\overline{10}_{[1,m+1]}$. On the other hand, it is possible to return to the word $\overline{001}_{[1,m+1]}$ within three steps going via the words $\overline{010}_{[1,m+1]}$ and $\overline{100}_{[1,m+1]}$. As two and three are relatively prime, the corresponding Markov matrix has to be aperiodic.

We show irreducibility directly by showing that all states ($m+1$ letter words in $\mathbb{X}_m^{c}$) - except for the forbidden ones - can be reached from any other state ($m+1$ letter words in $\mathbb{X}_m^{c}$). This can be done in the following way:
For the moment we assume that $\breve{c}$ has a unique expansion.
Let the first state be given as $u=u_1\ldots u_{m+1}$ and the second as $v=v_1\ldots v_{m+1}$. 

If $(uv)_{[\ell, \ell+m]}\neq \breve{c}_{[1,m+1]}$ for all $\ell=1,\ldots, m+1$, i.e.\ the forbidden  word does not occur by concatenating $u_1\ldots u_{m+1}v_1\ldots v_{m+1}$ we are done. Let us for the following assume that this is not the case, then we can proceed in the following way: First we assume that $m$ is sufficiently large such that $\breve{c}_{[1,m+1]}\notin\{ 0^{\ell}1^{m+1-\ell}, 1^{\ell}0^{m+1-\ell}\}$, for $\ell=0,\ldots, m+1$.
Since we excluded the words with a dyadically rational expansion for the moment, it is always possible to choose $m$ sufficiently large such that $\breve{c}_{[1,m+1]}$ has a different expansion than the ones just mentioned. Then we might consider $w= u  \breve{\underline{c}}_{m+1}^{m+1} \breve{\underline{c}}_1^{m+1} v$ 
and $w_{[\ell,\ell+m]}\neq \breve{c}_{[1,m+1]}$, for all $\ell$.

For all the following considerations we note that 
$$(uy_1\ldots y_iv)_{[1,m+1]}=u \neq  \breve{c}_{[1,m+1]}\text{ and }(uy_1\ldots y_iv)_{[m+2+i,2m+2+i]}=v\neq  \breve{c}_{[1,m+1]}$$ 
by the assumption on $u$ and $v$; we will not separately consider these cases.

We assume now that $\breve{c}$ is dyadically rational and has thus two representations. First we consider the case 
$\breve{c}=0^{\infty}=1^{\infty}$. 
In case there exists $\ell\in \{2,\ldots, m+1\}$ such that 
$(uv)_{[\ell,\ell+m]}=\breve{c}_{[1,m+1]}$ we may consider 
$(u\underline{u}_{m+1}v)$ instead and we have 
$(u\underline{u}_{m+1}v)_{[m+1,m+3]}=u_{m+1}\underline{u}_{m+1}u_{m+1}$ and for all $\ell\in \{2,\ldots, m+2\}$ we have $(u\underline{u}_{m+1}v)_{[\ell,\ell+m]}\neq\breve{c}_{[1,m+1]}$.

Next, we assume that $\breve{c}$ has the representations 
$\breve{c}_1\ldots\breve{c}_k10^{\infty}$ and 
$\breve{c}_1\ldots\breve{c}_k01^{\infty}$.
In the following we will assume that $m\geqslant 4k$ and for simplicity we set 
$\breve{c}_{[1,m+1]}=\breve{c}_1\ldots\breve{c}_k10^{m-k}$
and $\breve{c}^\ast_{[1,m+1]}=\breve{c}_1\ldots\breve{c}_k01^{m-k}$.
Again, let the first state be given as $u=u_1\ldots u_{m+1}$ and the second as $v=v_1\ldots v_{m+1}$
and assume the forbidden word $\breve{c}_{[1,m+1]}$
occurs by concatenating $uv$, then we distinguish some cases at which position $\ell$ of the concatenated word $uv$ the forbidden word appears, i.e.\ we distinguish for which $\ell$ we have $(uv)_{[\ell,\ell+m]}=\breve{c}_{[1,m+1]}$.
We will not treat the case $(uv)_{[\ell,\ell+m]}=\breve{c}^\ast_{[1,m+1]}$ as the argumentation is analogous in all cases.

$(uv)_{[\ell,\ell+m]}=\breve{c}_{[1,m+1]}$ for some $\ell\in \{m+2-k,\ldots,m+1\}$. 

\begin{adjustwidth}{.8cm}{0cm}
Then we may consider $u\underline{u}_{m+1}\breve{\underline{c}}_{m-\ell+2}v$ instead.
By having the subword $u_{m+1}\underline{u}_{m+1}\in \{01,10\}$ it is clear that for all $j\in \{2,\ldots, m-k\}$ we have 
$(u\underline{u}_{m+1}\breve{\underline{c}}_{m-\ell+2}v)_{[k+j,m+j]}\notin\{10^{m-k}, 01^{m-k}\}.$
Knowing that 
$$(uv)_{[\ell+k,\ell+k+1]}=(u\underline{u}_{m+1}\breve{\underline{c}}_{m-\ell+2}v)_{[\ell+k+2,\ell+k+3]}=10$$
 we can also conclude that for all $j\in \{\ell+k+3-m,\ldots,\ell+1\}$ we have 
$$(u\underline{u}_{m+1}\breve{\underline{c}}_{m-\ell+2}v)_{[k+j,m+j]}\notin\{10^{m-k}, 01^{m-k}\}.$$ 
Since we know that $\ell<m+2$, we can also restrict the range in this case to $j\in \{k+4,\ldots, \ell+1\}\subset\{m-k+1,\ldots, \ell+1\}$. 

Moreover, knowing that 
$$(uv)_{[\ell+k+1,\ell+m]}=(u\underline{u}_{m+1}\breve{\underline{c}}_{m-\ell+2}v)_{[\ell+k+3,\ell+m+2]}=0^{m-k}$$
we can also conclude that 
$$(u\underline{u}_{m+1}\breve{\underline{c}}_{m-\ell+2}v)_{[k+j,m+j]}\notin\{10^{m-k}, 01^{m-k}\}$$
holds for all $j\in \{\ell+3,\ldots, m+4\}$.
Combining the just made observations we can conclude that 
$$(u\underline{u}_{m+1}\breve{\underline{c}}_{m-\ell+2}v)_{[j',j'+m]}\notin \{\breve{c}_{[1,m+1]}, \breve{c}^\ast_{[1,m+1]}\},$$ for all 
$j'\in \{1,\ldots, \ell+1 \} \cup \{\ell+3,\ldots, m+4\}$.

For the only excluded case $j'= \ell+2$ we have 
$$(u\underline{u}_{m+1}\breve{\underline{c}}_{m-\ell+2}v)_{\ell+2+(m+1-\ell)}=(u\underline{u}_{m+1}\breve{\underline{c}}_{m-\ell+2}v)_{m+3}
=\breve{\underline{c}}_{m-\ell+2}$$
implying 
$(u\underline{u}_{m+1}\breve{\underline{c}}_{m-\ell+2}v)_{[\ell+2,\ell+2+m]}\notin \{ \breve{c}_{[1,m+1]}, \breve{c}^\ast_{[1,m+1]}\}$
finishing this case.
\end{adjustwidth}

 $(uv)_{[m+1-k, 2m+1-k]}=\breve{c}_{[1,m+1]}$. 

 \begin{adjustwidth}{.8cm}{0cm}
Here, we might consider the word 
$u\breve{\underline{c}}_k1v$ instead. 

For this word we have $(u\breve{\underline{c}}_k1v)_{[m+4,2m+4-k]}=v_{[1,m-k]}=0^{m-k}$. Hence, we can conclude
$(u\breve{\underline{c}}_k1v)_{[k+j,j+m]}\notin\{01^{m-k}, 10^{m-k}\}$, for all $j\geqslant m-k+4$.

Moreover, we know that 
$(u\breve{\underline{c}}_k1v)_{[m,m+4]}=u_mu_{m+1}\breve{\underline{c}}_k1v_1=\breve{c}_k1\breve{\underline{c}}_k10\in \{ 01110, 11010\}$ and thus 
$(u\breve{\underline{c}}_k1v)_{[k+j,j+m]}\notin \{01^{m-k}, 10^{m-k}\}$, for all $j\in \{2,\ldots, m-k+2\}$.

Together the last considerations imply that 
$(u\breve{\underline{c}}_k1v)_{[j',j'+m]}\notin \{\breve{c}_{[1,m+1]}, \breve{c}^\ast_{[1,m+1]}\}$, 
for all 
$j'\in \{2,\ldots, m-k+2\}\cup\{m-k+4,\ldots, m+3\}$.

For the last case $j'=m-k+3$ we have 
$$(u\breve{\underline{c}}_k1v)_{m-k+3+(k-1)}=(u\breve{\underline{c}}_k1v)_{m+2}
=\breve{\underline{c}}_{k}$$ 
and thus
$(u\breve{\underline{c}}_k1v)_{[m-k+3,2m-k+3]}\notin\{\breve{c}_{[1,m+1]}, \breve{c}^\ast_{[1,m+1]}\}$.
\end{adjustwidth}

 One of the following holds:
\begin{itemize}
 \item $(uv)_{[\ell,\ell+m]}=\breve{c}_{[1,m+1]}$
for $k+1< \ell \leqslant  m-k$;
\item  $(uv)_{[k+1,k+1+m]}=\breve{c}_{[1,m+1]}$ and at least one of $\breve{c}^\ast_{[1,m+1]}\neq 10^{k}1^{m-k}$
or $v_{[k+1,m]}\neq 1^{m-k}$ holds.
\end{itemize}

\begin{adjustwidth}{.8cm}{0cm}
In these cases we may consider $u\breve{\underline{c}}_k1v$ instead. 

For this word we have $(u\breve{\underline{c}}_k1v)_{[m+4,\ell+m+2]}=v_{[1,\ell-1]}=0^{\ell-1}$ and hence
$(u\breve{\underline{c}}_k1v)_{[j+k,j+m]}\notin\{01^{m-k}, 10^{m-k}\}$, for all $j= m-k+4, \ldots, m+2$ and additionally for $j=m+3$ in case $\ell> k+1$.
On the other hand, in case that $\ell=k+1$, we have that $(u\breve{\underline{c}}_k1v)_{[m+3,2m+2]}\notin \{\breve{c}_{[1,m+1]}, \breve{c}^\ast_{[1,m+1]}\}$ by our assumption on $\breve{c}^\ast_{[1,m+1]}$ and on $v$.

Moreover, we know that 
$(u\breve{\underline{c}}_k1v)_{[m+1,m+4]}=0\breve{\underline{c}}_k10$ and thus we can conclude that
$(u\breve{\underline{c}}_k1v)_{[k+j,m+j]}\notin \{01^{m-k}, 10^{m-k}\}$, for all $j\in \{3,\ldots, m-k+2\}$.

Furthermore, we have $(u\breve{\underline{c}}_k1v)_{[\ell+k,\ell+k+1]}=10$ and since $\ell+k  \leqslant  m$ this also implies $(u\breve{\underline{c}}_k1v)_{[k+2,m+2]}\notin \{01^{m-k}, 10^{m-k}\}$.

Together the last considerations imply that 
$(u\breve{\underline{c}}_k1v)_{[j',j'+m]}\notin \{\breve{c}_{[1,m+1]}, \breve{c}^\ast_{[1,m+1]}\}$, 
for all 
$j'\in \{1,\ldots, m-k+2\} \cup\{ m-k+4,\ldots, m+3\}$.

In order to consider the remaining case $j'=m-k+3$ we note that by construction $(u\breve{\underline{c}}_k1v)_{m-k+3+(k-1)}=(u\breve{\underline{c}}_k1v)_{m+2}=\breve{\underline{c}}_k\neq \breve{c}_k$ and thus also 
$(u\breve{\underline{c}}_k1v)_{[m-k+3,2m-k+3]}\notin \{\breve{c}_{[1,m+1]}, \breve{c}^\ast_{[1,m+1]}\}$.
\end{adjustwidth}

$(uv)_{[k+1,k+1+m]}=\breve{c}_{[1,m+1]}$ and $\breve{c}^\ast_{[1,m+1]}=10^{k}1^{m-k}$ and $v_{[k+1,m]}=1^{m-k}$.

\begin{adjustwidth}{.8cm}{0cm}
Then we may consider the word $(u101v)$ instead.
As $\breve{c}^\ast_{[1,m+1]}=10^{k}1^{m-k}$ implies 
$\breve{c}_{[1,m+1]}=10^{k-1}10^{m-k}$
the word can also be written as 
$$(u101v)=u_{[1,k]}10^{k-1}10^{m-2k}1010^k1^{m-k}v_{m+1}.$$
By a direct comparison we have that 
$(u101v)_{[j,j+m]}\notin \{\breve{c}_{[1,m+1]}, \breve{c}^\ast_{[1,m+1]}\}$, 
for all 
$j\in \{1,\ldots, m+5\}$.
\end{adjustwidth}

$(uv)_{[\ell,\ell+m]}=\breve{c}_{[1,m+1]}$ for 
 $1<\ell\leqslant k$.

\begin{adjustwidth}{.8cm}{0cm} 
In this case, we do a double construction; first we consider $u1v$. 

We have $(u1v)_{[m+1,m+3]}=010$ implying that 
$(u1v)_{[k+j,j+m]}\notin\{01^{m-k}, 10^{m-k}\}$, for all $j\in\{ 2, \ldots, m-k+1\}$. 
This implies that 
$(u1v)_{[i,i+m]}\notin \{\breve{c}_{[1,m+1]}, \breve{c}^\ast_{[1,m+1]}\}$, 
for all 
$i\in \{ 2, \ldots, m-k+1\}$.

If $(u1v)_{[i,i+m]}\notin\{\breve{c}_{[1,m+1]},\breve{c}^\ast_{[1,m+1]}\}$ also holds for all $i> m-k+1$, we are done. 
However, let us for the following assume the contrary, i.e.\ there exists $i>m-k+1$ such that 
$(u1v)_{[i,i+m]}\in\{\breve{c}_{[1,m+1]}, \breve{c}^\ast_{[1,m+1]}\}$. 

Here, we distinguish two cases: If $i=m+2$, we consider 
$u 10v$ instead.
In this case, we have $(u10v)_{[m+1,m+3]}=010$ implying 
$(u10v)_{[k+j,j+m]}\notin\{01^{m-k}, 10^{m-k}\}$, for all $j\in\{ 2, \ldots, m-k+1\}$. 

Moreover, we have 
$$(u10v)_{[m+3+k,2m+3]}=(u10v)_{[i+1+k,i+1+m]}=(u1v)_{[i+k,i+m]}\in \{01^{m-k},10^{m-k}\}$$ 
implying $(u10v)_{[j+k,j+m]}\notin \{01^{m-k},10^{m-k}\}$ for all $j\in \{k+3,\ldots, i\}\cup \{i+2\}$.

These together imply 
$(u10v)_{[j',j'+m]}\notin \{\breve{c}_{[1,m+1]}, \breve{c}^\ast_{[1,m+1]}\}$, 
for all 
$j'\in \{ 2, \ldots, i\}\cup \{i+2\}$.
For the last case $j'=i+1$ we note that 
$(u10v)_{i+1+ (m+2-i)}=(u10v)_{m+3}=0$. On the other hand, 
$(u1v)_{i+(m+2-i)}=(u1v)_{m+2}=1=\underline{\breve{c}}_{m+3-i}$ by construction and thus $(u10v)_{[i+1,i+1+m]}\notin \{\breve{c}_{[1,m+1]}, \breve{c}^\ast_{[1,m+1]}\}$.

If $i<m+2$ we consider the word $u 11v$.
Now, we have $(u11v)_{[m+1,m+4]}=0110$ implying that 
$(u11v)_{[k+j,j+m]}\notin\{01^{m-k}, 10^{m-k}\}$, for all $j\in\{ 3, \ldots, m-k+2\}$. 
The case $j=2$ is also clear as $(u11v)_{[k+\ell,k+\ell+1]}=10$
implying $(u11v)_{[k+2,2+m]}\notin\{01^{m-k}, 10^{m-k}\}$.

Moreover, we have $(u11v)_{[i+1+k,i+1+m]}=(u1v)_{[i+k,i+m]}\in \{01^{m-k},10^{m-k}\}$ implying $(u11v)_{[j+k,j+m]}\notin \{01^{m-k},10^{m-k}\}$ for all $j\in \{i+2+k-m,\ldots, i\}\cup \{i+2,\ldots, m+4\}$.

These together - noting that $i+2+k-m<m+k+4$ - imply 
$(u11v)_{[j',j'+m]}\notin \{\breve{c}_{[1,m+1]}, \breve{c}^\ast_{[1,m+1]}\}$, 
for all 
$j'\in \{ 2, \ldots, i\}\cup \{i+2,\ldots, m+4\}$.
For the last case $j'=i+1$ we note that 
$(u11v)_{i+1+(m+1-i)}=(u11v)_{m+2}=1$. On the other hand, 
$(u1v)_{i+(m+1-i)}=(u1v)_{m+1}=0=\underline{\breve{c}}_{m+2-i}$ by construction and thus $(u11v)_{[i,i+m]}\notin \{\breve{c}_{[1,m+1]}, \breve{c}^\ast_{[1,m+1]}\}$.
\end{adjustwidth}
\end{proof}

The next proposition provides us with a certain  exhaustion principle in the sense of \cite{MR3190211}.
\begin{prop}\label{prop: pm to p}
Let $c \neq 0$. For all $t<0$ we have that 
 \begin{align}
  \lim_{m\to\infty}\mathcal{P}_{\operatorname{top}}\left(t\psi^{c}\lvert\X_m^{c}\right)=\mathcal{P}_{\operatorname{top}}\left(t\psi^{c}\right). \label{eq: P rest= P unrest}
 \end{align}
 In particular, $\mathcal{P}_{\operatorname{top}}\left(t\psi^{c}\right)$ exists as a limit.
\end{prop}

\begin{proof}
The classical Thue{\ts}--Morse case $c=1/2$, i.e.\ $\breve{c}=0$, plays a special role and we start by considering this case.
 Our aim is to show that for $c=1/2$ both sides of \eqref{eq: P rest= P unrest} equal $\infty$. (In \cite{BGKS} a similar  convergence statement has already been shown. However, in there the definition of the topological pressure for $t<0$ was different, why, for completeness, we give a self-contained proof for this case here.)
 
 We first note that $\psi^c(x)
\leqslant 2 \log (\pi \rho(x, \breve{c}))$ - for a proof of that statement, see Lemma~\ref{LEM:psi-distance-bound} below.
 In order to bound the LHS from above we obtain for $\omega_n=\overline{0^m1^m}_{[1,n]}$ (again with $\overline{u}$ the concatenation of infinitely many copies of $u$) that
\begin{align*}
 \sup_{x\in\langle\omega_n\rangle\cap\X_m}\psi^{c}_{n}(x)
 &\leqslant 2n\log \pi+ 2\left\lfloor \frac{n}{m} \right\rfloor  \sum_{k=1}^m \log (2^{-k})
 =2n\log \pi -\left\lfloor \frac{n}{m} \right\rfloor m(m+1)\log 2\\
 &\leqslant -\frac{ n(m+1)}{2}+m^2, 
\end{align*}
for $m$ sufficiently large. 
Hence, remembering that $t<0$ we have
\begin{align*}
 \liminf_{m\to\infty} \mathcal{P}_{\operatorname{top}}\left(t\psi^{c}\lvert\X_m^{c}\right)
 &\geqslant  \liminf_{m\to\infty} \lim_{n\to\infty}\myfrac{t\ts \sup_{x\in\langle\omega_n\rangle\cap\X_m}\psi^{c}_{n}(x)}{n}
\\ &\geqslant  \liminf_{m\to\infty}\lim_{n\to\infty}t\left( - \myfrac{m+1}{2}+\myfrac{m^2}{n}\right)=\infty. 
\end{align*}
In particular, the limit of this expression as $m \to \infty$ exists. Moreover, as $\psi^c$ is H\"older continuous on $\X_m^c$, we can directly argue with $\mathcal{P}_{\operatorname{top}}$ not needing to consider the upper or lower pressure.
On the other hand, using again $\psi^c(x)
\leqslant 2 \log (\pi \rho(x, \breve{c}))$ we have for $\omega_n=0^n$ that 
\begin{align*}
  \sup_{x\in\langle\omega_n\rangle}\psi^{c}_{n}(x)
 &\leqslant 2\sum_{k=1}^n\log(\pi 2^{-k})
 = n(2  \log \pi -  (n+1)\log 2)
\end{align*}
and thus
\begin{align*}
\overline{\mathcal{P}}_{\operatorname{top}}\left(t\psi^{c}\right)
 \geqslant \underline{\mathcal{P}}_{\operatorname{top}}\left(t\psi^{c}\right)
 &\geqslant \lim_{n\to\infty}\myfrac{t\ts \sup_{x\in\langle\omega_n\rangle}\psi^{c}_{n}(x)}{n}
 \geqslant \lim_{n\to\infty}t(2  \log \pi - (n+1)\log 2)=\infty,
\end{align*}
proving the proposition for $c=1/2$.

Next, we prove the general case $c\notin\{0,1/2\}$. For the following we set 
\begin{align}
 \Sigma_{m}^{n} =\Sigma_m^n(c)=\left\{\omega\in\Sigma^n\colon \nexists\ell\in \{1,\ldots, n-m\}\colon \omega_{[\ell,\ell+m]}=\breve{c}_{[1,m+1]}\text{ or }\omega_{[\ell,\ell+m]}=\breve{c}^{\ast}_{[1,m+1]}\right\}.\label{eq: def Sigma nm}
\end{align}
In words, $\Sigma_{m}^{n}$ denotes the set of $n$-letter words which do not have an $m+1$-letter subword coinciding with $\breve{c}_{[1,m+1]}$ (or $\breve{c}_{[1,m+1}^{\ast}$) and thus $\Sigma_{m}^{n}=\{\omega\in\Sigma^n\colon \langle\omega\rangle \cap \X_m\neq \varnothing\}$. For $m\geqslant n$ we understand $\Sigma_m^n$ to be $\Sigma^n$.

Furthermore, we set 
\begin{align*}
 \widetilde{\Sigma}_{m}^{n}
 =\widetilde{\Sigma}_{m}^{n}(c) =\left\{\omega\in\Sigma^n\colon \right.&\nexists\ell\in \{1,\ldots, n-m\}\colon \omega_{[\ell, \ell+m]}\in \big\{\breve{c}_{[1,m+1]},\breve{c}^{\ast}_{[1,m+1]}\big\},\\
 & \hspace{-0.8mm}\left.\nexists i\in \{1, \ldots, m+1 \} \colon \omega_{[n-i+1,n]}\in \big\{\breve{c}_{[1,i]},\breve{c}^{\ast}_{[1,i]}\big\}
 \right\}.
\end{align*}
Here we have additionally excluded words $u \in \Sigma^n_m$ such that a suffix of $u$ coincides with a prefix of the forbidden word(s). This construction ensures that the concatenation of words in $\widetilde \Sigma^{n_1}_m$ and $\widetilde \Sigma^{n_2}_m$ still does not contain a forbidden word. More precisely, 
\begin{equation}
\label{EQ:concat}
\widetilde \Sigma^{n_1}_m \widetilde \Sigma^{n_2}_m \subset \widetilde \Sigma^{n_1 + n_2}_m,
\end{equation}
for all $n_1, n_2 ,m \in \N$. This property will be important for establishing monotonicity in a suitable approximation of $\mathcal{P}_{\operatorname{top}}(t \psi^c \vert \mathbb{X}_m^{c})$ if $t < 0$.

We start by showing an alternative expression for $\mathcal{P}_{\operatorname{top}}(t \psi^{c} \vert \mathbb{X}_m^{c})$, i.e.\ we first show that for
large enough $m\in\mathbb{N}$ and $t< 0$ we have that 
\begin{align}
 \lim_{n\to\infty}\frac{1}{n} \log \sum_{\omega\in \Sigma_m^{n}}\sup_{x\in  \langle\omega\rangle\cap\mathbb{X}_m^{c}}\exp\left(t \psi_{n}^{c}(x)\right)
 = \lim_{n\to\infty}\frac{1}{n} \log \sum_{ \omega\in \Sigma_{m}^{n}}\inf_{x\in  \langle\omega\rangle}\exp\left(t \psi_{n}^{c}(x)\right).\label{eq: lim w wo Xm1}
\end{align}
We will use a similar argumentation as in \cite[Proof of Lem.~4.1]{BGKS} and first note that 
there exists $K>0$ such that on $[0,1]\setminus C_m(\breve{c})$ we have $|\psi'|\leqslant K/2$. 
Then 
for all $\omega\in\Sigma_m^n$ we have 
\begin{align*}
 \sup_{x\in \langle\omega\rangle}\psi_n^{c}(x)- \inf_{y\in \langle\omega\rangle\cap\mathbb{X}_m^{c}}\psi_n^{c}(y)
 &\leqslant \sum_{k=1}^n \left(\sup_{x\in \langle S^{k-1}\omega\rangle}\psi^{c}(x)- \inf_{y\in \langle S^{k-1}\omega\rangle\cap\mathbb{X}_m^{c}}\psi^{c}(y)  \right)\notag\\
 &\leqslant  \sum_{k=1}^n K 2^{-(n-k)-1}\leqslant K.
\end{align*}
This follows from the fact that the supremum over $\langle S^{k-1}\omega\rangle$ is always attained outside $C_m(\breve{c})$.
Thus, for $t<0$ we obtain \eqref{eq: lim w wo Xm1}.
Since this expression is smaller than the (lower) topological pressure of $t \psi^c$, we obtain that
$
\mathcal{P}_{\operatorname{top}}(t \psi^{c} \vert \mathbb{X}_m^{c}) 
\leqslant \underline{\mathcal{P}}_{\operatorname{top}}(t \psi^{c}) 
$
for all $m \in \N$, and hence that 
\begin{equation}
\label{EQ:pressure-upper-estimate}
\limsup_{m \to \infty} \mathcal{P}_{\operatorname{top}}(t \psi^{c} \vert \mathbb{X}_m^{c})  \leqslant \underline{\mathcal{P}}_{\operatorname{top}}(t \psi^{c}).
\end{equation}

For $t<0$ we define the functions $\widetilde a , a\colon \left(\mathbb{N}\cup\{\infty\}\right)^2\to\mathbb{R}$ by 
 \begin{align*}
  \widetilde a\left(n,m\right)\coloneqq 
  \frac{1}{n}\, \log \sum_{\omega\in \widetilde \Sigma_{m}^{n}}\inf_{x\in  \langle\omega\rangle}\exp\left(t \psi_{n}^{c}(x)\right)
 \end{align*}
 and its acceleration $a(n,m) = \widetilde a (2^n,m)$.
 For $m=\infty$ we understand the above term as
 $$\widetilde \Sigma_{\infty}^n = \left\{\omega\in\Sigma^n\colon \nexists i\in \{1, \ldots, n \} \colon \omega_{[n-i+1,n]}\in\big\{\breve{c}_{[1,i]},\breve{c}^{\ast}_{[1,i]}\big\}
 \right\}.$$

 We notice that $a(n, \cdot)$ is monotonically increasing since for $m$ increasing we just add more summands.
  Moreover, $n \cdot \widetilde{a}(n, m)$ is superadditive in $n$ -- this can be derived in a straightforward manner, using \eqref{EQ:concat}.  
  Thus, 
  $$2^{n+1} a(n+1,m)= 2^{n+1} \widetilde{a}(2^{n+1},m)\geqslant 2^{n} (\widetilde{a}(2^n,m)+\widetilde{a}(2^n,m))= 2^{n+1} a(n,m)$$ 
  and thus the function $a(\cdot, m)$ is monotonically increasing and hence, $a(\infty, m)$ is well-defined by taking the limit $n\to\infty$.
  
Hence, for $t < 0$ the function $a(\cdot,\cdot)$ is monotonically increasing in both variables. This allows us to conclude
\begin{align}
  \liminf_{m \to \infty} \mathcal{P}_{\operatorname{top}}(t\psi^{c} \vert \mathbb{X}_m^{c})
&\geqslant \lim_{m \to \infty} \lim_{n \to \infty} a(n,m)
=\lim_{n \to \infty} \lim_{m \to \infty} a(n,m)\notag\\
&=\lim_{n \to \infty} \frac{1}{2^n}\, \log \sum_{\omega\in \widetilde \Sigma_{\infty}^{2^n}}\inf_{x\in  \langle\omega\rangle}\exp\left(t \psi_{2^n}^{c}(x)\right)\notag\\
&=\lim_{n \to \infty} \frac{1}{n}\, \log \sum_{\omega\in \widetilde \Sigma_{\infty}^{n}}\inf_{x\in  \langle\omega\rangle}\exp\left(t \psi_{n}^{c}(x)\right),\label{eq: lim m Ptop}
\end{align}
where the last limit exists by superadditivity.
It remains to show that 
\begin{align*}
 \lim_{n \to \infty} \frac{1}{n}\, \log \sum_{\omega\in \widetilde \Sigma_{\infty}^{n}}\inf_{x\in  \langle\omega\rangle}\exp\left(t \psi_{n}^{c}(x)\right)
  \geqslant \limsup_{n \to \infty} \frac{1}{n}\, \log \sum_{\omega\in \Sigma^{n}}\inf_{x\in  \langle\omega\rangle}\exp\left(t \psi_{n}^{c}(x)\right).
\end{align*}

In order to proceed we will give the following lemma. Its proof will be postponed to Section~\ref{sec:Technical}.
\begin{lemma}
\label{LEM:sliding-out-no-prefix}
Assume $c \notin \{0,1/2\}$. Then, for every $n\in\N$, there exists an $\ell_n \in \N$ with the following property. For all $v \in \Sigma^n$, there is $w = v u \in \Sigma^{n + \ell_n}$ such that $w_{[k,|w|]}$ is not a prefix of $\breve{c}$ for all $1\leqslant k \leqslant |w|$. For large enough $n$, we can choose $\ell_n\leqslant 5 \log_2 \log_{3/2} n$.
\end{lemma}

By Lemma \ref{LEM:sliding-out-no-prefix} 
for each $\omega\in \Sigma^n$ there exists a word $u$ with 
$|u|= \ell_n$ and $\ell_n\leqslant \lfloor  5 \log_2 \log_{3/2} n\rfloor$
such that no suffix of $wu$ is a prefix of $\breve{c}$. For the following we denote such a word by $u_{\omega}$. In case there are more than one such word, we choose one. 
Since $\psi^c\leqslant 0$ and thus $t\psi^c\geqslant 0$ if $t<0$, we have by \eqref{eq: lim m Ptop}
\begin{align*}
  \limsup_{n \to \infty} \frac{1}{n}\, \log \sum_{\omega\in \Sigma^{n}}\inf_{x\in  \langle\omega\rangle}\exp\left(t \psi_{n}^{c}(x)\right)
  &\leqslant \limsup_{n \to \infty} \frac{1}{n}\, \log \sum_{\omega\in \Sigma^{n}}\inf_{x\in  \langle\omega u_{\omega}\rangle}\exp\left(t \psi_{n}^{c}(x)\right) \\
  &\leqslant \limsup_{n \to \infty} \frac{1}{n}\, \log \sum_{\omega\in \Sigma^{n}}\inf_{x\in  \langle\omega u_{\omega}\rangle}\exp\left(t \psi_{n+\ell_n}^{c}(x)\right) \\
  &\leqslant \limsup_{n \to \infty} \frac{1}{n}\, \log \sum_{\omega\in \widetilde{\Sigma}_{\infty}^{n+\ell_n}}\inf_{x\in  \langle\omega \rangle}\exp\left(t \psi_{n+\ell_n}^{c}(x)\right) \\
  &= \limsup_{n \to \infty} \frac{1}{n+\ell_n}\, \log \sum_{\omega\in \widetilde{\Sigma}_{\infty}^{n+\ell_n}}\inf_{x\in  \langle\omega \rangle}\exp\left(t \psi_{n+\ell_n}^{c}(x)\right)\\
  &= \limsup_{n \to \infty} \frac{1}{n}\, \log \sum_{\omega\in \widetilde{\Sigma}_{\infty}^{n}}\inf_{x\in  \langle\omega \rangle}\exp\left(t \psi_{n}^{c}(x)\right)
  \\ &\leqslant \liminf_{m \to \infty} \mc P_{\operatorname{top}}(t \psi^c \vert \mathbb{X}_m^c).
\end{align*}
Combining this with \eqref{EQ:pressure-upper-estimate}, we obtain
\[
\overline{\mc P}_{\operatorname{top}}(t \psi^c)
\leqslant \liminf_{m \to \infty} \mc P_{\operatorname{top}}(t \psi^c | \mathbb{X}_m^c)
\leqslant \limsup_{m \to \infty} \mc P_{\operatorname{top}}(t \psi^c | \mathbb{X}_m^c)
\leqslant \underline{\mc P}_{\operatorname{top}}(t \psi^c),
\]
proving the statement of the proposition. 
\end{proof}

In the next steps we will prove an analog to Proposition \ref{prop: pm to p} for $t\geqslant 0$.
\begin{prop}\label{prop: PtopmtoPtop}
Let $c \neq 0$. For all $t \geqslant 0$, we have that
\[
\lim_{m \to \infty} \mc P_{\operatorname{top}} (t \psi^c | \mathbb{X}^c_m) = \mc P_{\operatorname{top}}(t \psi^c).
\]
\end{prop}
We note that here, the existence of $\mc P_{\operatorname{top}}$ already follows by subadditivity.
An important part of the proof of Proposition \ref{prop: PtopmtoPtop} is the following Proposition \ref{PROP:word-regularisation}. Its proof is rather technical and we postpone it to Subsection \ref{subsec: proof word regularisation} after we have collected all the necessary results in Section \ref{sec:Technical}.

\begin{prop}
\label{PROP:word-regularisation}
Let $\breve{c} \neq 1/2$ and $\varepsilon > 0$ and let $\Sigma^+$ be the set of finite letter words with letters in $\Sigma$ and let $\Sigma^n_m$ be defined as in \eqref{eq: def Sigma nm}. There are infinitely many $m$ such that there is a map $\theta_m \colon \Sigma^+ \to \Sigma^+$ with the following properties for all $n \in \N$ sufficiently large,
\begin{enumerate}
\item $\theta_m(\Sigma^n) \subset \Sigma^n_m$;
\item for all $\omega\in \Sigma^n_m$ we have $ \# \theta_m^{-1}(w) \leqslant 2^{\varepsilon n}$;
\item if $w \in \Sigma^n$ and $w'= \theta_m(w)$, we have 
\[
\sup_{x \in \langle w \rangle} \psi^c_n(x)
\leqslant
\sup_{y \in \langle w'  \rangle \cap \mathbb{X}_m^c } \psi^c_n(y) + \varepsilon n.
\]
\end{enumerate}
\end{prop}

\begin{proof}[Proof of Proposition \ref{prop: PtopmtoPtop}]
Let $t \geqslant 0$ and $\varepsilon > 0$. Then, the sequence $(p_m(t))_{m \in \N}$ with $p_m(t) = \mc P_{\operatorname{top}} (t \psi^c | \mathbb{X}^c_m)$ is monotonously increasing in $m$, and hence 
\[
p(t): = \lim_{m \to \infty} p_m(t) \leqslant \mc P_{\operatorname{top}}(t \psi^c) 
\]
certainly exists. Choose $m \in \N$ and $\theta_m$ as in Proposition~\ref{PROP:word-regularisation} and $n$ sufficiently large. Then
\begin{align*}
\sum_{w \in \Sigma^n} \sup_{x \in \langle w \rangle} \exp(t \psi^c_n(x))
& \leqslant \me^{ t \varepsilon n} \sum_{w \in \Sigma^n} \sup_{y \in \langle \theta_m(w)  \rangle \cap \mathbb{X}_m^c} \exp(t \psi^c_n(y))
\\ &\leqslant \me^{(t+\log 2) \varepsilon n} \sum_{w' \in \Sigma^n_m} \sup_{y \in \langle w' \rangle \cap \mathbb{X}_m^c } \exp(t\psi^c_n(y)).
\end{align*}
Taking logarithms, dividing by $n$ and taking limits, yields
\[
\mc P_{\operatorname{top}}(t \psi^c) \leqslant p_m(t)   + (t+\log 2) \varepsilon,
\]
for infinitely many $m$ and hence $\mc P_{\operatorname{top}}(t \psi^c) \leqslant p(t) + (t+\log 2) \varepsilon$. Since $\varepsilon>0$ was arbitrary, the claim follows.
\end{proof}

Finally, we are in the position to give a proof of Theorem \ref{thm: Ptop=Pvar}.

\begin{proof}[Proof of Theorem \ref{thm: Ptop=Pvar}] 
For $c=0$, the equality of topological pressure and (restricted) variational pressure follows from the explicit expressions in Proposition~\ref{PROP:top-P-at-0} and \eqref{EQ:var-P-at-0}. In the following, we assume $c \neq 0$. If we set  
 \begin{align*}
 \mathcal{P}_{\operatorname{var}}\left(t\psi^{c}\lvert \X_m^{c}\right) \coloneqq \sup \left\{h_{\nu} + t\int \psi^c\,\mathrm{d}\nu\colon \nu\in \mathcal{M}_{T}\left(\X_m^{c}\right)\right\},
\end{align*}
with
\begin{align*}
 \mathcal{M}_T\left(\X_m^{c}\right) = \left\{\nu \in \mathcal{M}_T\colon \supp(\nu) \subset \X_m^c\right\},
\end{align*}
then we have 
$\mathcal{P}_{\operatorname{var}}\left(t\psi^{c}\lvert \X_m^{c}\right)=\mathcal{P}_{\operatorname{top}}\left(t\psi^{c}\lvert \X_m^{c}\right)$ as $\psi^c$ is H\"older continuous on $\X_m^{c}$. 

Furthermore, for all $m\in\mathbb{N}$ there exists $\delta>0$ such that $\langle\breve{c}_{[1,m+1]}\rangle\supset B_{\delta}( \breve{c})$ 
(or $(\langle\breve{c}_{[1,m+1]}\rangle\cup \langle\breve{c}^{\ast}_{[1,m+1]}\rangle) \supset B_{\delta}(\breve{c})$ in case of a dyadic representation of $\breve{c}$) and on the other hand for all $\delta>0$ there exists $m\in\mathbb{N}$ such that $B_{\delta}(\breve{c})\supset \langle\breve{c}_{[1,m+1]}\rangle$ (or $B_{\delta} (\breve{c})\supset \langle\breve{c}_{[1,m+1]}\rangle\cup \langle\breve{c}^{\ast}_{[1,m+1]}\rangle$) in case of a dyadic representation of $\breve{c}$).
Hence, by Propositions \ref{prop: Xm Kdelta}, \ref{prop: pm to p} and \ref{prop: PtopmtoPtop} we have
\begin{align*}
 \mathcal{P}_{\operatorname{var},c}(t\psi^c)
 =\lim_{\delta\searrow 0}p_{\operatorname{var}}(t;\delta)
 =\lim_{m\to\infty}\mathcal{P}_{\operatorname{var}}\left(t\psi^c\lvert \X_m^c\right)
 =\lim_{m\to\infty}\mathcal{P}_{\operatorname{top}}\left(t\psi^c\lvert \X_m^c\right)
 =\mathcal{P}_{\operatorname{top}}\left(t\psi^c\right)
\end{align*}
giving \eqref{eq: pressure 1}. 

In order to prove \eqref{eq: pressure 2} we follow the proof of \cite[Prop.~1.21]{Bow}.
Let $\nu\in \mathcal{M}_T$. Then we have 
$\int t\psi_n^c/n \,\mathrm{d}\nu = t\int \psi^c\, \mathrm{d}\nu$. Hence, we obtain for $t\geqslant 0$ that
\begin{align*}
h_{\nu} + t\int \psi^{c}\, \mathrm{d}\nu
 &\leqslant \lim_{n\to\infty} \frac{1}{n}\left(\sum_{w\in \Sigma^n} \nu(\langle w \rangle)(-\log \nu(\langle w \rangle))+t \sup_{x\in \langle w \rangle} \psi_n^{c}(x) \right)\\
 &\leqslant \lim_{n\to\infty} \frac{1}{n}\log \sum_{w\in \Sigma^n}\sup_{x\in \langle w \rangle}\exp(t\psi_n^{c}(x))=\mathcal{P}_{\operatorname {top}}(t\psi^{c}),
\end{align*}
where the last inequality follows from elementary calculus, compare \cite[Lem.~1.1]{Bow}.
In particular, we have 
\begin{align*}
 \mathcal{P}_{\operatorname {var}}(t\psi^{c})
 =\sup_{\nu\in \mathcal{M}_T}h_{\nu} + t\int \psi^{c}\, \mathrm{d}\nu\leqslant \mathcal{P}_{\operatorname {top}}(t\psi^{c})
 \end{align*}
implying \eqref{eq: pressure 2}. 

Finally, we note that $\mathcal{P}_{\operatorname{var}}\left(t\psi^{c}\right)$ is a supremum over affine functions and hence convex.
\end{proof}

Two comments regarding the just concluded proof are in order.

\begin{remark}
By \cite[Thm.~4.4.11]{keller}, for $t\geqslant 0$ we already have that $\mathcal{P}_{\operatorname{var}}\left(t\psi^{c}\right)=\mathcal{P}_{\operatorname{top}}\left(t\psi^{c}\right)$. 
\end{remark}

\begin{remark}\label{rem: proof of FSSthm}
 The proof of Theorem \ref{cor: dim} is given in \cite[Sec.~6]{fan_generalized_2020}. One central part of this proof is \cite[Prop.~6.4]{fan_generalized_2020} stating 
 \begin{align}
   p^*(t) = \lim_{\delta\searrow 0} p^*(t; \delta).\label{eq: p* conv with delta}
 \end{align}
Our results provide an alternative method to show this statement by noting that 
  the proof of Theorem \ref{thm: Ptop=Pvar} implies that for each $\delta>0$ sufficiently small we can find $m, m'\in\N$ such that $\mathcal{P}_{\operatorname{top}}\left(t\psi^{c}\lvert \X_m^{c}\right)\leqslant p_{\operatorname{var}}(t;\delta)\leqslant \mathcal{P}_{\operatorname{top}}\left(t\psi^{c}\lvert \X_{m'}^{c}\right)$. Applying then \cite[Prop.~4.1]{MR2719683} we obtain \eqref{eq: p* conv with delta}.
\end{remark}

\section{Measure decay of the generalized Thue{\ts}--Morse measure}\label{sec:MeasureDec}

In this section for $\omega\in \Sigma^n$  we will first find ways to estimate $\mu_c(\langle \omega\rangle)$ comparing it with $\psi^c_{n}(x)$, $x\in  \langle \omega\rangle$. These results will help us to prove Theorem \ref{thm: dim spectrum} and Theorem \ref{thm: pressure beta}.

\subsection{Gibbs type properties}\label{subsec:Gibbs-type}

It is immediately clear that $\mu_c$ can not be a (weak) Gibbs measure (see \cite{K01} for a definition of weak Gibbs measures) for any $c\in \T$ as for any $\omega \in \Sigma^n$ we have $\inf_{x\in \langle \omega \rangle}\exp(\psi_n^c(x))=0$.
However, in this section we collect a number of statements presenting weak (partly one-sided) Gibbs properties which will be helpful in the following.

 We start by recalling a few basic facts about $g$-measures. To a given continuous $g$-function $g \colon \T \to \T$, satisfying $\sum_{y \in T^{-1}x} g(y) \equiv 1$, we assign a transfer operator $\varphi_g$ on the space $C(\T)$ of continuous functions from $\T$ to $\R$ via
\[
\varphi_g(f) \colon x \mapsto \sum_{y \in T^{-1}x} g(y) f(y).
\]
We define its dual action on the set of  Borel probability measure $\nu$ on $\T$ via the relation $(\varphi^{*}_g \nu)(f) = \nu(\varphi_g f)$, where as before $\nu(f) = \int f \dd \nu$. Following Keane \cite{Kea72}, we call $\nu$ a \emph{$g$-measure} (with respect to $g$), if it is invariant under $\varphi^{*}_g$. Due to a result by Ledrappier \cite{Ledrappier}, this is the case precisely if $\nu$ is an equilibrium measure with respect to the potential $\psi_g = \log g$. While the existence of $g$-measures is automatic for continuous $g$, uniqueness is more subtle in general and depends on the set of pre-images of one
$g^{-1}(\{1\})$  and the regularity of $g$. For convenience, we restrict ourselves to H\"{o}lder continuous functions $g$ in the following. As long as there is at most one position $x_1$ with $g(x_1) = 1$, uniqueness of the $g$-measure $\nu_g$ is guaranteed, and the corresponding $g$-measure is continuous, unless the position $x_1 = 0$ is the unique fixed point of $T$ (in which case $\nu_g = \delta_0$) \cite{ConzeRaugi}. If $x_1 \neq 0$, the measure $\nu_g$ has full topological support \cite[Thm.~4.1]{BCEG} and $\nu_g$ is singular with respect to Lebesgue measure unless $g \equiv 1/2$. 
For our family of $g$-functions $(g^c)_{c \in \T}$, this yields the following (for notational convenience, let $\varphi_c = \varphi_{g^c}$).
\begin{coro}\label{coro: unique g}
For each $c \in \T$, the measure $\mu_c$ is the unique $g$-measure with $\varphi_{c}^{*} \mu_c = \mu_c$. We have $\mu_0 = \delta_0$, and for all $c \neq 0$ the measure $\mu_c$ is singular continuous and has full topological support on $\T$.
\end{coro}

In fact, 
 $\mu_c$ is strongly mixing with respect to the dynamical system $(\TT,T)$ and is attractive in the sense that $(\varphi_{c}^*)^n \nu$ converges weakly to $\mu_c$ for every $\nu \in \mc M_T$ \cite{Kea72}. Choosing $\nu$ to be the Lebesgue measure, we obtain that $(\varphi_c^*)^n \nu$ is absolutely continuous with respect to the Lebesgue measure, with density function
\[
\prod_{m=0}^{n-1} 2 g^{c} (T^m x),
\] 
compare \cite[Prop.~1]{FanLau}.
For $n \to \infty$, this gives the product representation of $\mu_c$ in \eqref{eq:mu_c_product}. Using the transfer operator $\varphi_c^*$ will also be helpful to prove the following self-consistency relation.

\begin{lemma}
\label{LEM:mu-on-interval}
Let $n \in \N$ and $I \subset \mathbb{T}$ an interval of length at most $2^{-n}$. If $|I| = 2^{-n}$ assume that at least one of the boundary points is not contained in $I$.  Then, the restriction $T_I^n \colon I \to T^n I$ is bijective and, writing $T_I^{-n}$ for its inverse,
\[
\mu_c(I) = \int_{T^n I} \exp(\psi_n^c(T_I^{-n}x) \dd \mu_c(x).
\]
In particular,
\[
\log \mu_c(T^n I) + \inf_{x \in I} \psi_n^c(x)
\leqslant \log \mu_c(I)
\leqslant \log \mu_c(T^n I) + \sup_{x \in I} \psi_n^c(x).
\]
\end{lemma}

\begin{proof}
This follows from the fact that $\mu_{c}$ is a $g$-measure for the $g$-function $g(x) = \exp(\psi^c(x))$, the proof is similar to \cite[Prop.~6.1.14]{gohlke}. In fact, $\mu_{c}$ being invariant under the transfer operator $\varphi_{c}$, we have
\[
\mu_c(I) = \mu_c(\mathds{1}_I) = \mu_c(\varphi_c^n \mathds{1}_I),
\]
where $\mathds{1}_I$ is the indicator function on $I$,
and, writing $g_n^{c}(x) := \exp(\psi_n^c(x))$,
\[
(\varphi_c^n \mathds{1}_I)(x)
= \sum_{y \in T^{-n}x} g_n^{c}(y) \mathds{1}_I(y)
= g_n^{c}(T_I^{-n} x) \mathds{1}_{T^n I}(x),
\]
using that $y \in I$ requires $x = T^n y \in T^n I$ and, in this case, $y = T_I^{-n} x$ because the restriction of $T^n$ to $I$ is bijective. This shows the first claim. The second claim follows by estimating the function $g_n^{c}(T^{-n}_Ix)$ with its infimum (or supremum) and taking logarithms.
\end{proof}

\begin{remark}
If $c \neq 0$, the restriction in Lemma~\ref{LEM:mu-on-interval} that $I$ may not be closed if $|I|=2^{-n}$ is not essential. In this case, the claim remains true if we choose $T^{-n}_I$ to be any of the inverse branches of $T_I^n \colon I \to T^n I$ because $\mu_c$ assigns no mass to the boundary points.
\end{remark} 

\begin{coro}
\label{COR:mu-on-double-word}
Let $w \in \Sigma^n$ and $v \in \Sigma^m$ for some $n \in \N$ and $m \in \N_0$. Then, for all $c \in (0,1)$, 
\[
\log \mu_{c}(\langle v\rangle) + \inf_{x \in \langle wv\rangle}  \psi^c_n(x)
\leqslant \log \mu_{c}(\langle wv\rangle) 
\leqslant \log \mu_{c}(\langle v\rangle) + \sup_{x \in \langle wv\rangle} \psi^c_n(x)
\]
and for $v$ being the empty word this simplifies to 
\[
\inf_{x \in \langle w\rangle} \psi^c_n(x)
\leqslant \log \mu_{c}(\langle w\rangle) 
\leqslant \sup_{x \in \langle w\rangle} \psi^c_n(x).
\]
\end{coro}

There is a global lower bound for the decay of the measure on cylinder sets, as long as the measure $\mu_c$ is non-atomic.

\begin{prop}
\label{PROP:measure-decay}
Assume that $c \neq 0$. Then,
\[
\liminf_{n \to \infty} \frac{1}{n^2} \inf_{w \in \Sigma^n} \log \mu_c(\langle w \rangle) \geqslant -1.
\]
\end{prop}

\begin{remark}
For the case of arbitrary $c \neq 0$, the bound $-1$ in the preceding lemma is sharp. Indeed, for the Thue{\ts}--Morse measure, the limit exists and equals $-1$. Of course, for particular choices of $c$, better estimates can be obtained.
\end{remark}

The proof of this as well as of all the following propositions in this subsection are postponed to a later point in Subsection \ref{subsec: proofs gibbs type}.

The following proposition gives a two-sided bound for those systems where $\mc P_{\operatorname{top}}(t \psi^c) < \infty$.
It will play a central role in the proof of Theorem \ref{thm: pressure beta}.
Note that $\mc P_{\operatorname{top}}(t \psi^c) < \infty$ holds clearly for all $t \geqslant 0$, and $\mc P_{\operatorname{top}}(t \psi^c) = \infty$ holds for some $t<0$ if and only if it holds for all $t < 0$ as follows from a straightforward calculation.

\begin{prop}
\label{PROP:Gibbs-like}
Let $c \neq 0$. If $\mc P_{\operatorname{top}}(t\, \psi^c) < \infty$ for all $t\in \R$, we have 
\[
\sup_{w \in \Sigma^n} \Bigl|\log \mu_{c}(\langle w \rangle) - \sup_{x \in \langle w \rangle} \psi^c_n(x) \Bigr| \in o(n).
\]
\end{prop}

\begin{remark}
This result shows that if $\mc P_{\operatorname{top}}(t\,  \psi^c) < \infty$ and $c \neq 0$, the measure $\mu_{c}$ is in fact very close to a (weak) Gibbs measure. More precisely, we obtain that there is a sequence $(\ell_n)_{n \in \N}$ such that for all $n$ and $w \in \Sigma^n$,
\[
\sup_{x \in \langle w \rangle} \exp(\psi
^c_n(x) - \ell_n) \leqslant \mu(\langle w \rangle) \leqslant \sup_{x \in \langle w \rangle} \exp(\psi^c_n(x)),
\]
and $\ell_n \in o(n)$.
From a metrical point of view the case $\mc P_{\operatorname{top}}(t\,  \psi^c) < \infty$ indeed is the typical case; due to \cite[Thm.~A]{fan_generalized_2020} and Theorem \ref{thm: Ptop=Pvar} the set of points $c$ for which $\mc P_{\operatorname{top}}(t\,  \psi^c) = \infty$, for $t<0$ is a residual set of zero Hausdorff dimension.
\end{remark}

Finally, if we restrict ourselves to intervals that intersect the $T$-invariant subset $K_{\delta} \subset \TT$ given in \eqref{eq: def Kdelta}, then $\psi^c$ is H\"older continuous on this set and we get Gibbs bounds with a constant that depends on the parameter $\delta$. 
We also note that $K_\delta$ is closed and hence compact.
 
\begin{lemma}
\label{LEM:mu-on-Kdelta}
Let $c \neq 0$ and $0 < \delta < 1/2$. Then there exists a number $R:=R(c,\delta)>0$ with the following property. If $I$ is an interval of length $2^{-n}$, and  $x \in I \cap K_\delta \neq \emptyset$, then
\[
  R^{-1} \exp \left(\psi^c_n (x)\right)
  \leqslant  \mu^{\pa}_{c}(I)
  \leqslant R \exp \left(\psi^c_n (x)\right). 
\]
\end{lemma}

\subsection{Connection to $L^q$-spectrum -- proof of Theorem \ref{thm: pressure beta}}\label{subsec: Lq spectrum}

The results collected in Section \ref{sec: pressure} and the last Subsection \ref{subsec:Gibbs-type} enable us to prove Theorem \ref{thm: pressure beta}.

\begin{proof}[Proof of Theorem \ref{thm: pressure beta}]
 By the last line of Corollary~\ref{COR:mu-on-double-word} we have for $t\geqslant 0$ that 
 \begin{align*}
  \overline{\beta}_{\mu_c}(t)
  &\leqslant  \limsup_{n\to\infty} \frac{\log \left(\sum_{\omega\in\Sigma^n} \sup_{y\in \langle\omega\rangle} \exp\left(t\, \psi^c_n(y)\right)\right)}{n\log 2}
  = \frac{\mathcal{P}_{\operatorname{top}}(t\, \psi^c)}{\log 2}. 
 \end{align*}
 
 On the other hand, by Lemma~\ref{LEM:mu-on-Kdelta} and Proposition~\ref{prop: Xm Kdelta} we have for $t\geqslant 0$ and all $m\in\mathbb{N}$ sufficiently large that
 \begin{align*}
  \underline{\beta}_{\mu_{c}}(t)
   &\geqslant \liminf_{n\to\infty}
  \frac{\log R+\log \left(\sum_{\omega\in\Sigma^n_m} \sup_{y\in \langle\omega\rangle\cap \mathbb{X}_m^{c}} \exp\left(t\,\psi^c_n(y)\right)\right)}{n\log 2}
  \geqslant \frac{\mathcal{P}_{\operatorname{top}}(t\psi^{c}\vert\mathbb{X}_m^{c})}{\log 2}
 \end{align*}
 and by Proposition \ref{prop: pm to p} we have $\underline{\beta}_{\mu_c}(t)\geqslant \mathcal{P}_{\operatorname{top}}(t\,\psi^c)/\log 2$. 
 
 In the next steps we consider the case $t<0$. By Corollary~\ref{COR:mu-on-double-word} we have 
 \begin{align*}
   \underline{\beta}_{\mu_c}(t)
  &\geqslant  \liminf_{n\to\infty} \frac{\log \left(\sum_{\omega\in\Sigma^n} \inf_{y\in \langle\omega\rangle} \exp\left(t\, \psi^{c}_n(y)\right)\right)}{n\log 2}
  = \frac{\mathcal{P}_{\operatorname{top}}(t\, \psi^c)}{\log 2}. 
 \end{align*}

To show the upper bound for $t<0$ we notice that if
$\mc P_{\operatorname{top}}(t\, \psi^c) = \infty$, there is nothing to show. Otherwise, we can control the size of $\mu_{c}$ on cylinder sets via Proposition~\ref{PROP:Gibbs-like}. In particular, there is a sequence $(\iota_n)$ with $\iota_n \in o(n)$ such that
\[
\log \mu(\langle w \rangle) \geqslant \sup_{x \in \langle w \rangle} \psi^c_n(x) - \iota_n,
\]
for all $w \in \Sigma^n$. It follows that
\begin{align*}
 \overline{\beta}_{\mu_c}(t) \, \log 2 
& = \limsup_{n \to \infty} \frac{1}{n} \log \left( \sum_{w \in \Sigma^n} \exp(t\, \log \mu_{c}(\langle w \rangle))  \right)\\
& \leqslant \limsup_{n \to \infty} \frac{1}{n} \log
\left( \sum_{w \in \Sigma^n} \exp \Bigl(t \sup_{x \in \langle w \rangle}  \psi^c_n(x) \Bigr)\right) - \frac{t \iota_n}{n}
= \mc P_{\operatorname{top}}(t\, \psi^c),
\end{align*}
completing the proof.
\end{proof}

\subsection{Proof of Theorem \ref{thm: dim spectrum}}

It is well-known that the dimension spectrum of a measure can be bounded above by the Legendre transform of its $L^q$-spectrum, see for example \cite{LauNgai} - a fact that we will need to prove Theorem \ref{thm: dim spectrum}. Since definitions and formulations of the result differ slightly in the literature, we give a self-contained proof. 

\begin{lemma}
\label{LEM:f-alpha-lq-general}
Let $\nu$ be a (fully-supported) measure on $\TT$ and $f_{\nu}$ its dimension spectrum. Assume $\beta_{\nu}(q)$ exists for all $q \in \R$. Then, $f_{\nu}(\alpha) \leqslant \max\{ - \beta_\nu^{\ast}(-\alpha), 0 \}$ for all $\alpha \in \mathbb{R}$.
\end{lemma}

\begin{proof}
First, fix $\alpha \in \R$ and some $\varepsilon > 0$, and let us set $r_n = 2^{-n}$. Then, $d_{\nu}(x) = \alpha$ implies that $r_n^{\alpha + \varepsilon}\leqslant \nu(B_{r_n}(x)) \leqslant r_n^{\alpha - \varepsilon}$ for all large enough $n \in \N$. For each such $n$, we have $\langle x_1 \ldots x_{n+1}\rangle \subset B_{r_n}(x)$. That is, $x \in \langle v \rangle$ for some $v \in \Sigma^{n+1}$ with
\[
\nu(\langle v \rangle) \leqslant  \nu(B_{r_n}(x)) \leqslant r_n^{\alpha - \varepsilon}.
\]
If we set $\Gamma_n^{-} = \{ v \in \Sigma^n : \nu(\langle v \rangle) \leqslant r_{n-1}^{\alpha - \varepsilon}  \}$, we therefore get for all $m \in \N$,
\begin{equation}
\label{EQ:f-alpha-minus-bound}
\mc F_{\nu}(\alpha) = \{ x \in \mathbb{T}: d_\nu(x) = \alpha \} \subset \bigcup_{n \geqslant m} \bigcup_{v \in \Gamma_n^{-}} \langle v \rangle.
\end{equation}
Similarly, $x \in \langle w \rangle$ for words $w,w^{L},w^{R} \in \Sigma^n$ such that $B_{r_n}(x) \subset \langle w^L\rangle \cup \langle w \rangle \cup \langle w^{R} \rangle$. Here $w^L$ and $w^R$ are the left and right neighbors of $w$ in (cyclic) lexicographic order. The sets 
\[
\Gamma_n^+ = \{ w \in \Sigma^n :  \nu(\langle w^L \rangle) + \nu(\langle w \rangle) +  \nu(\langle w^R \rangle) \geqslant r_n^{\alpha + \varepsilon} \}
\]
therefore also provide a cover, for all $m \in \N$, via
\begin{equation*}
\mc F_{\nu}(\alpha) \subset \bigcup_{n\geqslant m} \bigcup_{w \in \Gamma_n^+} \langle w \rangle.
\end{equation*}
For $s > \max\{- \beta_{\nu}^{\ast}(-\alpha),0 \}$, we use these covers to show that the $s$-dimensional Hausdorff measure of $\mc F_{\nu}(\alpha)$ vanishes, implying the desired bound on the Hausdorff dimension of $\mc F_{\nu}(\alpha)$. By definition of the Legendre transform, we can find $q \in \R$ such that
\[
s > \alpha q + \beta_\nu(q).
\]
We proceed by a distinction of cases. First assume that $q\leqslant 0$. The value of $\beta_{\nu}(q)$ provides a bound on the cardinality of $\Gamma_n^-$ because for each $\delta > 0$, we get for large enough $n\in \N$,
\[
r_n^{- \beta_\nu(q) - \delta} \geqslant
\sum_{w \in \Sigma^n} \nu(\langle w \rangle)^q 
\geqslant \sum_{w \in \Gamma_n^-} \nu(\langle w \rangle)^q \geqslant r_{n-1}^{(\alpha - \varepsilon) q} \# \Gamma_n^-.
\]
We use the cover in \eqref{EQ:f-alpha-minus-bound} to bound the $s$-dimensional Hausdorff measure of $\mc F_\nu(\alpha)$ from above by
\[
\sum_{n\geqslant m} \# \Gamma_n^- r_n^s 
\leqslant C \sum_{n \geqslant m} r_n^{s - \beta_\nu(q) -\delta - (\alpha - \varepsilon) q} \xrightarrow{m \to \infty} 0,
\]
provided that $\varepsilon > 0$ and $\delta>0$ are chosen small enough. This shows that $f_\nu(\alpha) > s$ in the case $q \leqslant 0$. 
If $q>0$, we estimate $\# \Gamma_n^+$ via
\begin{align*}
3^{q+1} \sum_{w \in \Sigma^n}  \nu(\langle w \rangle)^q &=  \sum_{w \in \Sigma^n} 3^q \left(  \nu(\langle w^L \rangle)^q +  \nu(\langle w \rangle)^q +  \nu(\langle w^R \rangle)^q \right)
\\ &\geqslant \sum_{w \in \Gamma_n^+} \left(  \nu(\langle w^L \rangle) + \nu(\langle w \rangle) +  \nu(\langle w^R \rangle) \right)^q 
\geqslant \# \Gamma_n^+ r_n^{(\alpha + \varepsilon)q}.
\end{align*}
The remainder of the proof proceeds just like in the case $q\leqslant 0$.
\end{proof}

\begin{proof}[Proof of Theorem \ref{thm: dim spectrum}]
We first note that
\begin{align*}
 \{x\in  \T \colon  d_{\mu_c}(x) = \alpha\}
 \supseteq \left\{x\in K_{\delta} \colon d_{\mu_c} (x) = \alpha\right\},
\end{align*}
for all $\delta>0$ with $K_{\delta}$ as in \eqref{eq: def Kdelta}.
 In the following we will prove
\begin{align}
 \left\{x\in K_{\delta} \colon d_{ \mu_c} (x) = \alpha\right\}
 =\left\{x \in K_{\delta} \colon \overline{\psi}^c(x) = - \alpha \log 2\right \},\label{eq: sets for xin Kdelta}
\end{align}
where $\overline{\psi}^c(x)$ denotes the Birkhoff average $\lim_{n\to\infty}\psi_n^c(x)/n$.
We  note that 
\begin{align*}
\lim_{\delta\searrow 0}\dim_H\left\{x \in K_{\delta} \colon \overline{\psi}^c(x) =\alpha\right\}
= \lim_{\delta\searrow 0}-\frac{p_c^*(\alpha;\delta)}{\log 2}
=-\frac{p_c^*(\alpha)}{\log 2}
= \dim_H\left\{x \in \mathbb{T} \colon \overline{\psi}^c(x)=\alpha\right\},
\end{align*}
where the first equality follows from the fact that $\psi^c$ is H\"older continuous on $K_{\delta}$, the second equality is given in \eqref{eq: p* conv with delta} and the third equality is Theorem \ref{cor: dim}. 
Hence, \eqref{eq: sets for xin Kdelta} establishes the lower bound of Theorem \ref{thm: dim spectrum}.
Let $x \in K_\delta$. Applying Lemma~\ref{LEM:mu-on-Kdelta} to $I = B_{2^{-n-1}}(x)$, we obtain, for all $n \in \N$,
\[
R^{-1} \leqslant \frac{\mu_c(B_{2^{-n-1}}(x))}{\exp(\psi^c_n(x))} \leqslant R.
\]
This yields
\[
d_{\mu_c}(x) = \lim_{n \to \infty} \frac{\log \mu_c(B_{2^{-n}}(x))}{\log 2^{-n}} = \lim_{n \to \infty} \frac{\psi^c_n(x)}{-n \log 2} 
=- \frac{\overline{\psi}^c(x)}{\log 2},
\]
whenever the limit exists.

To prove the upper bound, recall from Lemma~\ref{LEM:f-alpha-lq-general} that $f_c(\alpha) \leqslant \max\{-\beta_{\mu_c}^{\ast}(-\alpha),0\}$ for all $\alpha \in \R$. By Theorem~\ref{thm: pressure beta}, we have $p_c =\beta_{\mu_c} \log 2$, which implies $\beta^{\ast}_{\mu_c}(\alpha) = p^{\ast}(\alpha \log 2)/\log 2$ for all $\alpha$. Hence,
 \[
 f_c(\alpha) \leqslant \max\left\{ - \frac{p_c^\ast(- \alpha \log 2)}{\log 2},0 \right\},
 \]
 which is precisely the expression for $b_c(- \alpha \log 2)$ in Theorem~\ref{cor: dim}, if $-\alpha \log 2 \neq \underline{\alpha}_c$.
\end{proof}

\section{Applications of the $L^q$-spectrum to Fourier dimensions, Kre\u{\i}n--Feller operators, and quantization dimensions}\label{sec: KreinFeller}
In this section we demonstrate how special points of the $L^q$-spectrum $\beta_{\mu_{c}}$   can be related to other analytic and fractal-geometric characteristics of $\mu_{c}$. A particularly prominent example is the \emph{Fourier dimension}  $\dim_{F}(\mu_c)$  as defined in \eqref{DimF} which we consider next.

\subsection{The Fourier dimension}
\label{SEC:Fourier-dimension}

The following statement holds for arbitrary measures  $\nu$ on $\R$ and gives the connection of the Fourier dimension and the $L^{q}$-spectrum, see for example \cite{Lau92} and \cite[Thm.~6.6]{Strichartz}. If the $L^q$-spectrum $\beta_\nu$ exists as a limit, we have 
\begin{equation}
\label{EQ:dim_F-Lq}
\dim_F(\nu) = - \beta_\nu(2).
\end{equation}

If $\nu$ has compact support, regarding it as a measure on the torus yields a similar relation that employs the discrete Fourier coefficients instead of the Fourier transform on the real line. This can be understood from the fact that the Pontryagin dual of the torus is given by the integer lattice (whereas the real line is self-dual). The natural analogue of \eqref{EQ:dim_F-Lq} reads as follows.

\begin{prop}
\label{PROP:FD-characterisation}
Let $\nu$ be a Borel probability measure on the torus $\T$ and assume that $\beta_{\nu}(2)$ exists. Then,
\begin{equation}
\label{EQ:dim-F-analogue}
\lim_{n \to \infty} \frac{\log\Bigl(\sum_{m=0}^{n-1} |\widehat{\nu}(m)|^2 \Bigr)}{\log n} = 1 - \beta_{\nu}(2).
\end{equation}
In particular, the limit exists and coincides with $1+\dim_F(\nu)$.
\end{prop}

\begin{proof}
The last statement follows immediately from \eqref{EQ:dim_F-Lq}. For the proof of \eqref{EQ:dim-F-analogue}, we follow similar ideas as in \cite{GH}.
First, let us regard $\nu$ as a measure on $[0,1] \subset \R$. Then, an alternative expression for $\beta_{\nu}(2)$ is given by
\[
\beta_{\nu}(2) = \lim_{r \to 0}  \frac{1}{\log r}  \log \Bigl( \int_{\R} |\nu(B_r(x))|^2 \dd x \Bigr) - 1,
\]
see for example \cite[Thm.~6.2]{Strichartz}. Considering the ball to be defined on the torus instead of $\R$ 
gives a negligible contribution since for $r < 1/2$,
\[
\int_{\R} |\nu(B_r(x))|^2 \dd x 
\leqslant  \int_{\T} |\nu(B_r(x))|^2 \dd x 
\leqslant 2 \int_{\R} |\nu(B_r(x))|^2 \dd x,
\]
which leaves the exponential growth behaviour unchanged. We may hence work with an integration over the torus instead. 
For convenience of notation, let us set
\[
\Gamma(r) = \int_{\T} |\nu(B_r(x))|^2 \dd x.
\]
Note that $\nu(B_r(x))$ as a function of $x$ can also be expressed as $\nu \ast \mathds{1}_{B_r}$. Using Parseval's identity, we hence get
\[
\Gamma(r) = \sum_{m \in \Z} |\widehat{\nu}(m) \widehat{\mathds{1}_{B_r}}(m)|^2,
\]
where 
\[
\widehat{\mathds{1}_{B_r}}(m) = 2r \sinc(2r \pi m), 
\quad \sinc(x) \coloneqq\frac{\sin(x)}{x}.
\]
Since $\sinc(x)$ is continuous in $x$ and $\sinc(0) =1$, it is easy to see that there is an interval $I = [-2 \pi K, 2 \pi K]$ with $K > 0$ such that $4 \sinc^2 (x) \geqslant \mathds{1}_I(x)$ for all $x \in \R$. Hence,
\[
\Gamma(r) \geqslant r^2 \sum_{m \in \Z} |\widehat{\nu}(m)|^2 \mathds{1}_{I}(2 r \pi m)
= r^2 \sum_{|m| \leqslant Kr^{-1}} |\widehat{\nu}(m)|.
\]
Along the sequence $r_n = K/n$ this yields
\[
\Gamma(r_n) = \frac{K^2}{n^2} \sum_{|m| \leqslant n} |\widehat{\nu}(m)|^2 \geqslant \frac{K^2}{n^2} M_n,
\]
where $M_n=\sum_{m \leqslant n} |\widehat{\nu}(m)|^2$.

Therefore, we obtain the upper bound
\[
\beta_{\nu}(2) = \lim_{n \to \infty} \frac{\log \Gamma(r_n)}{\log r_n} -1
\leqslant \liminf_{n \to \infty} \frac{\log (n^{-2}  M_n)}{-\log n} -1 = 1 - \limsup_{n\to\infty}\frac{\log  \sum_{m \leqslant n} |\widehat{\nu}(m)|^2}{\log n}.
\]
Showing the opposite bound is slightly more involved, because $\sinc(x)$ cannot be dominated by the indicator function on a finite interval. We therefore proceed a bit more cautiously and replace the function $\mathds{1}_{B_r}$ by a smoother function before taking the Fourier transform. More precisely, we choose a $C^{\infty}$-function $g \colon \R \to \R_{\geqslant 0}$ with compact support and such that $g(x) = 1$ on $[-1,1]$.
In particular, $g$ is in the Schwartz space of rapidly decreasing functions and the same holds for $\widehat{g}$.
Setting $g_r(x) = g(x/r)$ for all $r > 0$, we obtain for small enough $r$ that $g_r$ is supported on the unit interval $[-1/2,1/2]$ (identified with the torus $\T$ in the following). In this case, we have $\mathds{1}_{B_r} \leqslant g_r$ and therefore obtain
\[
\Gamma(r) \leqslant \int_{\T} |\nu \ast g_r(x)|^2 \dd x.
\]
Again, Parseval's identity yields
\[
\Gamma(r) \leqslant \sum_{m \in \Z} |\widehat{\nu}(m) \widehat{g_r}(m)|^2.
\]
Note that 
\[
\widehat{g_r}(m) = \int_{\T} g_r(x) \me^{-2 \pi \im  m x} \dd x 
= \int_{\R} g_r(x) \me^{-2 \pi \im m x} \dd x
= r \widehat{g}(rm).
\]
Splitting the summation region into a core and a tail, we get for arbitrary $\varepsilon > 0$,
\[
\Gamma(r) 
\leqslant r^2 |\widehat{g}(0)|^2 \sum_{|m| \leqslant r^{-(1+\varepsilon)}} |\widehat{\nu}(m)|^2
+ r^2 \sum_{|m| > r^{-(1+\varepsilon)}} |\widehat{g}(rm)|^2.
\]
Without loss of generality assume that $|\widehat{g}(k)|$ is symmetric for $k \mapsto -k$ and monotonically decreasing for $k > 0$.

We can then estimate the tail via
\[
\sum_{|m|> r^{-(1 + \varepsilon)}} |\widehat{g}(rm)|^2
\leqslant 2 \int_{r^{-(1+\varepsilon)}/2}^{\infty} |\widehat{g}(rx)|^2 \dd x
= \frac{2}{r} \int_{r^{-\varepsilon}/2}^{\infty} |\widehat{g}(y)|^2 \dd y,
\] 
which in the limit $r\to 0$ decays faster than any polynomial in $r$ because $\widehat{g}$ is rapidly decreasing. Hence, for $r_n = 1/n$ and large enough $n$,
\[
\Gamma(r_n) \leqslant \frac{2|\widehat{g}(0)|^2}{n^2} \sum_{|m| \leqslant n^{1+\varepsilon}}|\widehat{\nu}(m)|^2
\leqslant \frac{  4|\widehat{g}(0)|^2}{n^2} M_{\floor{n^{1+\varepsilon}}}.
\]

Similarly as before, this yields,
\[
\beta_{\nu}(2) \geqslant \limsup_{n \to \infty} \frac{\log(n^{-2} M_{\floor{n^{1+\varepsilon}}})}{-\log n} - 1 = 1 - (1 + \varepsilon)\, \liminf_{n\to\infty}\frac{\log  \sum_{m \leqslant n} |\widehat{\nu}(m)|^2}{\log n}.
\]
Since $\varepsilon > 0$ was arbitrary, the claim follows.
\end{proof}

The next corollary is a refinement to Theorem \ref{thm: beta(2) analytic}.
\begin{coro}
\label{COR:fourier-correlation}
For  $c\in \T$, we have
\[
\beta_{\mu_{c}}(2) = 1 - {D_2^c} = 1 - \frac{\log \lambda_1^{c}}{\log 2},
\]
where $\lambda_1^{c}$ is the largest eigenvalue of the matrix $M_{c}$, as defined in Proposition~\ref{PROP:v-M-recursion}. In particular, $c\mapsto \beta_{\mu_{c}}(2)$ is a real-analytic.
\end{coro}

\begin{proof}
Recall that $\mu_c$ coincides with the restriction of $\widehat{\gamma_c}$ to the torus, represented by the unit interval. Hence, we get from Lemma~\ref{LEM:FS-coefficients} that $\widehat{\mu_c}(m) = \overline{\eta^c_m}$ for all $m \in \Z$. In particular,
\[
\sum_{m=0}^{n-1} |\widehat{\mu_c}(m)|^2 = \sum_{m=0}^{n-1} |\eta_m^{c}|^2 = \Theta_n^c,
\]
for all $n \in \N$ and hence, via Proposition~\ref{PROP:FD-characterisation},
\[
1- \beta_{\mu_c}(2) = \lim_{n \to \infty} \frac{\log  \Theta_n^c}{\log n} = D_2^{c}.
\]
The relation to $\lambda_1^{c}$ was established in Proposition~\ref{PROP:D2-lambda}.
\end{proof}

\subsection{Kre\u{\i}n--Feller operators}
Classical Kre\u{\i}n--Feller operators have been introduced in 
\cite{kreinKFop, kackreinKFop, feller_KFop} as the second order weak derivative of a function with respect to the Lebesgue measure. 

We only  give a brief  introduction of Kre\u{\i}n--Feller operators with respect to  measures with full support as details for the general case can e.g.\ be found  in  \cite[Sec.~2.1]{KessNie} or \cite{KNSW}.
First, we define the formal Laplace operator $\Delta_{\nu}$ with respect to a finite Borel measure $\nu$. 
To ease notation we always assume that our underlying interval is $[0,1]=\supp(\nu)$. 
Let $H^{1,2}$ denote the usual real Sobolev space of weakly differentiable functions $f$ on $[0,1]$ with square-integrable derivatives $f'\in L^{2}(\mathrm{Leb}|_{[0,1]})$. 
 We define the Laplacian (here only with Neumann boundary conditions) via a Dirichlet form  on $H^{1,2} \times H^{1,2} $ given by
 \begin{align*}
 \mathcal{E}(u, v)  
=\int_{[0,1]}u' v' \dd \mathrm{Leb}.
\end{align*}
 For each finite Borel measure    $\nu$  with  $\supp(\nu)=[0,1]$ we have that the natural embedding of $H^{1,2} $ into $L^{2}(\nu)$ is continuous, injective and dense.
Then $\Delta_{\nu}$ is the non-negative self-adjoint operator with the following properties: 
$f \in H^{1,2}$ lies in the domain $\dom(\Delta_{\nu})$ of the Laplacian $\Delta_{\nu}$, if and only if there exists $h \in L^2(\nu)$ such that
\begin{align*}
 \mathcal{E} (g,f)= \langle g, h\rangle_{\nu},\quad \text{ for all } g \in H^{1,2}
\end{align*}
and we set $\Delta_{\nu} f:=h$.

In this section we will investigate the asymptotic spectral properties of $\Delta_{\nu}$. 
An element $f \in H^{1,2} \backslash\{0\}$ is called \emph{eigenfunction} of $\Delta_{\nu}$ with \emph{eigenvalue} $\kappa$, if for all $g \in H^{1,2}$, we have
\begin{align*}
 \mathcal{E}( g, f) = \kappa  \langle g, f\rangle_{\nu}.
\end{align*}
Moreover, we introduce the eigenvalue counting function $N_{\nu}$ as 
\[
N_{\nu}(x):=\#\{\kappa\leqslant x:  \kappa \text{ is eigenvalue of }\Delta_{\nu}\}
\]
with which we are able to introduce the \emph{upper} and \emph{lower spectral dimension} given by 
\begin{align*}
 \overline{s}_{\nu} := \limsup_{x\to\infty}\frac{\log (N_{\nu} (x))}{\log(x)}\qquad 
 \text{ and }\qquad
\underline{s}_{\nu} := \liminf_{x\to\infty}\frac{ \log (N_{\nu} (x))}{\log(x)}.
\end{align*}
We say that the \emph{spectral dimension} exists if $\underline{s}_{\nu}= \overline{s}_{\nu}$ and denote its common value by $s_{\nu}$. For the remaining theorems of this section we introduce 
\[q_{r}\coloneqq q_r (\nu) \coloneqq \inf \{q > 0 : \beta_{\nu} (q) < rq\}.\]

\begin{theorem}\label{thm: specdim}
 For each $c \neq 0$, the spectral dimension of $\mu_{c}$ exists and can be expressed in terms of the $L^{q}$-spectrum via
 \[
 s_{\mu_{c}}=q_1(\mu_{c}).
 \]
\end{theorem}

\begin{remark}
 For $c=0$, we have the trivial case of a Dirac measure at $0$, i.e.\ $\supp(\mu_0)=0$ and hence not fulfilling the condition $\supp(\nu)=[0,1]$.
\end{remark}

\begin{proof}[Proof of Theorem \ref{thm: specdim}]  
By  \cite[Thm.~1.1]{KessNie} we have for any finite Borel measure $\nu$ on $\T$  that
  \[\overline s_{\nu}=q_{1}(\nu).\] 
Further, by Theorem \ref{thm: pressure beta} we know that $\beta_{\mu_{c}}$ restricted to    $(0,1)$ exists as a limit  rather than a limit superior. Hence, it follows from   \cite[Thm.~1.1 \& 1.2]{KessNie} that this is sufficient for  the spectral dimension  $s_{\mu_{c}}$ to  exist and consequently $s_{\mu_{c}}=q_1(\mu_{c})$.  
\end{proof}

We would like to point out that the existence of the spectral dimension is by no means self-evident, as \cite[Ex.~5.5]{KessNie} --an example in which the spectral dimension does not exist-- demonstrates.

\subsection{Quantization dimension}
As a  good reference  on quantization theory we refer the reader to  \cite{GrafLuschgy}. The main object is the  {\em $n$th quantization error of order $r\geqslant 0$} of a measure $\nu$, say, as in our case, on $[0,1]$,  given by
\begin{align*}
 \mathfrak{e}_{n,r}(\nu)
 =\begin{cases}
   \inf_{\alpha\in\mathcal{A}_n}(\int \rho(x,\alpha)^r\,\mathrm{d}\nu(x))^{1/r}& r>0\\
   \inf_{\alpha\in\mathcal{A}_n}\exp(\int\log  \rho(x,\alpha)\,\mathrm{d}\nu(x))& r=0,
  \end{cases}
 \end{align*} 
where $\mathcal{A}_n:=\{\alpha \subset \mathbb{R}: 1\leqslant \card(\alpha)\leqslant n\}$
 and $\rho(x,\alpha):=\min_{y\in\alpha} \rho(x,y)$. 
 It is known that for $r>0$, we have $\mathfrak{e}_{n,r}(\nu)=O(n^{-r})$ and,  if $\nu$ is singular with respect to the Lebesgue measure, then $\mathfrak{e}_{n,r}(\nu)=o(n^{-r})$  \cite[p.~78]{GrafLuschgy}.

In order to study this behaviour in more detail  
 we consider the notion of  \emph{upper} and \emph{lower quantization dimension for $\nu$ of order $r$}  
 given by
 \begin{align*}
 \overline{\mathsf{D}}_r (\nu) := \limsup_{n\to\infty} -\frac{\log n}{\log \mathfrak{e}_{n,r}(\nu)}\quad \text{ and } \quad
 \underline{\mathsf{D}}_r (\nu) := \liminf_{n\to\infty} -\frac{\log n}{\log \mathfrak{e}_{n,r}(\nu)}.
\end{align*}
If $\overline{\mathsf{D}}_r (\nu) =\underline{\mathsf{D}}_r(\nu)$, then we denote the common value by $\mathsf{D}_r(\nu)$ and call it the \emph{quantization dimension of $\nu$ of order $r$}.
Indeed, as shown in \cite{KessNieZhu} (for  generalizations  to negative values of $r$ see  \cite{KessNieNegtive}), this concept is closely related to the $L^q$-spectrum. 
We also introduce the \emph{(upper) generalized R\'enyi dimension}  
\begin{align*}
 \overline{\mathfrak{R}}_{\nu}(q)
 =\begin{cases}\cfrac{\beta_{\nu}(q)}{1-q}&\text{if }q\neq 1\\
\limsup_{n\to \infty} \cfrac{\sum_{C\in \mathcal{D}_n} \nu(C) \log \nu(C)}{\log (2^{-n})}
&\text{if } q = 1,
  \end{cases}
\end{align*}
where $\mathcal{D}_n$ denotes the partition of $\mathbb{R}$ by half-open intervals of the form $(k 2^{-n}, (k+1) 2^{-n}]$ with $k\in\mathbb{Z}$. 
Moreover, for $q\neq 1$ the R\'enyi dimension can be expressed by the  
\emph{Hentschel-Procaccia generalized dimension}
\begin{align*}
 \overline{\mathfrak{R}}_{\nu}(q)
 =\frac{1}{1-q}\limsup_{r\to 0}\frac{\log \int\nu(B_r(x))^{q-1}\,\mathrm{d}\nu(x)}{-\log r}.
\end{align*}

A relation between the just introduced quantities is given in \cite{KessNieZhu} as follows:
\begin{prop}[{\cite[Thm.~1.1]{KessNieZhu}}]\label{prop: kessniezhu}
 Let $\nu$ be compactly supported probability measure
on $\R^{d}$. If $\sup_{x\in(0,1)}\beta_{\nu}(x)>0$, then for every
$r>0$, 
\[
\underline{\mathsf{D}}_{r}\left(\nu\right)\leq\overline{\mathsf{D}}_{r}\left(\nu\right)=\overline{\mathfrak{R}}_{\nu}\left(q_{r}\right).
\]
Otherwise, if $\beta_{\nu}(x)=0$ for all $x>0$, then $D_{r}\left(\nu\right)=0$
for all $r\geq0$.
\end{prop}

In case of the generalized Thue--Morse measure we get the following even stronger result:

\begin{theorem}
 For   $c\in\T$     we have that the quantization dimension of order $r\geqslant 0$ exists and, setting $q_{r}\coloneqq q_{r}(\mu_{c})$,  fulfills
 \begin{align*}
 \mathsf{D}_r(\mu_{c})=\overline{\mathfrak{R}}_{\mu_{c}}(q_r)=\frac{rq_r}{1-q_r}.
 \end{align*}
\end{theorem}
For an illustration  connecting  the quantization dimension and the $L^{q}$-spectrum see Fig.~ \ref{Dr}. 

\begin{proof}
We aim to apply Proposition \ref{prop: kessniezhu}. For $c\neq 0$ we first verify $\sup_{q\in (0,1)}\beta_{\mu_c}(q)>0$. 
Indeed, by Theorem \ref{thm: pressure beta}  we have $\beta_{\mu_c}(0)=1$ and $\beta_{\mu_c}(1)=0$ and since $\psi^c$ is non-positive, we can conclude that $\mathcal{P}(\cdot \psi^c)$ 
and thus also $\beta_{\mu_c}$ is monotonically decreasing. 
Moreover, the convexity of the pressure claimed in Theorem \ref{thm: Ptop=Pvar} implies  in particular continuity. Consequently,  $\beta_{\mu_c}(\epsilon)>0$ for some $\epsilon>0$ sufficiently small. 
Hence, for $c\neq 0$, we may apply the first part of Proposition \ref{prop: kessniezhu} giving the equality $\overline{\mathsf{D}}_r(\mu_{c})=\overline{\mathfrak{R}}_{\mu_{c}}(q_r)$. 
Furthermore, the equality $\overline{\mathsf{D}}_r(\mu_{c})=\underline{\mathsf{D}}_r(\mu_{c})$, $c\neq 0$, follows from \cite[Thm.~1.11]{KessNieZhu} using the fact  that $\mu_c$ is multifractal regular  --a definition that asks for the existence of the $L^q$-spectrum in a left-sided neighborhood  around $q_r$ (cf. \cite[Def.~1.10(2a)]{KessNieZhu}). Indeed, from Theorem \ref{thm: pressure beta} we know that the  $L^q$-spectrum exists on $(0,1)$.

For $c=0$, we have $\beta_{\mu_0}=0$, for all $x>0$ and thus, by the second part of Proposition \ref{prop: kessniezhu}, $\mathsf{D}_r(\mu_{0})=0$ and an easy calculation shows that $\overline{\mathfrak{R}}_{\mu_{c}}=0$.

Finally, by \cite[Rem.~1.3 (1)]{KessNieZhu}, for all $c$ we have the equality $\overline{\mathsf{D}}_r(\mu_{c})=rq_r/(1-q_r)$.
\end{proof}

\begin{figure}
\center{\begin{tikzpicture}[scale=01, every node/.style={transform shape},line cap=round,line join=round,>=triangle 45,x=1cm,y=1cm] \begin{axis}[ x=2.7cm,y=2.7cm, axis lines=middle, axis line style={very thick},ymajorgrids=false, xmajorgrids=false, grid style={thick,densely dotted,black!20}, xlabel= {$q$}, ylabel= {$\beta_{\mu_{c}} (q)$}, xmin=-0.3 , xmax=1.5 , ymin=-0.3, ymax=1.3,x tick style={thick, color=black}, xtick={0, .425,1},xticklabels = {0,$q_r(\mu_{c})$,1}, y tick style={thick, color=black}, ytick={0,1/1.15,1},yticklabels = {0,$\mathsf{D}_{r}(\mu_{c})$,1}] \clip(-0.5,-0.3) rectangle (4,4); 
\draw[line width=1pt,smooth,samples=180,domain=-0.0:1.4] plot(\x,{log10(0.8^((\x))+0.2^((\x)))/log10(2)}); 
 \draw [line width=01pt,dotted, domain=-0.05 :1.4] plot(\x,{1/(1.15)*(1-\x)});
\draw [line width=01pt,dashed, domain=-0.05 :1.0] plot(\x,{1.2*\x});

\node[circle,draw] (c) at (2.48 ,0 ){\,};
 
\draw [line width=.7pt,dotted, gray] (0.425 ,0.)--(0.425,1); 
 
\end{axis} 
\end{tikzpicture}}
    \caption{\label{Dr}Geometric reconstruction of the quantization dimension $\mathsf{D}_{r}(\mu_{c})$ from the $L^{q}$-spectrum of $\mu_{c}$ with $c=1/2$ and   $r=1.2$, which is the slope of the dashed line through the origin.
    }
\end{figure}
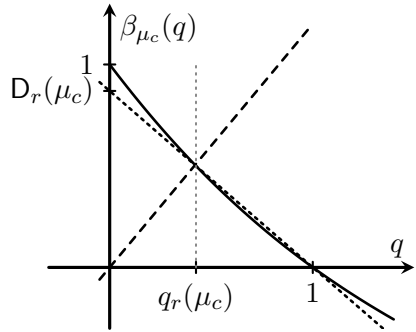

\section{Estimates on Birkhoff sums and combinatorial results}\label{sec:Technical}

In general, we are concerned with expressions of the form $\psi_n^c(x)$. On cylinders $\langle w \rangle$ with $|w| = r \geqslant n$, this takes the form
\[
\psi^c_n(x) = \sum_{k = 0}^{n-1} \psi^c(x_k),
\]
with $x_k \in \langle w_{[k,r]} \rangle$ for all $k$. Both the size and the variation of $\psi^c$ on $\langle w_{[k,r]} \rangle$ depend on how close this cylinder is to the singularity $\breve{c}$. This problem is not independent for different values of $k$ since being close to $\breve{c}$ determines a certain prefix of $w_{[k,r]}$. We are therefore led to investigate the combinatorial structure of $\breve{c}$ and words $w$. For some applications, we wish to find a modification of $w$ that allows a certain control of $\psi_n^c$ on $\langle w \rangle$. This falls basically into two classes: First, we try to find words $v$ that restrict the length of prefixes of $\breve{c}$ that appear in $wv$ (possibly starting in $w$). Another goal will be to concatenate words $w$ and $w'$ by relatively short words $v$ such that $u = w v w'$ has the property that $\langle u_{[k,|u|]} \rangle$ stays sufficiently far away from the singularity for $k \leqslant |wv|$. 
We split this section into two parts. In the first part, we work purely combinatorially, and in the second part we spell out the implications for the corresponding Birkhoff sums.

\subsection{Combinatorics of words}

For convenience of notation, we may think of a word $w$ as a dyadic representation of an integer such that adding or subtracting a number is well-defined, as long as the resulting number is non-negative. For the moment, let us assume that the singularity position $\breve{c}$ has a unique expansion in base $2$.
Again, we denote by $\breve{c}_{[m,n]}$ the finite word of the $m$th digit to the $n$th digit of $\breve{c}$.
We introduce
\[
\mc G_n = \mc G_n(\breve{c}) = \breve{c}_{[1,n]} + \{ 0,1,-1 \}; \quad \mc G =\mc G (\breve{c})= \bigcup_{n \geqslant 1} \mc G_n,
\]
together with the conventions $0^n - 1 = 1^n$ and $1^n + 1 = 0^n$.
That is, $\{\langle w \rangle\}_{w \in \mc G_n}$ consists of the length-$n$ cylinder set containing $\breve{c}$, along with its right and left neighbor. By this definition, 
\begin{align*}
x_{[1,n]} \in \mc G_n & \implies \rho(x,\breve{c}) \leqslant 2^{-n+1},
\\x_{[1,n]} \notin \mc G_n & \implies \rho(x,\breve{c}) \geqslant 2^{-n}.
\end{align*}

It is straightforward to see that, given $w \in \mc G_n$ and $m \leqslant n$, we have $w_{[1,m]} \in \mc G_m$. Conversely,
\[
w_{[1,m]} \notin \mc G_m \implies w_{[1,n]} \notin \mc G_n,
\]
whenever $m \leqslant n \leqslant |w|$.
If $\breve{c}$ is a dyadic point, let $(\breve{c}_n)_{n\in\N}$ and $(\breve{c}^\ast_n)_{n \in \N}$ be its expansions in base $2$. In this case we set 
\[
\mc G_n = \left\{\breve{c}_{[1,n]}, \breve{c}^\ast_{[1,n]} \right\},
\]
for all $n$ with $\breve{c}_n \neq \breve{c}^\ast_n$,
without altering any of the decisive properties above.

In the symbolic representation, self-returns correspond to the repetition of blocks. Given a word $u$ of length $m$,
and some $r = n + k/m$, with $n \in \N$ and $0\leqslant k \leqslant m-1$, the rational power $u^r$ is defined as
\[
u^r = u^n u_{[1,k]},
\]
where we denote by $u^n:=u^{n-1}u$.
If, $w = u^r$ for some $r \in \Q$, we say that $|u|$ is \emph{a} period of $w$. \emph{The} period of $w$ is the shortest such value.

The following are standard tools in combinatorics on words and are given without proofs; for the first statement compare \cite[Lem.~1]{ABCD}.

\begin{lemma}
\label{LEM:period-division}
Let $p$ be the period of a word $w$ and assume that $q$ is a period of $w$ with $q \leqslant |w|/2$. Then, $q$ is divisible by $p$.
\end{lemma}

\begin{lemma}
\label{LEM:word-power}
Assume that $w = u w_p$, where $w_p$ is a prefix of $w$. Then, $w$ is a (rational) power of the word $u$.
\end{lemma}

For the following, we write $a \triangleleft_p b$ if $a$ is a prefix of $b$ (this is true by convention if $a$ is the empty word).

Next, we introduce the concepts of hitting time and hitting depth.

\begin{definition}
Given a word $w$ and $j \in \N$ with $j < |w|$, the \emph{$j$-hitting times} (with respect to $w$) are given by
\[
\mc H_j(w) = \left\{ k \leqslant j: w_{[k,j]} \in \mc G \mbox{ and } w_{[k,j+1]} \notin \mc G \right\}.
\]
Likewise, the \emph{(full) hitting times} (with respect to $w$) are
\[
\mc H(w) = \left\{ k \leqslant |w|: w_{[k,|w|]} \in \mc G \right\}.
\]
\end{definition}
Note that every $k \in [1,|w|]$ is either a full hitting time or a $j$-hitting time for some unique $j < |w|$. That is, the sets $\mc H(w)$ and $\mc H_j(w)$ with $1\leqslant j < n$ form a partition of $[1,|w|]$.
If $k \in \mc H_j(w)$, we call $j$ the \emph{hitting depth} of $k$ and say that $k$ has \emph{full hitting depth} if $k \in \mc H(w)$.

We will see later in Section \ref{subsec: Variation psi} that the concept of a $j$-hitting time is useful because it allows us to find bounds for $\psi^c$ and its variation on the corresponding cylinders.

Moreover, the self-return properties of $\breve{c}$ impose some structure on the possible sets of \mbox{($j$-)}hit\-ting times. In order to analyze these self-returns in a purely combinatorial manner, we consider some (symbolic) variants of hitting times. Choose an arbitrary symbolic representation of $\breve{c}$, if it is not unique, in the following. We set
\[
\widetilde{\mc H}_j(w) = \left\{ k \leqslant j+1 : j \mbox{ is maximal with } w_{[k,j]} \triangleleft_p \breve{c}  \right\}
\]
 for $j<|w|$, 
 and
\[
\widetilde{\mc H} (w) = \left\{k \leqslant |w|: w_{[k,n]} \triangleleft_p \breve{c} \right\}.
\]
In analogy to the earlier discussion, $j$ is the \emph{symbolic hitting depth} of $k$ if $k \in \widetilde{\mc H}_j(w)$, and $k$ has \emph{full symbolic hitting depth} if $k \in \widetilde{\mc H} (w)$. In general, the hitting depth of $k$ will be larger than its symbolic variant. However, blocks of the form $01$ or $10$ form an obstacle for the additional depth.
\begin{lemma}
\label{LEM:hitting-depths-relations}
Let $w \in \Sigma^n$ and $r<n-1$ be such that $w_{[r+1,r+2]} \in \{01,10 \}$. If $k \in \widetilde{\mc H}_j(w)$ for some $j < r $, then the hitting depth of $k$ is at most $r+1$.
\end{lemma}

\begin{proof}
Assume $w_{[r+1,r+2]} = 01$, the other case is completely analogous. Since $j < r$, the word $u=w_{[k,r]}$ is not a prefix of $\breve{c}$. But this implies that each of the neighbors of $w_{[k,r+2]}$, given by $u00$ and $u10$, is also not a prefix of $\breve{c}$. This precludes $w_{[k,r+2]}$ from being in $\mc G$ and hence the hitting depth of $k$ must be smaller than $r+2$.
\end{proof}

Similarly, there is a relation between symbolic and non-symbolic hitting depth if the length of blocks of the type $0^m$ and $1^m$ is bounded in $\breve{c}$. We spell that out in the particular case that $\breve{c} \neq 0$ is periodic.

\begin{lemma}
\label{LEM:periodic-hitting-depth}
Assume that $\breve{c} \neq 0$ has a periodic expansion of the form $\breve{c} = u^\infty$ with $|u| = p > 1$. Then, if $w \in \Sigma^n$ and $k \in \widetilde{\mc H}_j(w)$ for some $j < n-p$, the hitting depth of $k$ is at most $j+p$.
\end{lemma}

\begin{proof}
By assumption, $w' = w_{[k,j]} \triangleleft \breve{c}$, but $w_{[k,j+1]} = w'a$ is not a prefix of $\breve{c}$. Again, denote by $\underline{a}$ the flip of $a$
and let $v$ with $|v| = p$ be the unique word such that $w'\underline{a}v \triangleleft \breve{c}$. Since $\breve{c} \neq 0$ by assumption, we certainly have $v \notin \left\{0^p, 1^p\right\}$. Hence, the neighbors of $w'\underline{a}v$ are of the form $w'\underline{a}v'$ for some $v' \in \Sigma^p$. Since $a\neq \underline{a}$, we therefore conclude that $w_{[k,j+1+p]} = w'a v'' \notin \mc G$ and it follows that the hitting depth of $k$ can be at most $j+p$.
\end{proof}

Given a word $w \in \Sigma^n$, in view of Corollary~\ref{COR:mu-on-double-word}, it seems desirable to find a word $v$ such that $\psi^c_n$ avoids the singularities on $\langle wv\rangle$ in order to get a meaningful lower bound on $\mu(\langle w \rangle)$. A priori, it is not clear how long $v$ needs to be in order to fulfill this purpose. Since $\langle w \rangle$ can contain up to $n$ singularities, slicing $\langle w \rangle$ into $n$ intervals of the same length should do the job. This would require $|v|$ to be of the order $\log_2 n$. Unfortunately, this bound is not sufficient for our purposes. The next result shows that we can do considerably better. Intuitively, this is because the occurrence of many singularities of $\psi_n^c$ in $\langle w \rangle$ requires them to be appropriately clustered, as we will see in the proof.

\begin{lemma}
\label{LEM:hitting-word-structure}
Let $w \in \Sigma^n$ and let $\mc W(w) = \left\{ w_{[k,n]} : k \in \widetilde{\mc H}(w) \right\}$ be the set of full hitting words. Then there are $s = s(w) \in \N_0$, pairwise disjoint words $u_1, \ldots, u_s$, and multiplicities 
$N_1, \ldots, N_s \in \N$ with the following properties. Let $v_i = u_i^{N_i}$ for each $1\leqslant i \leqslant s$. Then,

\begin{enumerate}
\item\label{en: 1} $\mc W(w) = \bigcup_{i=1}^s \mc W_i$, where
\[
\mc W_i = \left\{ u_i^j v_{i-1} \cdots v_1: 1 \leqslant j \leqslant N_{i} \right\}.
\]
\item\label{en: 2} Each element of $\mc W_i$ is a rational power of $u_i$.
\item\label{en: 3} $|u_{i+1}| \geqslant \max_{w \in \mc W_i}|w|/2$ for all $1\leqslant i < s$.
\end{enumerate}
In particular, $\# \widetilde{\mc H}(w) = \sum_{i=1}^s N_i$ and $s \leqslant \lceil \log_{3/2} n \rceil$.
\end{lemma}

Note that in the last lemma, to ease notation, we have used $u_i$ and $w_i$ to refer to a specific word, not the $i$th letter of a word.

\begin{proof}
Note that all elements of $\mc W(w)$ are suffixes of each other and prefixes of $\breve{c}$. This imposes some additional structure. 
In particular, for every $w',w''  \in \mc W$ with $|w'| < |w''|$, the word $w'$ is both a suffix and a prefix of $w''$.
\begin{figure}[h]
 \begin{tikzpicture}[x=1cm,y=0.4cm]
    \draw (0,0) node[anchor=south] {$\breve{c}_1$};
    \draw (1.5,0) node[anchor=south] {$\breve{c}_2$};
    \draw (3,0) node[anchor=south] {$\breve{c}_3$};
    \draw (4.5,0) node[anchor=south] {$\ldots$};
    \draw (6,0) node[anchor=south] {$\breve{c}_{n-k-1}$};
    \draw (7.5,0) node[anchor=south] {$\breve{c}_{n-k}$};
    \draw (9,0) node[anchor=south] {$\breve{c}_{n-k+1}$};
    \draw (10.5,0) node[anchor=south] {$\breve{c}_{n-k+2}$};
        \draw (12,0) node[anchor=south] {$\ldots$};
        \draw [decorate,     decoration = {brace}] (0,1.6) --  (9,1.6);
        \draw (4.5,2) node[anchor=south] {$w_{[k,n]}$};
        \draw [decorate,     decoration = {brace}] (3,-0.4) --  (0,-0.4);
        \draw [decorate,     decoration = {brace}] (9,-0.4) --  (5.6,-0.4);
        \draw (1.5,-2) node[anchor=south] {$w_{[n-2,n]}$};
        \draw (7.4,-2) node[anchor=south] {$w_{[n-2,n]}$};%
  \end{tikzpicture}
   \caption{Example of $w'=w_{[n-2,n]}$, $w''=w_{[k,n]}$ with $w',w''\in \mc W$ and  $|w'|<|w''|$. The word $w'$ is both a prefix and suffix of $w''$.}
\end{figure}

Ordered according to their length, we denote the elements of $\mc W(w)$ by $w_1,w_2,\ldots,w_r$ with $r = \#  \widetilde{\mc H}(w)$.
Since $w_j$ is always a suffix of $w_{j+1}$, we can write 
\begin{equation*}
w_{j+1} =  \tilde{u}_{j+1} w_j
\end{equation*}
for all $1 \leqslant j \leqslant r-1$. For notational convenience, we also set $ \tilde{u}_1 = w_1$. 
Due to Lemma~\ref{LEM:word-power}, it follows that each $w_j$ is a (rational) power of $ \tilde{u}_j$.
We may now set $u_1=\tilde{u}_1$ and if $u_{j}=\tilde{u}_\ell$ then we set $u_{j+1}=\tilde{u}_{\ell+i}$, where $i$ is the minimal term such that $\tilde{u}_{\ell+i}\neq \tilde{u}_{\ell}$.

We observe that the (smallest) period of $w_j$ is $|\tilde{u}_j|$.
If this statement were not true, we could write $w_j =  u^p$ for some $|u|< |\tilde{u}_j|$ and $p > 1$. 
In this case, it would follow that $w' = u^{p-1} \in \mc W(w)$ with $|w_{j-1}| < |w'| < |w_j|$. This would  contradict the definition of $w_j$ as the lengthwise successor of $w_{j-1}$.

This also implies that $\tilde{u}_j$ is a prefix of $\tilde{u}_{j+1}$. This is because $w_j$, being a prefix of $w_{j+1}$, can also be written as a power of $\tilde{u}_{j+1}$, which cannot be larger than $\tilde{u}_j$ due to the minimality of the period $|\tilde{u}_j|$ of $w_j$.

In particular, we obtain that $u_j$ is a strict prefix of $u_{j+1}$ for all $1 \leqslant j \leqslant s-1$, implying that all of these words are pairwise disjoint.

The above construction already gives the structure claimed in \eqref{en: 1} and \eqref{en: 2}.

Finally, if $|\tilde{u}_{j+1}| \leqslant |w_j|/2$ we can use Lemma~\ref{LEM:period-division} and the statement that the smallest period of $w_j$ is $|\tilde{u}_j|$ to conclude that $|\tilde{u}_{j+1}|$ is divisible by $|\tilde{u}_j|$. Hence, $\tilde{u}_{j+1}$ is an integer power of $\tilde{u}_j$ and the same follows for $w_{j+1}$. Unless $\tilde{u}_j = \tilde{u}_{j+1}$ this is in contradiction to the minimality of $|\tilde{u}_{j+1}|$ as the period of $w_{j+1}$. 

Item \eqref{en: 3} implies that the growth of $|w_j|$ alternates between phases of linear growth (due to periodic continuation) and exponential growth with a factor of at least $3/2$. 

The cardinality $s$ of the set $\mc U = \left\{ \tilde{u}_1, \ldots, \tilde{u}_r \right\}$ is in general smaller than $r$ because several of the words might be equal. Ordering the elements of $\mc U$ according to their length we obtain the words $u_1, \ldots, u_s$. By the observations above, we can find a strictly increasing sequence $(j_k)_{1\leqslant k \leqslant s}$ of non-negative integers such that 
\[
w_{j+1} = u_k w_j
\]
for all $j_k \leqslant j < j_{k+1}$ (with the convention that $w^0$ is the empty word). By \eqref{en: 3}, we see that $|w_{j_{k}+1}| > 3/2 |w_{j_k}|$ for all $1\leqslant k \leqslant s$. Iterating this relation yields
$
|w_{j_s}| > (3/2)^{s-1}.
$
Since $|w_{j_s}| \leqslant n$, we obtain
\begin{equation*}
s \leqslant \lceil \log_{3/2} n \rceil.
\end{equation*}
\end{proof}

In the notation of Lemma~\ref{LEM:hitting-word-structure}, we interpret $s(w)$ as the number of \emph{independent} hitting times in $\widetilde{\mc H}(w)$. We show that the length of $v$ required to terminate singularities starting in $w$ can be bound in terms of $s(w)$.

\begin{lemma}
\label{LEM:closing-extension-II}
There is a constant $C>0$ with the following property. Let $w \in \Sigma^n$ and $s = s(w)$ the number of independent hitting times in $\widetilde{\mc H}(w)$. Then there is a word $v$ such that
\begin{itemize}
\item $w_{[k,n]}v \notin \mc G$ for all $1\leqslant k \leqslant n$
\item $|v| \leqslant \max \left\{C, 2 \log_2 s \right\}$.
\end{itemize}
In particular, there is $n_0 \in \N$ such that for all $n \geqslant n_0$, we have $|v| \leqslant 2 \log_2 \log_{3/2} n$.
\end{lemma}

\begin{proof}
First, let us note that whenever $u \in \Sigma^m$ is not a prefix of $\breve{c}$ and $v' \in \{01,10\}$, then $u v' \notin \mc G$ due to Lemma~\ref{LEM:hitting-depths-relations}. 
Hence, choosing the first two letters of $v$, $v_1 v_2 \in \{01,10 \}$, the first property is automatic, unless $w_{[k,n]} \triangleleft_p \breve{c}$.
We therefore restrict our attention to the sets $\mc H(w)$ and $\mc W(w)$ with the latter one defined as in Lemma \ref{LEM:hitting-word-structure}.

Let $L_n \in \N$ be an integer length to be determined later and let
\[
\widetilde{ \mc G} = \widetilde{ \mc G}(\mc W(w)) := \left\{ v \in 01\Sigma^{L_n} : wv \in \mc G \mbox{ for some } w \in \mc W(w) \right\}
\]
be the forbidden extensions of words in $\mc W(w)$.
Provided that the cardinality of $\widetilde{\mc G}$ is strictly smaller than $\# 01 \Sigma^{L_n} = 2^{ L_n}$, there is some $v \in 01\Sigma^{L_n}$ such that $wv \notin \mc G$ for all $w \in \mc W$. Due to the prefix $01$, it follows in fact for every $v \in 01\Sigma^{L_n} \setminus \widetilde{\mc G}$ that
\begin{equation}
\label{EQ:not-in-F}
w_{[k,n]} v \notin \mc G,
\end{equation}
for all $1\leqslant k \leqslant n$, as we discussed at the beginning of the proof. It remains to show the existence of such an $L_n$ that also satisfies the bound $L_n + 2 \leqslant \log_2 s $. 
Let us try to find an appropriate estimate for the cardinality of $\widetilde{\mc G}$. Given a fixed word $w \in \mc W$, all words of the form $wv$ with $v \in 01\Sigma^{L_n}$ have the same length and hence, there can be at most $3$ such words $v$ with $wv \in \mc G$. This gives the immediate bound 
\begin{equation}
\label{EQ:F-trivial-bound}
\# \widetilde{\mc G}(\mc W)(w) \leqslant 3 \#\mc W(w). 
\end{equation}
This is however not quite sufficient for our purposes. To obtain a finer estimate, we consider the disjoint decomposition $\mc W(w) = \bigcup_{1\leqslant k \leqslant s} \mc W_k$ with $\mc W_k$ as in Lemma \ref{LEM:hitting-word-structure}.
Note that all words in $\mc W_k$ differ only by some integer power of $u_k$ at the beginning of the word. 
Accordingly, we may decompose
\[
\widetilde{\mc G}(\mc W(w)) = \bigcup_{1 \leqslant k \leqslant s}  \widetilde{\mc G}(\mc W_k). 
\]
If $\# \mc W_k \leqslant {L_n}$, the bound $\# \widetilde{\mc G}(\mc W_k) \leqslant 3 { L_n}$ is immediate; compare \eqref{EQ:F-trivial-bound}. In general, the elements in $\mc W_k$ are precisely those of the form
\[
w'_m = (u_k)^{m+q}
\]
for some $q \in \Q$ and $1\leqslant m \leqslant m_0 := \# \mc W_k$. In particular, $\breve{c}$ starts with the word
\[
\breve{c}_p = (u_k)^{m_0 + q}.
\]
If $m_0 - m \geqslant {L_n} + 2$, it follows in particular that $|\breve{c}_p| - |w'_m| \geqslant {L_n} + 2$ and hence 
\[
|w'_m v| \leqslant |\breve{c}_p|
\]
for all $v \in 01\Sigma^{L_n}$. Therefore, $w'_m v$ being a prefix of $\breve{c}$ requires that $w'_m v$ is a periodic repetition of $u'_k$. This singles out a word $v'$ uniquely, which depends on $q$ and $u_k$ but \emph{not} on $m$. In this case, $w'_m v \in \mc G$ precisely if $w'_m v$ is in the `neighborhood' of $w'_m v'$, that is, if
\[
w'_m v \in w'_m v' + \{0,1,-1\},
\]
which requires
\[
v \in v' + \{0,1,-1\}.
\]
But this means that $\widetilde{\mc G}(\left\{ w^m \right\})$ is the same set (of at most $3$ elements) as long as $m \leqslant m_0 -({L_n}+2)$. Hence, we obtain
\[
\#  \widetilde{\mc G}(\mc W_k) \leqslant 3 ({L_n} + 2)
\]
in all cases. Due to the bound on $s$ in Lemma \ref{LEM:hitting-word-structure}, this implies
\[
\# \widetilde{\mc G}(\mc W) \leqslant \sum_{k=1}^s \# \widetilde{\mc G}(\mc W_k) \leqslant 3 ({L_n} + 2) \lceil  s \rceil.
\]
Now, let us choose
\begin{equation}
\label{EQ:c-n-definition}
{L_n} = \lfloor 2 \log_2  s \rfloor -2.
\end{equation}
Then, there is $n_0 \in \N$ such that for $n \geqslant n_0$, we have
\[
\# \widetilde{\mc G}(\mc W) \leqslant  s^2.
\] 
Therefore, $\log_2 \# \widetilde{\mc G}(\mc W) \leqslant 2 \log_2  s < {L_n} $, implying
\[
\# \widetilde{\mc G}(\mc W) < 2^{L_n}.
\]
Hence, there exists a word $v \in 01\Sigma^{ L_n}\setminus \widetilde{\mc G}(\mc W)$. Due to \eqref{EQ:not-in-F} and \eqref{EQ:c-n-definition} this word $v$ has the desired properties.
The bound $s \leqslant \lceil \log_{3/2} n \rceil$ was given in Lemma~\ref{LEM:hitting-word-structure}.
\end{proof}

For later use, we record that a large number of hitting times (relative to the word length) requires the number of independent hitting times to be small.

\begin{lemma}
\label{LEM:s-hitting-relation}
Let $w \in \Sigma^n$ and let $\delta \in (0,1)$ be such that $\delta n = \# \widetilde{\mc H}(w) - s$, with $s$ as in Lemma~\ref{LEM:hitting-word-structure}. Then we have $s \leqslant 2 + \log_{3/2} (2\delta^{-1})$.
\end{lemma}

\begin{proof}
In the notation of Lemma~\ref{LEM:hitting-word-structure}, let $m_i = |u_i|$ for all $1\leqslant i \leqslant s$. Then, for every $k$, the largest word in $\mc W_k$ has length $\ell_k = \sum_{i=1}^k N_i m_i$. Furthermore, $m_{i+1} \geqslant \ell_i/2$, and thus
\[
\ell_{i+1} = N_{i+1} m_{i+1} + \ell_i
\geqslant \frac{1}{2} (N_{i+1} + 2) \ell_i,
\]
for all $1\leqslant i < s$. Iterating this relation, and using $\ell_1 = N_1 m_1 \geqslant N_1$ yields
\begin{equation}
\label{EQ:ell-s-product}
\ell_s \geqslant 2^{-s+1} N_1 \prod_{i=2}^s (N_i + 2).
\end{equation}
Note that for $a_1,a_2\geqslant 0$ and $k \in \N$, we have
$
(a_1 + k)(a_2 + k) \geqslant k(a_1 + a_2 + k)
$
and inductively
\[
\prod_{i=1}^r (a_i + k) \geqslant k^{r-1} \biggl( k+\sum_{i=1}^r a_i  \biggr),
\]
given $a_i \geqslant 0$.
Applying this with $a_i = N_{i+1} -1$ and $k=3$ gives
\[
\prod_{i=2}^s (N_i + 2) \geqslant 3^{s-2} \biggl( 3 + \sum_{i=2}^s (N_i - 1) \biggr).
\]
Moreover, 
from \eqref{EQ:ell-s-product}, the fact that $a N_1\geq a+N_1-1$ for all $a\geq 1$ and $N_1\in\mathbb{N}$ and $\# \widetilde{\mc H}(w) = \sum_{i=1}^s N_i$ we get
\[
2^{s-1} \ell_s \geqslant N_1 \prod_{i=2}^s (N_i + 2) \geqslant 3^{s-2}\biggl( 3 + \sum_{i=2}^s (N_i - 1) \biggr)N_1 \geqslant 3^{s-2}\biggl( 3 + \sum_{i=1}^s (N_i - 1) \biggr) \geqslant 3^{s-2} \delta n.
\]
Since the length $\ell_s$ of the largest word in $\mc W(w)$ is bounded by $n$, we therefore conclude that 
$
2n \geqslant (3/2)^{s-2} \delta n,
$
which gives the desired bound for $s$.
\end{proof}

So far we have worked on finding words $v$ for a given word $w$ such that $wv$ does not have large prefixes of $\breve{c}$ (or neighbouring words) starting in $w$. For some applications it will be essential to obtain a similar property that also excludes large prefixes of $\breve{c}$ starting in $v$. This is the general idea of the results below.
In the following, if $\breve{c}$ is a dyadic point, arbitrarily choose one of its dyadic representations.

\begin{lemma}
\label{LEM:sliding-out-block-I}
Let $v \in \Sigma^n$. Then, there is a word $w = v v'$ and $ab \in \Sigma^2$ such that $w_{[k,|w|]}ab$ is not a prefix of $\breve{c}$ for all $1\leqslant k \leqslant |w|$.
Furthermore, we can choose $|v'| < |v|$ if $n$ is large enough, with $|v'|$ being uniform for all $v \in \Sigma^n$.
\end{lemma}

\begin{proof}
We build the word $v'$ inductively. First, note that if $n_1 \in \N$ is such that $2^{n_1} > n$, there is certainly a word $u_1 \in \Sigma^{n_1}$ such that $v_{[k,n]} u_1$ is not a prefix of $\breve{c}$ for all $1\leqslant k \leqslant n$. Likewise, if $2^{n_2} > n_1$, there is $u_2 \in \Sigma^{n_2}$ such that $(u_1)_{[k,n_1]} u_2$ is not a prefix of $\breve{c}$ for $1\leqslant k \leqslant n_1$. We hence find a finite sequence of words $u_i \in \Sigma^{n_i}$ and $n_{i+1} > \log_2 n_i$ with $1\leqslant i \leqslant r$ such that $v' = u_1 \cdots u_{r-1}$ and $u_r$ have the property that $w= vv'$ satisfies that $w_{[k,|w|]} u_r$ is not a prefix of $\breve{c}$ for all $k$. Note that as long as $n_i > 2$, there is a number $n_{i+1} < n_i$ with the desired property $n_{i+1} > \log_2 n_i$. Therefore, we can assume that the sequence $n_i$ is monotonically decreasing and that $n_r = 2$. 
Since the numbers $n_i$ can be chosen independently of the particular word $v$ in $\Sigma^n$, the same holds for the length $|v'|$.
Furthermore, since the iterated logarithms decay very fast, it is straightforward to see that $|v'|$ can be chosen smaller than $n$, provided $n$ is large enough. 
\end{proof}

\begin{proof}[Proof of Lemma~\ref{LEM:sliding-out-no-prefix}]
By Lemma~\ref{LEM:closing-extension-II} 
we can choose a word $\widetilde{v}$ such that $w_{[k,n]}\widetilde{v}\notin \mc G$, for all $1\leqslant k\leqslant n$ with $|\widetilde{v}|\leqslant 2\log_2\log_{3/2} n$ if $n$ is large enough. Moreover, by Lemma~\ref{LEM:sliding-out-block-I} we can choose a word $\widetilde{v}'$ and $ab \in \Sigma^2$ such that setting $v'=\widetilde{v}\widetilde{v}'$ we have 
$v'_{[k,|v'|}ab$ is not a prefix of $\breve{c}$ for all $1\leqslant k\leqslant |v'|$ and
$|\widetilde{v}'|\leqslant |\widetilde{v}|$ if $|\widetilde{v}|$ is large enough.
Combining these two observations, we have that $v'$ and $ab$ are of the form that
$w'=vv'$ satisfies that $w'_{[k,|w'|]}ab$ is not a prefix of $\breve{c}$ for all $1\leq k\leq |w'|$ and we can assume that $|v'| \leqslant 4 \log_2 \log_{3/2}n$ for large enough $n$.
In the following, we give an explicit extension $v''$ such that no suffix of $abv''$ is a prefix of $\breve{c}$. We do this via a case distinction that covers all possible $\breve{c} \in (0,1/2)$. Note that the values $0,1/2$ are excluded by assumption and that we can extend the argument to all $\breve{c} \in (0,1)\setminus\{1/2\}$ by symmetry.
First, let us assume that $\breve{c}$ is not a dyadic point. Then, a possible set of choices is provided in the list below.

\begin{center}
    \begin{tabular}{ | l | l | l | l | l |}
    \hline
    $\breve{c}$ & $000\ldots$ & $0010\ldots$ & $0011\ldots$ & $01^n0\ldots$  \\ \hline
    $abv''$ & $ab11$ & $ab11$ &  $ab101$ & $ab1^{n+1}$  \\ 
    \hline
    \end{tabular}
\end{center}
If $\breve{c} \in (0,1/2)$ is a dyadic point, it is of the form $\breve{c} = p 0\overline{1} = p 1\overline{0}$ for some non-empty word $p$. In the case that $p$ starts with any of the prefixes considered in the table above, we can use the choice of $v''$. All other cases are considered in the table below.

\begin{center}
    \begin{tabular}{ | l | l | l | l | l |}
    \hline
    $p$ & $0$ & $00$ & $001$ & $01^n$  \\ \hline
    $abv''$ & $ab1011$ & $ab11$ &  $ab111$ & $ab1^{n+2}$  \\ 
    \hline
    \end{tabular}
\end{center}
This completes the proof.
\end{proof}

\begin{lemma}
\label{LEM:sliding-out-block-II}
Let $c \neq 0$. Then, there are lengths $\ell,\ell' \in \N$ such that for all $ab \in \Sigma^2$ and every word $v \in \Sigma^{\ell}$, there is a word $v' \in \Sigma^{\ell'}$ and $u \in \{01,10\}$ such that $w = ab u v'$ satisfies $w_{[k,|w|]} v \notin \mc G$ for all $1\leqslant k \leqslant |w|$. 
\end{lemma}

\begin{proof}
We prove this statement via a tedious case distinction, based on the location of $\breve{c}$. For convenience of notation, we let $\mc W = \{ w_{[k,|w|]}v: 1\leqslant k \leqslant |w| \}$ be all suffixes of $wv$ that overlap the word $w$. 
By symmetry, it is enough to consider $0\leqslant \breve{c} \leqslant 1/2$.

Let $\breve{c}_{[1,4]} = 0000$. Then, choose $\ell=3$ and $\ell'=1$. If $v=0w'$ for some $w'\in \Sigma^2$, we set $u=10$ and $v'=1$. In this case $w=ab101$, and 
\[
\mc W = \{ab1010w',b1010w',1010w',010w',10w' \}.
\]
 None of these words has a prefix in $\mc G_4 = \{0000,0001,1111 \}$ and hence, none of the words can be in $\mc G$. The claim follows for this special case. Likewise, if $v=1w'$, we set $u=10$ and $v' = 0$, instead, giving 
 \[
 \mc W = \{ab1001w', b1001w', 1001w', 001w', 01w' \},
 \]
and again, none of them has a prefix in $\mc G_4$. 

For $\breve{c}_{[1,4]} \in \{0001,0010,0011,0100,0101,0110 \}$, it suffices to consider $\ell \leqslant 6$ and $\mc G_{\ell+1}$, and no case distinction needs to be made based on the choice of $v \in \Sigma^\ell$. We give some possible choices of $u$ and $v'$, along with the resulting word $w v$ (where $w=abuv'$) in the table below. We also provide a set $\mc G_{\ell+1}' \supset \mc G_{\ell+1}$ of all words that could be in $\mc G_{\ell+1}$, given the choice of $\breve{c}_{[1,4]}$. For notational convenience, we use $x$ as a place holder for all possible words of the appropriate length. It can be verified by direct inspection, that none of the relevant suffixes starts with a word in $\mc G'_{\ell+1}$, and hence none of them is in $\mc G$.
\begin{center}
    \begin{tabular}{ | l | l | l | l | l | l | l |}
    \hline
    $\breve{c}_{[1,4]}$ & $u$ & $\ell'$ & $v'$ & $wv$ & $\ell$ & $\mc G_{\ell+1}'$ \\ \hline
    $0001$ & $10$ & $1$ &  $1$ & $ab101v$ & $4$ & $\{0001x,00001,00100 \}$  \\ 
     $0010$ & $01$ & $4$ &  $1101$ & $ab011101v$ & $6$ & $\{0010x,0001111,0011000 \}$  \\ 
     $0011$ & $10$ & $4$ & $1011$ & $ab101011v$ & $5$ & $\{0011x,001011,010000 \}$  \\ 
     $0100$ & $10$ & $2$ &  $11$ & $ab1011v$ & $4$ & $\{0100x,00111,01010 \}$  \\ 
     $0101$ & $10$ & $5$ &  $00111$ & $ab1000111v$ & $6$ & $\{0101x,0100111,0110000 \}$  \\ 
     $0110$ & $01$ & $6$ &  $001111$ & $ab01001111v$ & $4$ & $\{0110x,01011,01110 \}$  \\ 
    \hline
    \end{tabular}
\end{center}
It remains to treat the case that $\breve{c}_{[1,4]} = 0111$. Since we excluded the case $\breve{c} = 1/2$ by assumption, there is certainly an integer $m \geqslant 3$ such that $\breve{c}$ starts with the word $01^m0$. We choose $\ell = m+2$ and note that $\mc G'_{\ell+1} = \{ 01^m 0 x, 01^{m-1}011, 01^{m+1}0 : x \in \Sigma^2 \}$ certainly contains all words in $\mc G_{\ell+1}$. With $u = 01$, $\ell'= m+1$ and $v'=1^{m+1}$, we get that $wv = ab01^{m+2}v$. Clearly, none of the relevant suffixes of $wv$ starts with a word in $\mc G'_{\ell+1}$. This completes the last remaining case.
\end{proof}

\subsection{Variation of $\psi_n^c$ on cylinders}\label{subsec: Variation psi}

It will be important to assess the distance of cylinder sets to the singularity point $\breve{c}$.
This is because $\psi^c$ and the variations of $\psi^c$ can be estimated in terms of the Euclidean distance $\rho$ to $\breve{c}$.

\begin{lemma}
\label{LEM:psi-distance-bound}
For every $x$, we have
\[
2 \log (2 \rho(x,\breve{c})) \leqslant \psi^c(x)
\leqslant 2 \log (\pi \rho(x, \breve{c})).
\]
Furthermore, if $x,y \in \mathbb{T} \setminus B_r(\breve{c})$ for some $r>0$, we have the Lipschitz estimate
\[
|\psi^c(x) - \psi^c(y)| \leqslant \frac{2}{r} \rho(x,y).
\]
\end{lemma}

\begin{proof}
Recall that $\psi^c(x) = 2 \log | \cos(\pi (x-c))|$.
Regarding $y = x-\breve{c}= x- c+1/2$ as an element on the torus, we observe that 
\[
2\rho(y,0) \leqslant |\sin (\pi y)| \leqslant \pi \rho(y,0).
\]
This is straightforward to verify for $y\leqslant 1/2$ and the general case follows by symmetry. Since $\rho(y,0) = \rho(x,\breve{c})$, 
the first claim follows. The second claim follows easily from the mean value theorem; compare \cite[Lem.~4.1]{BGKS} for details.
\end{proof}

We use these bounds to obtain estimates on cylinders that correspond to $j$-hitting times.

\begin{lemma}
\label{LEM:psi-on-hitting}
Let $w \in \Sigma^n$ and $k \in \mc H_j(w)$ for some $j < n$. If $x,y \in \langle w_{[k,n]} \rangle$, we have
\[
(k-j-1)\, 2 \log 2
\leqslant \psi^c(x) \leqslant 
(k-j +2)\, 2 \log 2
\]
and
\[
|\psi^c(x) - \psi^c(y)|
\leqslant  2^{j-n + 2}.
\]
The lower bound on $\psi^c(x)$ also holds for $j=n$.
\end{lemma}

\begin{proof}
By assumption, we have $w_{[k,j]} = x_{[1,j-k+1]}$ and  $w_{[k,j+1]} =  x_{[1,j-k+2]}$. Hence, $k \in \mc H_j(w)$ (together with the properties of $\mc G$) implies that 
\[
2^{-(j-k+2)} \leqslant \rho(x,\breve{c}) \leqslant 2^{-(j-k)},
\]
and likewise for $y$, with the upper bound holding also for $j=n$. Using furthermore that $\rho(x,y) \leqslant 2^{-(n-k+1)}$, all estimates follow readily from the bounds in Lemma~\ref{LEM:psi-distance-bound}.
\end{proof}
It is remarkable that the variation bound in the preceding lemma does \emph{not} depend on the choice of the $j$-hitting time $k$.

For every word $w$ let $\kappa(w) = \# \mc H(w)$ be the cardinality of the set of (full) hitting times in $w$,
and let $\kappa_n = \max_{w \in \Sigma^n} \kappa(w)$. It is straightforward to verify that $\kappa_n$ is increasing in $n$. Indeed, if $v \in \Sigma^m$ is a suffix of $w \in \Sigma^n$, we obtain that $\kappa(v) \leqslant \kappa(w)$. Since for $m\leqslant n$, every word $v \in \Sigma^m$ is a suffix of some $w(v) \in \Sigma^n$, we obtain
\[
\kappa_m = \max_{v \in \Sigma^m} \kappa(v) 
\leqslant \max_{v \in \Sigma^m} \kappa(w(v))
\leqslant \max_{w \in \Sigma^n} \kappa(w) = \kappa_n.
\]

\begin{lemma}
\label{LEM:fn-bound}
If the pressure function $\mc P_{\operatorname{\operatorname{top}}}(t \psi^c)$ is finite for some $t<0$, it follows that
\[
\limsup_{n \to \infty} \frac{\kappa_n^2}{n} < \infty.
\]
In particular, the sequence $(\kappa_n/\sqrt{n})_{n \in \N}$ is bounded in this case.
\end{lemma}

\begin{proof}
For a proof via contradiction, assume $\limsup_{n \to \infty} \kappa_n^2/n = \infty$ and $t < 0$. We estimate $\psi^c_n$ on $\langle w \rangle$ with $w \in \Sigma^n$. First, note that
\[
\sup_{x \in \langle w \rangle}  \psi^c_n(x) 
\leqslant \sum_{k=1}^n \sup_{x \in  \langle w_{[k,n]}\rangle} \psi^c(x)
\leqslant \sum_{k \in \mc H(w)} \sup_{x \in  \langle w_{[k,n]}\rangle} \psi^c(x).
\]
By Lemma~\ref{LEM:psi-on-hitting}, we have
$\psi^c(x) \leqslant (k-n+2) \log 4$ for $x\in\langle w_{[k,n]}\rangle$ and hence \[
\sup_{x \in \langle w \rangle} \psi^c_n(x) 
\leqslant  \sum_{k \in \mc H(w)} (k-n+2)\log 4
\leqslant  \sum_{j =1}^{\kappa(w)} (3-j)\log 4.
\]
For each $n \in \N$ let $w^n$ be a word such that $\kappa_n = \kappa(w^n)$, yielding
\[
\sup_{x \in \langle w^n\rangle} \psi^c_n(x) 
\leqslant  (6 \kappa_n - \kappa_n^2)\log 2 
\leqslant - \frac{\kappa_n^2}{2},
\]
provided $\kappa_n$ is large enough. For the pressure function, this implies
\begin{align*}
\mc P_{\operatorname{top}}(t \psi^c)
& = \limsup_{n\to \infty} \frac{1}{n} \log \left( \sum_{w \in \Sigma^n} \exp\biggl(t \sup_{x \in \langle w \rangle} \psi_n^{c}(x) \biggr) \right)
\geqslant \limsup_{n \to \infty} t \sup_{x \in \langle w^n\rangle} \psi_n^{c}(x)
\\ &\geqslant \limsup_{n \to \infty} (-t) \frac{\kappa_n^2}{2n} = \infty
\end{align*}
and we have reached the desired contradiction.
\end{proof}

\begin{lemma}
\label{LEM:inf-sup-relation}
Let $w \in \Sigma^n$ and $v \in \Sigma^m$ and assume that $w_{[k,n]} v \notin \mc G$ for all $1\leqslant k \leqslant n$. Then,
\[
\sup_{x \in \langle w \rangle} \psi^c_n(x) - \inf_{y \in \langle wv\rangle} \psi^c_n(y) \leqslant 2^{m+2} \kappa_{m+n}.
\]
\end{lemma}

\begin{proof}
By definition, we have
\begin{align*}
\sup_{x \in \langle w \rangle} \psi^c_n(x) - \inf_{y \in \langle wv\rangle} \psi^c_n(y)
& \leqslant \sum_{k=1}^n \sup_{x \in \langle w_{[k,n]}\rangle} \psi^c(x) - \inf_{y \in \langle w_{[k,n]} v\rangle} \psi^c(y)
\\& = \sum_{k=1}^n \bigl(\psi^c(x_k) - \psi^c(y_k) \bigr),
\end{align*}
with $x_k \in \langle w_{[k,n]}\rangle$ and $y_k \in \langle w_{[k,n]}v\rangle$ maximising the difference. In particular, we observe that
\begin{equation}
\label{EQ:xk-yk-difference}
\rho(x_k,y_k) \leqslant |w_{[k,n]}| = 2^{-(n-k+1)}.
\end{equation}
For notational convenience, let us set $u = wv$.
Due to the assumption that $w_{[k,n]}v \notin \mc G$ for all $1\leqslant k \leqslant n$, we obtain that $\mc H(u)\cap \{1,\ldots,n\} = \emptyset$, and we can therefore partition $\{1,\ldots,n\}$ into sets of the form
\[
A_j = \mc H_j(u) \cap \{1,\ldots,n\},
\]
with $1\leqslant j < n+m$. If $j < n$ and $k \in A_j$, we immediately get from Lemma~\ref{LEM:psi-on-hitting} the estimate $\psi^c(x_k) - \psi^c(y_k) \leqslant 2^{j-n+2}$.
If $j \geqslant n$ and $k \in A_j$, we use that $x_k$ maximises the value of $\psi^c$ on $\langle w_{[k,n]}\rangle$. By the monotonicity of $\psi^c$, this implies that $x_k$ maximises also the distance to $\breve{c}$. Since $\langle w_{[k,n]}\rangle$ has diameter $2^{-(n-k + 1)}$, we obtain
\[
\rho(x_k,\breve{c}) \geqslant 2^{-(n - k +2)} \geqslant 2^{-(j-k+2)}.
\]
Since $y_k \in \langle u_{[k,j + 1]}\rangle$ and $k \in A_j$, the prefix of $y_k$ of length $j - k+2$
is not in $\mc G$ and we also get $\rho(x_k,\breve{c}) \geqslant 2^{-(j-k+2)} $. Due to Lemma~\ref{LEM:psi-distance-bound}, this implies that both $x_k$ and $y_k$ lie in a region with Lipschitz constant  $2^{j - k + 3}$.
Hence, using \eqref{EQ:xk-yk-difference},
\[
\psi^c(x_k) - \psi^c(y_k)
\leqslant 2^{j - k + 3} \rho(x_k,y_k) \leqslant 2^{j - n + 2},
\]
which takes the same form as for $j<n$.
By the definition of $A_j$,
\[
\# A_j \leqslant \kappa(u_{[1,j]}) \leqslant \kappa_j \leqslant \kappa_{n+m}.
\]
Since every $k$ is contained in precisely one of the sets $A_j$, we get
\begin{align*}
\sup_{x \in \langle w \rangle} \psi^c_n(x) - \inf_{y \in \langle wv\rangle} \psi^c_n(y) 
&\leqslant \sum_{j=1}^{n+m-1} \sum_{k \in A_j} \psi^c(x_k) - \psi^c(y_k)
 \leqslant \sum_{j=1}^{n+m-1} (\# A_j)\, 2^{j-n+2} 
\\ &\leqslant \kappa_{n+m} 2^{2-n} \sum_{j=1}^{n+m-1} 2^j
\leqslant \kappa_{n+m} 2^{m+2},
\end{align*}
which is precisely the claimed bound.
\end{proof}

\begin{lemma}
\label{LEM:block-for-variation-bound}
Assume the binary expansion of $\breve{c}$ to be non-periodic and let $\varepsilon > 0$. Then, there are infinitely many $n\in \N$ such that for each $w \in \mc G_n$, there is a word $v$ with $|v| \leqslant 2 \log_2 \log_{3/2} n$ and $w_{[k,n]}v \notin \mc G$ for all $1\leq k \leq n$, and such that, given $x \in \langle w \rangle$ and $y \in \langle wv \rangle$,
\[
\psi^c_n(x) - \psi^c_n(y) \leqslant \varepsilon n.
\]
The same holds for periodic $\breve{c}$, if additionally $x_{[1,n+2]} \triangleleft_p \breve{c}$.
\end{lemma}

\begin{proof}
The general strategy of proof is as follows. Given $w \in \mc G_n$, we choose $s = s(w)$ and choose $v$ as in Lemma~\ref{LEM:closing-extension-II}.
Note that
\begin{equation}
\label{EQ:psi-variation-decompose}
\psi^c_n(x) - \psi^c_n(y) = \sum_{k=1}^n \psi^c(x_k) - \psi^c(y_k)
\end{equation}
for some $x_k \in \langle w_{[k,n]} \rangle$ and $y_k \in \langle w_{[k,n]} v \rangle$. For each $k$, we can find an estimate as follows. By the choice of $v$, we have $w_{[k,n]}v \notin \mc G$, implying that $\rho(y_k,\breve{c}) \geqslant 2^{-(n + |v| + 1 - k)}$. Furthermore, we certainly have $\psi^c(x_k) \leqslant \psi^c(x_k')$ where $x_k'$ is the point in $\langle w_{[k,n]} \rangle$ maximizing the distance to $\breve{c}$, in particular $\rho(x_k', \breve{c}) \geqslant \rho(y_k,\breve{c})$. Since $\rho(x_k,y_k) \leqslant 2^{-(n+1-k)}$, we get from the Lipschitz estimate in Lemma~\ref{LEM:psi-distance-bound},
\begin{equation}
\label{EQ:psi-variation-general}
\psi^c(x_k) - \psi^c(y_k)
\leqslant \psi^c(x_k') - \psi^c(y_k)
\leqslant 2^{|v|+1}.
\end{equation} 
However, using this bound for every $k$ will provide an estimate that is too rough for our purpose.
We are therefore concerned with finding finer estimates for $\psi^c(x_k) - \psi^c(y_k)$ based on the (symbolic) hitting depth of $k$.

If $\breve{c} \neq 0$ is eventually constant, it can be written in the form $\breve{c} = u10^\infty = u01^\infty$, for some word $u$ (possibly the empty word). As before, we choose $v$ as in Lemma~\ref{LEM:closing-extension-II}, implying in particular that $w_{[k,n]}v \notin \mc G$ for all $1\leqslant k \leqslant n$. Further, we use the identity in \eqref{EQ:psi-variation-decompose} and consider different cases for $k$. Fix some $n > |u|$. Then, $\mc G_n = \{ u01^r, u10^r \}$  with $r = n-|u|-1$. If $2 \leqslant k \leqslant n - |u| - 1$, we argue that the hitting depth of $k$ is at most $k+|u|$. Indeed, in this case $w_{[k,k+|u| + 1]}$ ends with $00$ or $11$, and it can therefore not be contained in $\mc G_{|u| + 2} = \{ u01, u10\}$. 
In particular, due to Lemma~\ref{LEM:psi-on-hitting},
\[
\psi^c(x_k) - \psi^c(y_k) \leqslant 2^{|u| + 2} 2^{k-n}.
\]
Evaluating a geometric sum for such $k$, and using the general bound in \eqref{EQ:psi-variation-general} for $k=1$ and $n-|u| \leqslant k \leqslant n$, we obtain
\[
\sum_{k=1}^n \psi^c(x_k) - \psi^c(y_k)
\leqslant (|u| + 2) 2^{|v| + 1} + 4,
\]
and, due to the estimate $|v| \leqslant 2 \log_2 \log_{3/2} n$, this expression is certainly smaller than $\varepsilon n$ for large enough $n$.

Assume now that $\breve{c}$ is non-periodic and not eventually constant.
We fix four large integer constants $m, p \ll n_p \ll n'$ in a hierarchical order, with precise relations to be determined later. Since $\breve{c}$ is not eventually constant, there are infinitely many choices of words $w_1$ such that $w_1 01$ is a prefix of $\breve{c}$. We choose one such $w_1$ with $|w_1| \geqslant n'$. Let $w_2'$ be the unique word of length $m$ such that $w' = w_1 01 w_2'$ is a prefix of $\breve{c}$ and set $n = |w'|$. 
Note that every $w \in \mc G_n$ is of the form $w = w_1 01 w_2''$, $w = w_1 10 0^m$, or $w = w_1 00 1^m$. In any case, $w = w_1 u w_2$,
with $u \in \{01,10\}$ and $|w_2| \in \{m-1,m\}$, redefining $w_1 0$ as $w_1$ in the last case.
For such $w$, we choose $s=s(w)$ and $v$ as in Lemma~\ref{LEM:closing-extension-II} and use the splitting in \eqref{EQ:psi-variation-decompose}. In particular $w_{[k,n]}v \notin \mc G$ for all $1 \leqslant k \leqslant n$.
If $k \in \widetilde{H}_{j}(w)$ for some $j<|w_1|$ 
we obtain from Lemma~\ref{LEM:hitting-depths-relations} that the hitting depth $j'$ of $k$ is no more than $|w_1| + 1$,
because $w_1$ is followed by $u \in \{01,10\}$.
Hence, recalling the splitting in \eqref{EQ:psi-variation-decompose} and using that $x_{k},y_k$ are in $\langle w_{[k,n]}\rangle$, we obtain via Lemma~\ref{LEM:psi-on-hitting},
 \[
\psi^c(x_k) - \psi^c(y_k) \leqslant 2^{j'-n+2} \leqslant 2^{-m+2} \leqslant \varepsilon/4,
\]
for large enough $m$. In particular,
\[
\sum_{j < |w_1|} \sum_{k \in \widetilde{\mc H}_j(w)} \psi^c(x_k) - \psi^c(y_k) \leqslant n \varepsilon/4.
\]
Given $p$, let $n_p$ be large enough to ensure that $\breve{c}_{[1,n_p]}$ is not $q$-periodic for any $1\leqslant q \leqslant p$. 
Assume $k_i \in \widetilde{\mc H}_{j_i}(w)$, with $|w_1| \leqslant j_i < |w|$ for $i =1,2$.
If $k_1 < k_2 \leqslant |w_1|- n_p$, this requires that $w_{[k_1,|w_1|]}$ and $w_{[k_2,|w_1|]}$ are both prefixes of $\breve{c}$, and hence $w_{[k_1,|w_1|]}$ is a rational power of $w_{[k_1,k_2-1]}$, with period length $k_2 - k_1$; compare Lemma~\ref{LEM:word-power}. 
Since $w_{[k_1,|w_1|]}$ is a prefix of $\breve{c}$ of length exceeding $n_p$, this is only possible if $k_2 - k_1 > p$. Due to this minimal distance, there are at most $n/p + n_p + m  + 2$ positions $k$ of symbolic hitting depth $j > |w_1|$. Indeed, there are at most $n_p + m + 2$ choices for $|w_1| - n_p < k \leqslant n$ and for $k \leqslant |w_1| - n_p$ there are no more than $n/p$ elements due to the separation of $p$ established above.

Note that for every $\delta'>0$ we can choose first $p$ and then $n' > n_p,m$ large enough that $n>n'$ guarantees
\[
n/p + n_p + m + 2\leqslant \delta' n.
\]
For $k \in \widetilde{\mc H}_j(w)$ and $j < |w|$, we get $\psi^c(x_k) - \psi^c(y_k) \leqslant 2$, again via Lemma~\ref{LEM:psi-on-hitting}. Hence,
\[
\sum_{|w_1| \leqslant j < |w|} \sum_{k \in \widetilde{\mc H}_j(w)} \psi^c(x_k) - \psi^c(y_k) \leqslant 2 (n/p + n_p + m + 2) \leqslant n \varepsilon/4,
\]
for large enough $p$ and $n'$. 
Finally, we deal with $k \in \widetilde{\mc H}(w)$, meaning $w_{[k,n]} \triangleleft_p \breve{c}$. By the same argument as above, we obtain that $\# \widetilde{\mc H}(w) \leqslant n/p + n_p + m + 2 \leqslant \delta_0 n$, for arbitrarily small $\delta_0>0$, to be chosen later. 
Using the general bound in \eqref{EQ:psi-variation-general}, we obtain
\[
\sum_{k \in \widetilde{\mc H}(w)} \psi^c(x_k) - \psi^c(y_k) \leqslant 2^{|v| +1} \# \widetilde{\mc H}(w).
\]
Recall from Lemma~\ref{LEM:closing-extension-II} that $2^{|v|} \leqslant \max\left\{2^C, s^2\right\}$, where $s$ is the number of independent hitting times in $\widetilde{\mc H}(w)$.
If $s^2 \leqslant 2^C$, we can choose $\delta_0$ such that $\delta_0 2^C \leqslant \varepsilon /3$ and the proof if complete. If $s^2 > 2^C$, let $\delta$ be such that $\delta n = \# \widetilde{\mc H}(w) - s < \delta_0 n$. Then,
\[
2^{|v|} \# \widetilde{\mc H}(w)
\leqslant s^2 (s+ \delta n) = s^3 + s^2 \delta n.
\]
Note that again by Lemma~\ref{LEM:closing-extension-II} $s^3 \leqslant (2 \log_2 \log_{3/2} n)^3$ which is smaller than $n \varepsilon/8$ for large enough $n$. If $\delta \leqslant 0$, we bound the second term with $0$. 
If $\delta>0$, Lemma~\ref{LEM:s-hitting-relation} yields $s \leqslant 2 - \log_{3/2}(\delta/2)$ 
and therefore
\[
s^2 \delta n \leqslant  (2- \log_{3/2}(\delta/2) )^2 \delta n \leqslant n \varepsilon/8,
\]
by choosing $\delta_0$ small enough and recalling $\delta < \delta_0$. In any case,
\[
\sum_{k \in \widetilde{\mc H}(w)} \psi^c(x_k) - \psi^c(y_k) \leqslant n \varepsilon/2,
\]
and the proof is complete in case that $\breve{c}$ is not periodic.

In the last step, we consider the case that the binary representation of $\breve{c}$ is periodic, of the form $\breve{c} = u^\infty$ for some $u \in \Sigma^p$ and $p \in \N$.
Without loss of generality, the length of $u$ is minimal with this property. Consider $n = mp$ for some $m \in \N$ and note that, by assumption, $x \in \langle w v' \rangle$ with $v' \in \Sigma^2$ and $wv'$ a prefix of $\breve{c}$, hence a power of $u$. 
In particular, $w=u^m$ and $v'$ is a rational power of $u$. We prove that the claimed estimate holds for some $v$ that starts with some fixed, arbitrary word in $\Sigma^2 \setminus \mc G_2$. 
Note that 
\[
\psi^c_n(x) - \psi^c_n(y) = \sum_{k=1}^n \psi^c(x_k) - \psi^c(y_k)
\]
for some $x_k \in \langle w_{[k,n]} v' \rangle$ and $y_k \in \langle w_{[k,n]} v \rangle$.
For $k$ of the form $k = rp+1$ with $0\leqslant r < m$, we have that $w_{[k,n]} v' = u^q$ (for some $q \in \Q$) is a prefix of $\breve{c}$. Hence, for such $k$, we have with $\ell_k = |w_{[k,n]}| + 2$ that $\rho(x_k,\breve{c}) \leqslant 2^{-\ell_k}$.
If $\breve{c} = 0^{\infty} = 1^{\infty}$, assuming $v_{[1,2]} \notin \{00,11\}$ implies that $w_{[k,n]} v_{[1,2]} \notin \mc G$.
If $\breve{c} \neq 0$, we choose $v_{[1,2]}$  not to be a neighbor of $v'$, and hence $w_{[k,n]} v_{[1,2]}$ is not a neighbor of $w_{[k,n]}v'$, again implying $w_{[k,n]}v_{[1,2]} \notin \mc G$. Thus, $\rho(y_k,\breve{c}) \geqslant 2^{-\ell_k} \geqslant \rho(x_k,\breve{c})$, and we conclude that
\begin{equation}
\label{EQ:psi-variation-0}
\psi^c(x_k) - \psi^c(y_k) \leqslant 0
\end{equation}
holds for all $k=rp + 1$ with $0 \leqslant r < m$.
If $p=1$, this holds for all possible $k$ and we are done. Otherwise, we proceed with $k$ of the form $k = rp+ t$ with $0 \leqslant r < m-2$ and $2\leqslant t \leqslant p$. 
In this case, the minimality of the period $p$ implies that $w_{[k,n]}$ does not start with $u$ and therefore $w_{[k,k+p-1]}$ is not a prefix of $\breve{c}$. By Lemma~\ref{LEM:periodic-hitting-depth}, this means that the hitting depth in $wv$ can be at most $k+2p-2$. 
This means that by Lemma~\ref{LEM:psi-on-hitting} we obtain
\[
\psi^c(x_k) - \psi^c(y_k) \leqslant 2^{k+2p -n},
\]
and combining this with \eqref{EQ:psi-variation-0}
gives
\[
\sum_{k=1}^{(m-2)p - 1} \psi^c(x_k) - \psi^c(y_k)
\leqslant \sum_{k=1}^{mp - 2p - 1} 2^{k+2p -mp} \leqslant 1.
\]
Finally, we consider $(m-2)p \leqslant k \leqslant mp$. Applying Lemma~\ref{LEM:closing-extension-II} to the word $w_{[(m-2)p,n]} v_{[1,2]}$ we can choose $v$ in such a manner that $w_{[k,n]} v \notin \mc G$ for all such $k$ and $|v|$ is bounded by $\ell_0 = \max \{ C ,2\log_2 \log_{3/2} (2p + 3) \} + 2$. 
For all such $k$, we use the general bound in \eqref{EQ:psi-variation-general} to obtain
\[
\sum_{k = (m-2)p}^{mp} \psi^c(x_k) - \psi^c(y_k)
\leqslant (2p + 1) 2^{\ell_0 + 1}.
\]
The sum of all contributions can be therefore bounded in terms of an expression in $p$ which is smaller than $\varepsilon n$ for large enough $n$.
With $n$ large enough, the bound $\ell_0$ for $|v|$ is smaller than $2 \log_2 \log_{3/2} n$ and by the properties established above, we see that $w_{[k,n]}v \notin \mc G$ for all $1\leqslant k \leqslant n$.
\end{proof}

\begin{prop}
\label{PROP:word-modification-control}
Let $\breve{c} \neq 1/2$ have non-periodic binary expansion and let $\varepsilon > 0$. Then there are infinitely many $n \in \N$ and a constant $\ell \in \N$ with the following property. For each $w \in \mc G_n$, there is a length $r \in \N$ with $r \leqslant \varepsilon n$, such that for all $v\in \Sigma^\ell$, there is a word $u = u(v,w) \in \Sigma^r$ such that the following holds. If $y \in \langle wuv \rangle$ and $x \in \langle wu'v \rangle$ for some $u' \in \Sigma^r$, we have
\[
\psi_{n+r}(x) - \psi_{n+r}(y) \leqslant \varepsilon n\quad \text{ and } \quad (wuv)_{[k,|wuv|]} \notin \mc G \text{ for all } 1\leqslant k \leqslant |wu|.
\]
The same conclusion holds if $\breve{c}$ is periodic and $ wu'_{[1,2]}$ is a prefix of $\breve{c}$.
\end{prop}

\begin{proof}
We construct $r$ and $u$ in several steps. First, if $\breve{c}$ is non-periodic, and if we choose $n\in \N$ as in Lemma~\ref{LEM:block-for-variation-bound}, we can find $r_1 \leqslant 2\log_2 \log_{3/2} n$ and a word $u_1 \in \Sigma^{r_1}$ such that
\begin{equation}
\label{EQ:big-block-variation}
\psi^c_n(x) - \psi^c_n(y) \leqslant n \varepsilon /2,
\end{equation}
for $x \in \langle w \rangle$ and $y \in \langle wu_1 \rangle$, as well as $w_{[k,n]}u_1 \notin \mc G$ for all $1 \leqslant k \leqslant n$.
The same holds for $x \in \langle wu'_{[1,2]} \rangle$ if $\breve{c}$ is periodic and $wu'_{[1,2]}$ is a prefix of $\breve{c}$. 

Due to Lemma~\ref{LEM:sliding-out-block-I}, we can further choose for $u_1$ a word $u_2$ of length $r_2 < 2\log_2 \log_{3/2} n$
(provided $n$ is large enough) and $ab \in \Sigma^2$ such that $(u_1 u_2)_{[k,|u_1 u_2|]}ab$ is not a prefix of $\breve{c}$ for all $1\leqslant k \leqslant |u_1 u_2|$. 
If $\widetilde{u} \in \{01,10\}$, this implies
\begin{equation}
\label{EQ:u-1-u-2-stop}
(u_1 u_2)_{[k,|u_1 u_2|]}ab \widetilde{u} \notin \mc G
\end{equation} 
for all $1 \leqslant k \leqslant |u_1 u_2|$; compare Lemma~\ref{LEM:hitting-depths-relations}.
Finally, let $\ell, \ell'$ be as in Lemma~\ref{LEM:sliding-out-block-II} and set $r = r_1 + r_2 + \ell' + 4$. Given $v \in \Sigma^{\ell}$, we choose words $\widetilde{u} \in \{01,10\}$ and $v' \in \Sigma^{\ell'}$, as provided by Lemma~\ref{LEM:sliding-out-block-II}, such that $u_3 = ab \widetilde{u} v'$ satisfies 
\begin{equation}
\label{EQ:u-3-stop}
(u_3)_{[k,|u_3| ]} v \notin \mc G
\end{equation}
for all $1\leqslant k \leqslant |u_3|$. For $n$ sufficiently large, we may choose $u = u_1 u_2 u_3$, satisfying $|u| = r$. 
 Since $u_3$ starts with $ab \widetilde{u}$ for some $\widetilde{u} \in \{01,10\}$, we obtain from \eqref{EQ:u-1-u-2-stop} and \eqref{EQ:u-3-stop} that no $k$ with $1\leqslant k \leqslant r$ has full hitting depth in $uv$. Since we already observed that $w_{[k,n]} u_1 \notin \mc G$ for all $1\leqslant k \leqslant n$, this also implies that
\[
(wuv)_{[k,|wuv|]} \notin \mc G,
\]
for all $1 \leqslant k \leqslant |wu|$.
If the hitting depth of $1\leqslant k \leqslant r$ in $uv$ is $j < |uv|$, we obtain for $y' \in \langle (uv)_{[k,|uv| ]}\rangle$ with the help of Lemma~\ref{LEM:psi-on-hitting} that
\[
\psi^c(y') \geqslant (k-j-1) 2\log 2 \geqslant - |uv| 2 \log 2.
\]
Recall that $|uv| \leqslant 4 \log_2 \log_{3/2} n + \ell' + \ell + 4$
for large enough $n$. We therefore get for $y'\in \langle uv \rangle$ that
\begin{equation}
\label{EQ:small-block-bound}
- \psi^c_r(y') \leqslant |uv|^2 2 \log 2 \leqslant n \varepsilon/2,
\end{equation}
provided that $n$ is large enough. 
Finally, note that for $x \in \langle wu'v \rangle$ and $y \in \langle wuv \rangle$, and $r=|u| = |u'|$ we can write
\[
\psi^c_{n+r}(x) - \psi^c_{n+r}(y)
= \psi^c_n(x) - \psi^c_n(y) + \psi^c_r(x') - \psi^c_r(y'),
\]
for some $x' \in \langle u'v \rangle$ and $y' \in \langle uv \rangle$. Using the trivial bound $\psi^c_r(x') \leqslant 0$, together with \eqref{EQ:big-block-variation} and \eqref{EQ:small-block-bound}, we get the desired estimate $\psi_{n+r}(x) - \psi_{n+r}(y) \leqslant \varepsilon n$.
\end{proof}

\subsection{Proof of Proposition \ref{PROP:word-regularisation}}\label{subsec: proof word regularisation}

\begin{proof}[Proof of Proposition \ref{PROP:word-regularisation}]
The basic idea of this proof is the following: Whenever an element of $\Sigma^n$ has  a subword  that lies in $\mc G_{\iota(m)}$, where $\iota$ is a function to be defined later with $\iota(m)\sim m$,  the letters following the subword are replaced with a word as in  Proposition ~\ref{PROP:word-modification-control} for a given $\varepsilon>0$. This changes the Birkhoff sum by no more than $\varepsilon m$.
 Summing over all possible replacements gives at a correction  to the Birkhoff sum of at most $\varepsilon n$. Each of the replacing words has length at most $\varepsilon m$, so the total amount of replaced letters is no more than $\varepsilon n$. This bounds the number of fibers to be at most $2^{\varepsilon n}< \me^{\varepsilon n}$. The details follow.

Let $\varepsilon > 0$ and $\breve{c} \neq 1/2$. For infinitely many numbers $m'$, Proposition~\ref{PROP:word-modification-control} guarantees the existence of a constant $\ell \in \N$ such that for all $w' \in \mc G_{m'}$ we can find a length $r = r(w')$ with $r \leqslant \varepsilon m'$
such that for all $v \in \Sigma^\ell$, we can find $u = u(v,w') \in \Sigma^r$ with the property that $(w'uv)_{[k,|w'uv|]} \notin \mc G$ for all $1 \leqslant k \leqslant |w'u|$ and
\begin{equation}
\label{EQ:variation_estimate_after_modification}
\psi_{m'+r}(x) - \psi_{m'+r}(y) \leqslant \varepsilon m',
\end{equation}
whenever $y \in \langle w' u v\rangle$ and $x \in \langle w' u' v \rangle$ for some $u' \in \Sigma^r$ (assuming that $w'u'_{[1,2]} \in \mc G$ if $\breve{c}$ is periodic). 
For such a number $m'$, let us set $r_{\max} = \max_{w' \in \mc G_{m'}} r(w')$ and $m = m' + r_{\max} + \ell$.

We define $\theta_m$ on $\Sigma^n$ via an appropriate splitting of $w \in \Sigma^n$. 
For a moment, let us assume that $\breve{c}$ is non-periodic.
First, let $w = p_1 w_1 s_1$ be such that $w_1$ is the first occurrence of a word from $\mc G_{m'}$ in $w$ and such that $|w_1 s_1| \geqslant m$. If no such splitting exists, we note that $w \in \Sigma^n_m$ and we simply set $\theta_m(w) = w$. With $r_1 = r(w_1)$ as above, we let $u_1'$ be the word in $\Sigma^{r_1}$ such that $w$ starts with $p_1 w_1 u_1'$. Next, we define $p_2 w_2$ such that $w_2$ is the first occurrence of a word from $\mc G_{m'}$ in $w = p_1 w_1 u_1' p_2 w_2 s_2$ that appears after $u_1'$ and satisfies $|w_2 s_2| \geqslant m$. Iterating this construction, we obtain a splitting
\[
w = p_1 w_1 u_1' p_2 w_2 u_2' \cdots p_J w_J u_J' s,
\]
with the properties that $w_q \in \mc G_{m'}$, any word of length $m'$ starting in $p_q$ is not in $\mc G_{m'}$, as well as $|u'_q| = r(w_q)$, for all $1 \leqslant q \leqslant J$, and such that $|s| \geqslant \ell$ and $J$ is maximal with this property. Further, for each $1\leqslant q \leqslant J$ let $v_q$ be the word of length $\ell$ that follows $u_q'$ in the splitting above. 
For all $q$, we choose $u_q = u(v_q,w_q)$ with properties as in Proposition~\ref{PROP:word-modification-control} and define $\theta_m(w)$ by replacing each $u_q'$ by $u_q$ in the splitting above, that is,
\[
\theta_m(w) = p_1 w_1 u_1 p_2 w_2 u_2 \cdots p_J w_J u_J s.
\]
We inherit from the structure of $w$ that every word of length $m'$ starting in $p_q$ in the above splitting of $\theta_m(w)$ is not contained in $\mc G$. Additionally, if a word of length $m$ starts in some $w_q$ or $u_q$ it certainly also completely contains $v_q$ because $|w_q u_q v_q| \leqslant m$, and therefore it is not in $\mc G$, due to the defining properties of $u_q$. Finally, there is no word from $\mc G_m$ in $s$ due to the maximality of $J$. It follows that $\theta_m(w)$ contains no prefix of $\breve{c}$ of length $m$, proving the first claim.

If $\breve{c}$ is periodic, we define $\theta_m$ in the same manner up to a slight modification. Instead of $w_1$ being the first instance of a word in $\mc G_{m'}$, we define it as the prefix of length $m'$ of the first word in $\mc G_{m'+2}$ in $w$ (such that $|w_1 s_1| \geqslant m$). Corresponding modifications are applied to the definition of $w_i$ for all $1\leqslant i \leqslant J$. Formally, we obtain the same kind of splitting of $w$ as before and again we define $\theta_m$ as the map replacing $u_i'$ by the word $u_i$ from Proposition~\ref{PROP:word-modification-control} for all $1\leqslant i \leqslant J$. The observation that $\theta_m(w)$ contains no word from $\mc G_m$ follows in a similar manner as before. The only point where we need to adapt the argument concerns words of length $m'+2$ that start one letter before a word of type $w_i$. For large enough $m'$, this word and $w_i$ can only be simultaneously in $\mc G$ if $\breve{c} = 0$, the word $w_i$ is constantly equal to one letter $a$ and the letter preceding $w_i$ is also $a$. But this contradicts the minimality of the position in the definition of $w_i$. Hence, such words cannot be in $\mc G$ and we conclude as before.

In order to estimate the cardinality of preimages of a word $w' = \theta_m(w)$, we first identify the positions that might have been affected by the application of $\theta_m$. Let us call an index $i$ a \emph{suspect} in $w'$ if $w'_{[i+1,i+m']} \in \mc G$. We call a suspect $i$ \emph{essential} if there is no suspect $j$ with $i+r_{\max} + \ell \leqslant j \leqslant i + m'$. This notion is relevant because if $i+1$ is the starting position of some $w_q$ in the splitting of $w$, then $i$ is an essential suspect in $\theta_m(w)$.
Indeed, since $w_q$ is not affected by $\theta_m$, we obtain that $(\theta_m(w))_{[i+1,i+m']} = w_q \in \mc G$. On the other hand, if $i + r_{\max} + \ell \leqslant j \leqslant i + m'$, we conclude that $(\theta_m(w))_{[j+1,j+m']}$ starts with
$(w_q u_q v_q)_{[j-i+1,|w_q u_q v_q|]} \notin \mc G$, implying that $j$ is not a suspect.
Due to the repellent nature in the definition of essential suspects, they come in groups $(\mc E_q)_{q =1}^N$ of diameter at most $r_{\max} + \ell$, separated by gaps of size at least $m' - r_{\max} - \ell$. That is, 
the number of groups is bounded by 
\begin{equation}
\label{EQ:bound_on_essential_suspect_groups}
N \leqslant \lfloor n / (m' - r_{\max} - \ell) \rfloor + 1 \leqslant \lfloor n/((1 - 2\varepsilon)m') \rfloor,
\end{equation}
assuming that $m'$ is large enough to ensure $\ell < \varepsilon m'$ and $n$ sufficiently larger than $m'$. Note that $w$ and $\theta_m(w)$ only differ on intervals of the form 
\[
D_i = [i+m'+1,i+m'+r] \subset i + m' + [1,r_{\max}]
\] 
with $i+1$ the starting position of some $w_q$, implying that $i \in \mc E = \bigcup_{q=1}^N \mc E_q$ is an essential suspect. We conclude that all changes between $w$ and $w' = \theta_m(w)$ appear in
\[
D = \bigcup_{i \in \mc E} D_i = \bigcup_{q=1}^N \bigcup_{i \in \mc E_q} D_i.
\]
Due to the fact that the elements in $\mc E_q$ are no more than $r_{\max}+\ell$ apart we obtain that each $D^{(q)} = \bigcup_{i \in \mc E_q} D_i$ is contained in an interval of length at most $2 r_{\max} + \ell \leqslant 3 \varepsilon m'$ for large enough $m'$. Combining this with \eqref{EQ:bound_on_essential_suspect_groups}, we get that the cardinality of $D$ is bounded by
\[
\# D \leqslant 3 \varepsilon m' N \leqslant \frac{3 \varepsilon}{1-2\varepsilon}n. 
\]
Since $\varepsilon > 0$ was arbitrary, we can also achieve $\# D \leqslant \varepsilon n$ for a fixed given $\varepsilon > 0$ and since all words $w \in \theta_m^{-1}(w')$ can differ only on $D$, their total cardinality is bounded by
\[
\# \theta_m^{-1}(w') \leqslant 2^{\# D} \leqslant 2^{\varepsilon n},
\] 
proving the second claim of the proposition.

Finally, we aim to estimate the difference $\psi_n^c(x)$ and $\psi_n^c(y)$, where $x \in \langle w \rangle$, $y \in \langle \theta_m(w) \rangle \cap \mathbb{X}_{m}^c$ are chosen such that they maximise the corresponding expression. 
Recall the decompositions of $w$ and $\theta_m(w)$ given above.
For each $1\leqslant q \leqslant J$ let $j_q +1$ be the starting index of $p_q$ in these decompositions and set $x_q = S^{j_q}x$ as well as $y_q = S^{j_q}y$. Similarly, we set $x_q' = S^{i_q}x$ and $y_q' = S^{i_q}y$, where $i_q + 1$ is the starting index of $w_q$. And finally, let $j+1$ be the starting index of $s$ and set $\widetilde{x} = S^j x$ and $\widetilde{y} = S^j y$.
We can split up the difference
\[
\psi^c_n(x) - \psi^c_n(y)
= \psi^c_{|s|}(\widetilde{x}) - \psi^c_{|s|}(\widetilde{y}) + \sum_{q=1}^J \bigl( \psi^c_{|p_q|} (x_q) - \psi^c_{|p_q|}(y_q) \bigr)
+ \sum_{q=1}^J \bigl( \psi^c_{|w_q u_q|} (x_q') - \psi^c_{|w_q u_q|} (y_q') \bigr).
\]
First, we rewrite
\[
\psi^c_{|s|}( \widetilde{x}) - \psi^c_{|s|}( \widetilde{y}) \leqslant \sum_{k=1}^{|s|} \psi^c(\widetilde{x}_k) - \psi^c( \widetilde{y}_k),
\]
where both $ \widetilde{y}_k = S^{k-1}\widetilde{y}$ and $\widetilde{x}_k$ start with $s_{[k,|s|]}$ and $\widetilde{x}_k$ is chosen to maximise the distance to $\breve{c}$ with this property.
Note that for $k$ with $1\leqslant k \leqslant |s| - m$, the length of the words $s_{[k,|s|]}$ is at least $m+1$ and hence $s_{[k,k+m]} \notin \mc G$. In particular, $k \in \mc H_{j}$ for some $j < k+m$ and we obtain from Lemma~\ref{LEM:psi-on-hitting} that
\[
\psi^c(\widetilde{x}_k) - \psi^c( \widetilde{y}_k) \leqslant 2^{k+m +1- |s|}.
\]
For $|s|-m < k \leqslant |s|$, it follows from $\widetilde{y}_k \in \mathbb{X}^c_m$ and the maximising property of $\widetilde{x}_k$ that both points have a distance of at least $2^{-m-1}$ to the singularity and at the same time they share a common prefix of length $|s| - k + 1$ implying a bound on their distance given by
\[
\rho(\widetilde{x}_k,\widetilde{y}_k) \leqslant 2^{k - |s| - 1}.
\]
Using Lemma~\ref{LEM:psi-distance-bound} once again,
\[
\psi^c(\widetilde{x}_k) - \psi^c(\widetilde{y}_k) \leqslant 2^{k+m+1- |s|},
\]
so this bound holds for all $1\leqslant k \leqslant |s|$ and we get 
\[
\psi^c_{|s|}( \widetilde{x}) - \psi^c_{|s|}(\widetilde{y}) \leqslant 2^{m+2}.
\]
For the next estimate, note that $x_q$ and $y_q$ both have $p_q w_q$ as a common prefix and we estimate
\[
\psi^c_{|p_q|} (x_q) - \psi^c_{|p_q|}(y_q)
= \sum_{k=1}^{|p_q|} \psi^c(S^{k-1} x_q) - \psi^c(S^{k-1} y_q)
\leqslant \sum_{k=1}^{|p_q|} \psi^c(x_{q,k}) - \psi^c(y_{q,k}),
\] 
where both $y_{q,k} = S^{k-1} y_q$ and $x_{q,k}=S^{k-1} x_q$ start with the word $(p_q w_q)_{[k,|p_q w_q|]}$. By the defining property of $p_q$, the prefix of length $m'+2$ of this word is not in $\mc G$ (length $m'+1$ if $k=|p_1|$), and hence the hitting depth is at most $k + m'$ (respectively $k+m'-1$ if $k = |p_1|$). Hence, by Lemma~\ref{LEM:psi-on-hitting} and recalling $|w_q| = m'$,
\[
\psi^c(x_{q,k}) - \psi^c(y_{q,k}) \leqslant 2^{k - |p_q|  + 2 },
\]
and summing up over all $1 \leqslant k \leqslant |p_q|$ yields
\[
\psi^c_{|p_q|} (x_q) - \psi^c_{|p_q|}(y_q) \leqslant 8.
\]
Finally, we recall that due to the choice of $u_q$ we have for all $1\leqslant q \leqslant J$ that
\[
\psi^c_{|w_q u_q|} (x_q') - \psi^c_{|w_q u_q|} (y_q') \leqslant \varepsilon m',
\]
due to the fact that $x_q'$ starts with $w_q u_q' v_q$ and $y_q'$ starts with $w_q u_q v_q$ (and $w_q (u'_q)_{[1,2]} \in \mc G$ if $\breve{c}$ is periodic), compare \eqref{EQ:variation_estimate_after_modification}. 
Combining all the previous estimates, we obtain,
\[
\psi^c_n(x) - \psi^c_n(y)
\leqslant 2^{m+2}  {+ 8J} + J \varepsilon m' \leqslant  2^{m+2} + 2\varepsilon n,
\]
for sufficiently large $m'$,
using $J \leqslant \lfloor n / m' \rfloor$ in the last step. Provided that $n$ is sufficiently larger than $m$,
this expression can be bound by $3 \varepsilon n$. Since $\varepsilon$ was arbitrary, this yields the claim.
\end{proof}

\subsection{Proofs of the Gibbs type properties}\label{subsec: proofs gibbs type}

\begin{proof}[Proof of Proposition \ref{PROP:measure-decay}]
The assumption $c \neq 0$ or $\breve{c} \neq 1/2$ is clearly necessary since $\mu_{0} = \delta_0$ for $c = 0$. On the other hand, it follows from Corollary \ref{coro: unique g} that $\mu_{c}$ has full support if $c \neq 0$.
Let $w \in \Sigma^n$ with $n \geqslant n_0$ and $v$ as in Lemma~\ref{LEM:closing-extension-II}. Due to Corollary~\ref{COR:mu-on-double-word}, we have
\[
\log \mu_{c}(\langle w \rangle) \geqslant \log \mu_{c}(\langle wv\rangle)
\geqslant \log \mu_{c}(\langle v\rangle) + \inf_{x \in \langle wv\rangle} \psi^c_n(x).
\]
We can estimate
\[
\inf_{x \in \langle wv\rangle} \psi^c_n(x)
\geqslant \sum_{k=1}^n \inf_{x \in \langle w_{[k,n]}v\rangle} \psi^c(x). 
\]
Let $m_n = |wv| = n+|v|$.
By the choice of $v$, we have that $w_{[k,n]}v \notin \mc G$ and hence for $x \in \langle w_{[k,n]}v\rangle$,
\[
\rho(x,\breve{c}) \geqslant 2^{-(m_n -k+1)},
\]
which, by Lemma~\ref{LEM:psi-distance-bound} implies
\[
\psi^c(x) \geqslant 2 \log( 2^{k-m_n}) = (k-m_n) 2\log 2.
\]
Hence,
\[
\inf_{x \in \langle wv\rangle} \psi^c_n(x)
\geqslant \sum_{k=1}^n (k-m_n) 2\log 2
= n \log 2(n+1-2m_n)
\geqslant - (n^2+2n|v|) \log 2.
\]
Let $\varepsilon >0$ and $|v| \leqslant 2 \log_2 \log_{3/2} n =: h(n)$. For large enough $n$ we get
\[
 n^2+2n |v| \leqslant n^2 + 4 n \log_2 \log_{3/2}n
\leqslant (1+\varepsilon) n^2
\]
and therefore
\[
\inf_{x \in \langle wv\rangle} \psi^c_n(x)
\geqslant - (1+\varepsilon)n^2 \log 2.
\]
We thereby get the recursion
\[
\log \mu_{c}(\langle w \rangle) \geqslant - (1+ \varepsilon)n^2 \log 2 + \log \mu_{c}(\langle v\rangle).
\]
If $|v| > n_0$, we can iterate the relation and get
\[
\log \mu_{c}(\langle w \rangle) \geqslant - (1+\varepsilon)n^2 - (1+\varepsilon) (h(n))^2 + \log \mu_{c}(\langle v'\rangle),
\]
for some $v'$ with $|v'| \leqslant h^2(n) = h( h(n))$. The iteration stops after $r_n$ steps, where $r_n$ is the minimal integer such that $h^{r_n}(n) \leqslant n_0$. Since the iterates of $h$ are decaying very fast, it is straightforward to see that $r_n < n$. We obtain
\[
\inf_{w \in \Sigma^n} \log \mu_{c}(\langle w \rangle)
 \geqslant -(1+\varepsilon)(n^2 + n (h(n))^2 ) + \inf_{u \in \Sigma^{n_0}} \log \mu_{c}(\langle u\rangle).
\]
Note that the last term is finite because $\mu_{c}$ has full support.
Dividing by $n^2$, taking the liminf and sending $\varepsilon \to 0$ yields the assertion.
\end{proof}

\begin{proof}[Proof of Proposition \ref{PROP:Gibbs-like}]
The upper bound
\[
\log \mu_{c}(\langle w \rangle) \leqslant \sup_{x \in \langle w \rangle} \psi^c_n(x) 
\]
holds in full generally.
For the lower bound, let $w \in \Sigma^n$, with $n \geqslant n_0$ and $v$ as in Lemma~\ref{LEM:closing-extension-II}. By Corollary~\ref{COR:mu-on-double-word} and Proposition~\ref{PROP:measure-decay}, we have
\begin{align*}
\log \mu_c(\langle w \rangle) & 
\geqslant \log \mu_c(\langle v\rangle) + \inf_{x \in \langle wv\rangle} \psi^c_n(x)
\geqslant - 2 |v|^2 + \inf_{x \in \langle wv\rangle}  \psi^c_n(x)
\\ & \geqslant - 8(\log_2 \log_{3/2} n)^2 + \inf_{x \in \langle wv\rangle} \psi^c_n(x),
\end{align*}
provided $n$ is large enough. Due to Lemma~\ref{LEM:inf-sup-relation}, we have the estimate
\[
\inf_{x \in \langle wv\rangle} \psi^c_n(x)
\geqslant \sup_{x \in \langle w \rangle} \psi^c_n(x) - 2^{|v| +2} \kappa_{n + |v|}
\geqslant \sup_{x \in \langle w \rangle} \psi^c_n(x) - 4 \kappa_{2n} (\log_{3/2} n)^2.
\]
Overall, this yields
\[
\log \mu_c(\langle w \rangle) \geqslant \sup_{x \in \langle w \rangle} \psi^c_n(x) - \iota_n
\]
with
\[
\iota_n = 8 (\log_2 \log_{3/2} n)^2 + 4 \kappa_{2n} (\log_{3/2} n)^2.
\]
Since $\mc P_{\operatorname{top}}(t \psi^c) < \infty$, we know from Lemma~\ref{LEM:fn-bound} that $\kappa_n \in O(\sqrt{n})$ and hence $\iota_n \in o(n)$. The claim follows.
\end{proof}

\begin{proof}[Proof of Lemma \ref{LEM:mu-on-Kdelta}]
We start with the upper bound. Given $n \in \N$ we have by Lemma~\ref{LEM:mu-on-interval} that
\[
\mu_c(I) \leqslant \exp \Big( \sup_{y \in I} \psi_n^c(y) \Big).
\]
We define a regularized version of $\psi^{c}$ via
\[
\psi^{c,\delta} (y)
= \begin{cases}
\psi^c(y), & \mbox{if } \rho(y,\breve{c}) \geqslant \delta,
\\ \psi^c(\breve{c} - \delta), & \mbox{if } \rho(y,\breve{c}) < \delta.
\end{cases}
\]
Since $\psi^c(\breve{c}-\delta) = \psi^c(\breve{c} + \delta)$ this function is continuous and, in fact, H\"{o}lder continuous. Also note that $\psi^c(y) \leqslant \psi^{c,\delta}(y)$ for all $y \in \TT$.
For $x \in I \cap K_\delta$ we have $\psi^c(T^k x) = \psi^{c,\delta}(T^k x)$ for all $k \in \N_0$. For every $y \in I$, we obtain
\[
\psi^c_n(y) - \psi^c_n(x) = \sum_{k=0}^{n-1} \psi^c(T^k y) - \psi^c(T^k x)
\leqslant \sum_{k=0}^{n-1} \psi^{c,\delta}(T^k y) - \psi^{c,\delta}(T^k x) \leqslant \log(R),
\]
for an appropriate number $R>1$. This yields the upper bound.

For the lower bound, let $n \in \N$ and fix some large number $N \in \N$ to be determined later. Decompose $I$ into $N$ intervals $I_i$ with $1\leqslant i \leqslant N$, each occupying a fraction $r = 1/N$ of the original length $2^{-n}$. By Lemma~\ref{LEM:mu-on-interval}, we have
\begin{equation}
\label{EQ:subinterval-bound}
\mu_c(I_i) \geqslant \mu_c(T^n I_i)\,  \exp \left( \inf_{y \in I_i} \psi_n^c(y) \right).
\end{equation}
We define a subset of indices with good variation control
\[
\mc V = \mc V_{c, \delta, r}=\left\{ 1 \leqslant i \leqslant N: \rho(\breve{c},T^k I_i) \geqslant r \delta, \, \forall 0\leqslant k \leqslant n-1 \right\},
\]
and its complement $\mc B = ([1,N] \cap \N) \setminus \mc V$. Recall $x \in I \cap K_{\delta}$. If $i \in \mc V$, we get for all $y \in I_i$ that both $T^k y$ and $T^k x$ are bounded away from $\breve{c}$ by at least $r \delta$ for all $0\leqslant k \leqslant n-1$. Hence, we can use the H\"{o}lder continuity of $\psi^c$ restricted to the complement of $B_{r\delta}(\breve{c})$ to conclude that 
\[
\psi_n^c(x) - \psi_n^c(y) = \sum_{k=0}^{n-1} \psi^c(T^k x) - \psi^c(T^k y) \leqslant \log(R')
\]
for some $R' = R'(c,\delta,r) > 1$.
Combining this with the estimate in \eqref{EQ:subinterval-bound}, we obtain for $i \in \mc V$,
\[
\mu_c(I_i) \geqslant \mu_c(T^n I_i)\, \frac{1}{R'}\, \exp\left( \psi_n^c(x) \right).
\] 
Hence,
\begin{equation}
\label{EQ:mu-interval-lower}
\mu_c(I) \geqslant \sum_{i \in \mc V} \mu_c(I_i) 
\geqslant \frac{1}{R'} \exp\left(\psi_n^c(x) \right)\, \sum_{i \in \mc V} \mu_c(T^n I_i).
\end{equation}
Note that the union over all $T^n I_i$ with $1\leqslant i \leqslant N$ is given by $T^n I = \TT$, implying
\begin{equation}
\label{EQ:bad-vs-good-indices}
\sum_{i \in \mc V} \mu_c(T^n I_i) = 1 - \sum_{i \in \mc B} \mu_c(T^n I_i).
\end{equation}
We argue that the cardinality of $\mc B$ cannot be too large. By definition $i \in \mc B$ requires that $i \in \mc B_k$ for some $0 \leqslant k \leqslant n-1$, where
\[
\mc B_k = \left\{ 1 \leqslant i \leqslant N: \rho(\breve{c} , T^k I_i) < r\delta \right\}.
\]
Let $y \in I_i$ for some $i \in \mc B_k$ and recall $x \in I \cap K_{\delta}$.
Since $\rho(x,y) \leqslant 2^{-n}$, we have, if $2^{k-n} \leqslant \delta/2$,
\[
\rho(T^k y, \breve{c}) \geqslant \rho(T^k x, \breve{c}) - \rho(T^k x, T^k y)
\geqslant \delta - 2^{k-n} \geqslant \delta/2 \geqslant r \delta.
\]
That is, $\mc B_k = \emptyset$ as long as $k \leqslant n-1 + \ceil{\log_2 \delta}$. 
It therefore suffices to consider $k = n-m$ with $m = 1, \ldots, -\ceil{\log_2 \delta}$. For such $k = n-m$, each interval $T^k I_i$ has length
\[
|T^k I_i| = 2^{n-m}  |I_i| = r 2^{-m} \geqslant r \delta,
\]
and therefore there can be at most two indices $i$ such that $\rho(\breve{c},T^k I_i) < r \delta$, implying that $\# \mc B_k \leqslant 2$ for such $k$. Overall, we obtain
\[
\# \mc B \leqslant \sum_{k=n+ \floor{\log_2 \delta}}^{n-1} \# \mc B_k 
\leqslant -2 \floor{\log_2 \delta} =: C.
\]
In particular, this bound is independent of $r$. Note that each interval $T^n I_i$ has length $r$. Since $\mu_c$ is a continuous measure for $c \neq 0$, we can find an $r = 1/N$ such that each interval $J$ of length $r$ is bounded in measure by $\mu_c(J) \leqslant 1/(2C)$. With this choice of $r = r(c,\delta)$, we obtain 
\[
\sum_{i \in \mc B} \mu_c(T^n I_i) \leqslant \frac{C}{2C} = \frac{1}{2}
\]
and, combining this with \eqref{EQ:mu-interval-lower} and \eqref{EQ:bad-vs-good-indices}, we get the required lower bound with $R = 2 R'$.
\end{proof}

\section*{Acknowledgements} 
~
It is our pleasure to thank Michael Baake and J\"org Schmeling for helpful discussions.

PG acknowledges
support from the German Research Foundation (DFG) through Project 509427705 and Project Z-72-00539-00-16100116. Furthermore, he thanks the Universities of Bielefeld, Bremen and Lund for their hospitality. 

The research of TS has been supported by the
European Union’s Horizon Europe
research and innovation programme under the Marie Sklodowska-Curie grant agreement ErgodicHyperbolic - 101151185 and by a grant from the Priority Research Area SciMat under the Strategic Programme Excellence Initiative at Jagiellonian University. 
Moreover, she was supported by the Austrian Science Fund FWF: P 33943-N
and acknowledges the support of UniCredit Bank R\&D group for financial support through the ``Dynamics and Information Theory Institute'' at
the Scuola Normale Superiore.
She furthermore thanks the Universities Bielefeld, Bremen and Lund for their hospitality. 

MK and TS would like to thank the Institut Mittag-Leffler for their hospitality.

\end{document}